\documentclass[reqno]{amsart}
\usepackage{xcolor}
\usepackage[colorlinks,citecolor=green!50!black,linkcolor=blue]{hyperref}
\usepackage[utf8]{inputenc}
\usepackage[T2A]{fontenc}
\usepackage[russian,english]{babel}
\usepackage{amsfonts,amsmath}	
\usepackage{bm} 
\usepackage{dsfont}
\usepackage{amssymb}
\usepackage{mathrsfs}
\usepackage{mathtools}
\mathtoolsset{showonlyrefs}
\usepackage{graphicx}
\usepackage{tikz,tikz-3dplot}
\usetikzlibrary{patterns}
\usepackage{pgfplots}
\pgfplotsset{compat=1.16}
\usetikzlibrary{math} 
\usetikzlibrary{intersections,calc}
\usetikzlibrary{decorations.markings}
\usepgfplotslibrary{colormaps}
\pgfplotsset{trig format plots=rad}
\usepackage{wrapfig}
\usepackage{subcaption}
\usepackage[normalem]{ulem}

\date{\today}

\let\div\relax
\DeclareMathOperator{\div}{div}
\DeclareMathOperator{\dive}{div}

\DeclareMathOperator{\dist}{dist}

\let\d\relax
\newcommand{\d}{\partial}
\newcommand{\cd}{\mathfrak{c} \partial}
\newcommand{\eps}{\varepsilon}
\newcommand{\lam}{\lambda}

\newcommand{\R}{\mathbb{R}}

\newcommand{\N}{\mathbb{N}}
\newcommand{\Z}{\mathbb{Z}}
\newcommand{\T}{\mathbb{T}}
\newcommand{\G}{\mathbf{G}}

\newcommand{\fhi}{\varphi}
\newcommand{\ve}{\boldsymbol{v}}
\newcommand{\1}{\mathds{1}}
\renewcommand{\L}{\mathscr{L}}
\newcommand{\Lip}{\operatorname{\mathrm{Lip}}}
\newcommand{\reg}{\operatorname{\mathsf{reg}}}
\newcommand{\meas}[1]{\L^d\left(#1\right)}

\renewcommand{\H}{\mathscr{H}}
\DeclareMathOperator{\rest}{\llcorner}
\DeclareMathOperator{\BV}{BV}
\DeclareMathOperator{\TV}{TV}
\DeclareMathOperator{\FV}{FV}
\DeclareMathOperator{\Tan}{{\mathsf {Tan}}}
\newcommand{\sat}{\operatorname{sat}}
\newcommand{\diam}{\operatorname{diam}}
\newcommand{\loc}{\text{\rm loc}}

\renewcommand{\liminf}{\mathop{\underline{\mathrm{lim}}}\limits}

\newcommand{\dd}{\mathrm{d}}
\def\Xint#1{\mathchoice
   {\XXint\displaystyle\textstyle{#1}}%
   {\XXint\textstyle\scriptstyle{#1}}%
   {\XXint\scriptstyle\scriptscriptstyle{#1}}%
   {\XXint\scriptscriptstyle\scriptscriptstyle{#1}}%
   \!\int}
\def\XXint#1#2#3{{\setbox0=\hbox{$#1{#2#3}{\int}$}
     \vcenter{\hbox{$#2#3$}}\kern-.5\wd0}}

\def\dashint{\Xint-}

\newtheorem{lemma}{Lemma}[section]
\newtheorem{prop}[lemma]{Proposition}
\newtheorem{theorem}[lemma]{Theorem}
\newtheorem{conj}[lemma]{Conjecture}
\newtheorem{cor}[lemma]{Corollary}
\newtheorem{defn}[lemma]{Definition}

\theoremstyle{definition}
\newtheorem{rem}[lemma]{Remark}

\title[]{The Nelson conjecture and chain rule property}

\author{Nikolay A. Gusev}
\address[N.G.]{Moscow Institute of Physics and Technology,
	9 Institutskiy per., Dolgoprudny, Moscow Region, Russia 141700
}
\email{ngusev@phystech.su}
\email{ngusev@phystech.edu}
\email{n.a.gusev@gmail.com}

\author{Mikhail V. Korobkov}
\address[M.K.]{School of Mathematical Sciences,
Fudan University,
Shanghai 200433,
People’s Republic of China}
\address[M.K.]{Sobolev Institute of Mathematics, Acad. Koptyug pr. 4, Novosibirsk, Russia.}
\email{korob@math.nsc.ru}

\begin{document}
\begin{abstract}
Let $p\ge 1$ and let $\ve \colon \mathbb R^d \to \mathbb R^d$ be a compactly supported vector field
with $\ve \in L^p(\mathbb R^d)$ and $\operatorname{div} \ve = 0$ (in the sense of distributions).
It was conjectured by Nelson that if $p=2$, then the operator $\mathsf{A}(\rho) := \ve \cdot \nabla \rho$
with the domain $D(\mathsf{A})=C_0^\infty(\R^d)$ is essentially skew-adjoint on $L^2(\R^d)$.
A~counterexample to this conjecture for $d\ge 3$ was constructed by Aizenmann.
From recent results of Alberti, Bianchini, Crippa and Panov it follows that this conjecture is false even for $d=2$.

Nevertheless, we prove that for $d=2$ the condition $p\ge 2$ is necessary and sufficient
for the following \emph{chain rule property of $\ve$}:
for any $\rho \in L^\infty(\mathbb R^2)$ and any $\beta\in C^1(\mathbb R)$
the equality $\dive(\rho \ve) = 0$ implies that $\dive(\beta(\rho) \ve) = 0$.

Furthermore, for $d=2$ we prove that $\ve$ has the renormalization property if and only if
the stream function (Hamiltonian) of $\ve$ has the weak Sard property,
and that both of the properties are equivalent to uniqueness
of bounded weak solutions to the Cauchy problem for the corresponding continuity equation.
These results generalize the criteria established for $d=2$ and $p=\infty$ by Alberti, Bianchini and Crippa.
\end{abstract}

\maketitle

\tableofcontents

\section{Introduction}

Let $\ve \colon \R^d \to \R^d$ be a compactly supported vector field which belongs to $L^p(\R^d),$ where $p\in[1,\infty].$
Suppose that 
\begin{equation}
\dive \ve =0 \label{div(v)=0}
\end{equation}
(in the sense of distributions). Let $q\in[1,\infty]$.
We would like to investigate whether the vector field $\ve$
has the following \emph{chain rule property (CRP$_q$):}
for any $\rho \in L^q(\mathbb R^d)$ and
for any $\beta\in C^1(\mathbb R)$ such that $\rho\,\ve$ and $\beta(\rho) \ve$ are integrable,
the equality
\begin{equation}
\dive(\rho \ve) = 0 \label{div(rho v)=0}
\end{equation}
implies that
\begin{equation}
\dive(\beta(\rho) \ve) = 0. \label{CRP}
\end{equation}
(Equations \eqref{div(rho v)=0} and \eqref{CRP} are understood in the sense of distributions.) If $\rho\in C^1(\R^d)$, then this implication is immediate: 
$\dive(\rho \ve) = \ve \cdot \nabla \rho$ and
$\dive(\beta(\rho) \ve) = \ve \cdot \beta'(\rho) \nabla \rho = \beta'(\rho) \dive(\rho \ve) = 0$ by \eqref{div(rho v)=0}.
Let us also mention that CRP$_q$ with $q\in[1,\infty)$ always implies CRP$_\infty$, since $\ve$ is compactly supported.

The chain rule property can be naturally considered as
a steady version of the renormalization property, which was introduced by DiPerna and Lions in connection with the problem of uniqueness of weak solutions to the initial value problem for the continuity equation.
Let us recall that $\ve$ has the renormalization property (RP$_q$) if for any $T>0$ and $\rho_0 \in L^q(\R^d)$ any weak solution $\rho\in L^\infty(0,T; L^q(\R^d))$
of
\begin{subequations}
\begin{equation}
\d_t \rho + \dive (\rho \ve) = 0,
\label{continuity-equation}
\end{equation}
\begin{equation}
\rho|_{t=0} = \rho_0
\label{continuity-equation-initial-condition}
\end{equation}
\end{subequations}
for any $\beta\in C^1(\R)$ satisfies
\begin{subequations}
\begin{equation}
\d_t \beta(\rho) + \dive (\beta(\rho) \ve) = 0,
\end{equation}
\begin{equation}
\beta(\rho)|_{t=0} = \beta(\rho_0),
\end{equation}
\end{subequations}
provided that all the terms in the weak formulation of the equations above are integrable.
If $1/p + 1/q \le 1$, then RP$_q$ implies both CRP$_q$ and uniqueness of weak solutions to \eqref{continuity-equation}--\eqref{continuity-equation-initial-condition}.
In \cite{BouchutCrippa2006} it was proved that for bounded, but possibly non-autonomous divergence-free vector fields $\ve$ the RP$_2$ is equivalent to uniqueness of weak solutions to \emph{both} forward and backward Cauchy problems for \eqref{continuity-equation}.

It follows from the well-known results of DiPerna and Lions \cite{DiPernaLions1989} that
if $1/p + 1/q \le 1$ and $\ve \in W^{1,p}(\R^d)$, then $\ve$ has the renormalization property RP$_q$
and consequently the CRP$_q$ holds.
In fact, the renormalization property holds also for non-stationary vector fields
in the class $\ve \in L^1(0,T;W^{1,p}(\R^d))$.
Later the results of DiPerna and Lions were generalized for certain classes of vector fields with bounded variation, see, e.g., \cite{Ambrosio2004,ADLM07, BG16, BianchiniBonicatto2019}.
On the other hand, if $p$ and $q$ are too low (or when $\rho$ is not bounded in time),
then RP$_q$ may fail for non-stationary vector fields:
such counterexamples were constructed in \cite{ModenaSzekelyhidi2019}, \cite{ModenaBuck2023}
and \cite{CheskidovLuo2024} using the convex integration method.

In the present work we do not assume any weak differentiability of $\ve$
and focus on the case
\begin{equation*}
 q=\infty
\end{equation*}
which is important, in particular, in connection with the theory of regular Lagrangian flows \cite{Ambrosio2004, DeLellis07, AmbrosioCrippa2008}.
In fact the difficult problem of existence and uniqueness of flows
under such minimal regularity assumptions was studied already by Nelson~\cite{Nelson_1963}.
In the case $p=2$ he considered on $L^2(\R^d)$ the unbounded operator
\begin{equation}
\mathsf{A}(\rho) = \ve \cdot \nabla \rho
\label{oper-A}
\end{equation}
with the domain $D(\mathsf{A}) = C^\infty_0(\R^d)$. Nelson formulated the following
conjecture (cf. \cite[p. 288]{Aizenman_1978}):
\begin{conj}[Nelson]\label{NC}
If $\ve\in L^2(\R^d)$ is compactly supported and divergence-free, then $\mathsf{A}$ is essentially skew adjoint, i.e., 
${\mathsf A}^*=-\overline{\mathsf A}$.
\emph{(Here $\overline{\mathsf A}$ denotes  the closure of the operator $\mathsf A$.)}
\end{conj}

As we will see shortly, the skew-adjointness of the operator ${\mathsf A}$ turns out to be closely related to the renormalization property. Originally the Conjecture~\ref{NC} was based on the following two intuitive observations:

\begin{itemize}
\item[(a)] for $p<2$ Nelson constructed a simple example of vector field $\ve\in L^p(\R^2)$ for which there exist infinitely many measure-preserving flows. But if $\mathsf A$ was essentially skew-adjoint, then such flow would be unique;
\item[(b)] for $p= 2$ under the additional assumption that $\ve$ is Lipschitz on the complement of a closed set of capacity zero Nelson proved that the operator $\mathsf{A}$ is indeed essentially skew adjoint.
\end{itemize}

Later it turned out that the integrability condition is too weak, and in the general case, without imposing additional restrictions (on differentiability, etc.), Nelson's Conjecture~\ref{NC} fails. For $d\ge 3$ corresponding counterexamples were constructed by Aizenman~\cite{Aizenman_1978}.
For $d=2$ we have not managed to find such counterexamples in the literature, and evidently the planar case is more difficult and requires new ideas, since the construction of Aizenman is essentially 3-dimensional. 
As we said before, the skew-adjointness of the operator ${\mathsf A}$ is closely related to the~renormalization property: 
indeed, for autonomous bounded divergence-free vector fields $\ve$ Panov proved \cite{Panov2015, Panov2018} that (for all $d\in \N$) the following claims are equivalent:
\begin{enumerate}
 \item $\ve$ has the renormalization property RP$_2$;
  \item operator $\mathsf{A}$ is essentially skew adjoint;
 \item the~uniqueness of weak solutions to \eqref{continuity-equation} in the class $L^\infty(0,T; L^2(\R^d))$ holds to both forward and backward Cauchy problems.
\end{enumerate}
(The equivalence of the last two properties was extended in~\cite{Panov2025} for generic skew-symmetric operator $\mathsf{A}$ on some Hilbert space.) 
By this Panov's result, in order to refute the Nelson's conjecture for $d=2$, it is sufficient to construct an autonomous bounded compactly supported divergence-free vector field on the plane, for which the RP$_2$ fails.
In this regard, in \cite{ABC_2013_LipSard,ABC_2014} Alberti, Bianchini and Crippa recently established an elegant geometric criterion of uniqueness of bounded weak solutions to \eqref{continuity-equation}--\eqref{continuity-equation-initial-condition} with bounded divergence-free vector field on the plane. They also constructed examples of such vector fields, for which uniqueness and even RP$_\infty$ fail.
This implies that Nelson's conjecture fails even for $d=2$. To make the long story short, we can summarize the above facts in the following proposition:
\begin{prop}\label{main-I}
	For any natural $d\ge 2$ there exists a compactly supported divergence-free vector field $\ve \in L^\infty(\R^d)$ 	such that the corresponding operator \eqref{oper-A} is not essentially skew adjoint on $L^2(\R^d)$.
\end{prop}

(A more detailed proof of Proposition~\ref{main-I} is given on p.~\pageref{Nelson-conj-2d}.) Despite the fact that Nelson's Conjecture was disproved in any dimension, its general intuition still seems to be interesting and fruitful.
Since any solution $\rho$ to \eqref{div(rho v)=0}
is a stationary solution of \eqref{continuity-equation},
the chain rule property can be considered as a weakened version of the renormalization property.
Since, as we already mentioned above, the essential skew adjointness of the operator~$\mathsf{A}$ implies
the renormalization property, this leads us to the following

\begin{conj}[{\bf Weakened Nelson's Conjecture}]\label{NC-w}
If $\ve\in L^2(\R^d)$, then CRP$_\infty$ holds.
\end{conj}

In order to support Conjecture~\ref{NC-w} one could observe that the CRP$_\infty$ fails for the Nelson's examples only for $L^p(\R^2)$-vector fields with~$p<2$.
However, for $d\ge 3$ even this weakened conjecture fails:
Depauw~\cite{Depauw2003} proved that
for $d\ge 3$ the CRP$_\infty$ may fail even for bounded vector fields!
Furthermore, it was proved in \cite{CGSW2017} that if $d \ge 3$, then for any strongly convex $\beta$
and any nonzero distribution $g\in \mathscr{D}'(\R^d)$ from a sufficiently large class
it is possible to construct bounded~$\ve$ and $\rho$
such that \eqref{div(v)=0}--\eqref{div(rho v)=0} hold and $\div(\beta(\rho) \ve) = g.$
In other words, not only CRP$_\infty$ fails, but it is also possible to prescribe the \emph{defect} $g$
in \eqref{div(rho v)=0}.

In the context of so many negative answers and counterexamples, it sounds astonishing that, nevertheless, the weakened Nelson's conjecture is valid for the plane case! Namely, the following  sharp characterization of the chain rule property holds:

\begin{theorem}\label{chain-rule}
If $p<2$, then there exists a compactly supported divergence-free $\ve\in L^p(\R^2)$ which does not have the CRP$_\infty$.
If $p \ge 2$, then any compactly supported divergence-free $\ve \in L^p(\R^2)$ has the CRP$_\infty$.
\end{theorem}

The first part of Theorem~\ref{chain-rule} can be proved by considering a minor modification
of the vector field constructed by Nelson \cite{Nelson_1963} (cf. \cite{Aizenman_1978}; see also \cite[p. 541--542]{DiPernaLions1989}). 
The proof of sufficiency is more involved; we will postpone the discussion of the underlying ideas for later, and now we proceed to the formulation of the second main result of the paper.

Since Nelson considered compactly supported square-integrable divergence-free vector fields on the plane, the next natural question is whether the Alberti--Bianchini--Crippa uniqueness criterion can be extended for such vector fields.
The main feature of the planar case is that by \eqref{div(v)=0} there exists a function $f\in W^{1,p}(\R^2)$ such that
\begin{equation}
	\ve = \nabla^\perp f,
	\label{stream-function}
\end{equation}
where $\nabla^\perp f \equiv (-\d_2 f, \d_1 f)$.
The~function $f$ is known as the \emph{stream function} (or \emph{Hamiltonian}) of $\ve$.
(Since~$\ve$ has a~compact support, the stream function $f$ can be chosen to be compactly supported as well.)
If $p=\infty$, then $f$ can be chosen to be Lipschitz,
and in this case it was proved in \cite{ABC_2014} that uniqueness
of bounded weak solutions to \eqref{continuity-equation}--\eqref{continuity-equation-initial-condition}
is equivalent to the following \emph{weak Sard property} of the stream function $f$:
\begin{equation*}
	f_\# \left(\L^2 \rest \{E^* \cap \{\nabla f = 0\}\}\right) \perp \L^1.
\end{equation*}
Here $\L^d$ denotes the Lebesgue measure on $\R^d$,
$E^*$ denotes the union of all connected components with positive length of all level sets of $f$,
$f_\# \mu$ denotes the image of the measure $\mu$ under the map $f$,
and $\mu \rest E$ denotes the restriction of the measure $\mu$ on the set $E$. Our second main result is the following generalization of the Alberti--Bianchini--Crippa uniqueness criterion, which originally was established for bounded divergence-free vector fields:

\begin{theorem}\label{uniqueness-criterea}
Let $p\ge 2$ and let $\ve \in L^p(\R^2)$ be a compactly supported divergence-free vector field.
Let $f \in W^{1,p}(\R^2)$ be a stream function of $\ve$.
Then the following assertions  are equivalent:
\begin{enumerate}
\item[$(i)$] $f$ has the weak Sard property;
\item[$(ii)$] $\ve$ has the renormalization property (RP$_\infty$);
\item[$(iii)$] uniqueness holds for bounded weak solutions to \eqref{continuity-equation}--\eqref{continuity-equation-initial-condition}.
\end{enumerate}
\end{theorem}

By Theorem~\ref{chain-rule} for $p\ge 2$ the CRP$_\infty$ holds without any further assumptions on $\ve$,
while by Theorem~\ref{uniqueness-criterea} uniqueness of bounded weak solutions to \eqref{continuity-equation}--\eqref{continuity-equation-initial-condition} is equivalent to the weak Sard property of $f$.
This highlights that the chain rule property (CRP$_\infty$) is strictly weaker than the renormalization property~(RP$_\infty$).

Now, let us describe the new ideas that need to be brought to the proof of this result in comparison with~\cite{ABC_2014}. One of the main difficulties, of course, is that in our case the stream function~$f\in W^{1,2}(\R^2)$ can be discontinuous. 
Indeed, in the approach of Alberti, Bianchini and Crippa, when $f$ is Lipschitz, equation \eqref{continuity-equation} (and, similarly, its steady version \eqref{div(rho v)=0})
is foliated into an equivalent family of one-dimensional equations in three steps:
\begin{enumerate}
	\item the equation is foliated into the equations on the level sets of $f$;
	\item for almost every level set the corresponding equation is further foliated into
	the equations on the connected components of the~level set;
	\item the equation on each nontrivial connected component (which is a closed simple Lipschitz curve) is rewritten using its parametrization.
\end{enumerate}
The second step relied on closedness of the level sets of $f$, which is a consequence of continuity of~$f$,
see \cite[\S 2.8]{ABC_2013_LipSard} and \cite[Lemma 3.8]{ABC_2014}.
But the level sets of a bounded function from $W^{1,2}(\R^2)$ may fail to be closed for an interval of values. For a concrete example, one can consider the function
\begin{equation}
	f(x) = \begin{cases}
		\sin \ln |\ln |x||, & x \in \R^2 \setminus \{0\}, \\
		0, & x=0,
	\end{cases}
	\quad\text{where}\quad |x| = \sqrt{x_1^2 + x_2^2}.
\end{equation}

Our starting observation in the case $p\ge 2$ is that, fortunately, Step (2) is not needed when $f$ is \emph{monotone} in the sense of \cite[Definition 3]{BT}, see Definition~\ref{def-monotone-BT}.
This is possible because almost all level sets of monotone functions are closed simple Lipschitz curves, up to $\H^1$-negligible subsets (see Remark~\ref{non-deg-rv-of-monotone-fn}).
Monotone functions $f$ also have other useful properties, in particular for them the weak Sard property turns out to be equivalent to the following \emph{relaxed Sard property} (see Corollary~\ref{pushforward-monotone-WSP}):
\begin{equation}
	f_\# \left(\L^2 \rest \{\nabla f = 0\}\right) \perp \L^1.\label{RSP-intro}
\end{equation}
In order to reduce the case of generic stream function to the case of monotone one,
we decompose it into an at most countable sum of monotone functions.
For functions of bounded variation such decomposition was established in \cite{BT}, and it was adopted to the Sobolev setting in our recent paper  (for more details see Section~\ref{sec:decomposition-into-monotone-functions}):

\begin{theorem}[\cite{BT,GK2025}]\label{decomposition-W1p}
	Let $p\in [1,\infty]$ and let $f\in W^{1,p}(\R^d)$ be compactly supported.
	Then there exists an at most countable family of compactly supported monotone functions $f_i \in W^{1,p}(\R^d)$ {\rm(}called monotone components of~$f${\rm)} such that  for each $i$ either $0 \le f_i \le 1$ or $-1 \le f_i \le 0$ a.e. and
	\begin{enumerate}
		\item[(a)] $f = \sum_{i} f_i$ {\rm(}the series converge strongly in $L^1_\loc(\R^d)${\rm)};
		\item[(b)] $|Df| = \sum_{i} |D f_i|$ {\rm(}as measures{\rm)};
		\item[(c)] the sets $\{\nabla f_i \ne 0\}$ are pairwise disjoint {\rm(}up to negligible subsets{\rm)}.
	\end{enumerate}
\end{theorem}

Property (c) in Theorem~\ref{decomposition-W1p} implies that
for all $i\ne j$ we have $\{\nabla f_j \ne 0\} \subset \{\nabla f_i = 0\} \mod \L^2$. The last property can be strengthened to the following refined version: 
\begin{equation}\label{partial-WSP}
	(f_i)_\# (\L^2 \rest \{\nabla f_j \ne 0\}) \perp \L^1.
\end{equation}
 (see Lemma~\ref{lemma-mutual-weak-Sard-property}).
This result 
allows us to prove that for a compactly supported $f\in W^{1,2}(\R^2)$
\begin{equation}\label{split-1}
\mbox{the~equation $\dive(\rho \nabla^\perp f) = 0$ is equivalent to the family of equations $\dive (\rho \nabla^\perp f_i) = 0$}
\end{equation}
(see Proposition~\ref{div-decomp}\,). 
This reduces \eqref{div(rho v)=0} to the case of monotone stream function, in which one can essentially follow the original approach of Alberti, Bianchini and Crippa.
In particular, we prove that $\dive(\rho \nabla^\perp f_i) = 0$ holds if and only if $\rho$ is constant $\H^1$-a.e. on almost all level sets of~$f_i$ (see Proposition~\ref{div-monotone-no-rhs}). This already enough for the proof of Theorem~\ref{chain-rule}.

The proof of Theorem~\ref{uniqueness-criterea} is much more involved, but it is also based on the reduction to the case of monotone stream function.
The key step in this approach is the following result, which might be of independent interest:

\begin{theorem}\label{wsp-and-decomp}
	Suppose that the assumptions of Theorem~\ref{decomposition-W1p} hold with $p=d=2$. 
	Then the function $f$ has the weak Sard property if and only if for each $i$ the monotone component $f_i$ has the weak Sard property.
\end{theorem}

Note, that here and henceforth we always assume, that all Sobolev functions (in particular, $f$ and $f_i$ from the formulation of the last theorem) are {\it precisely represented} (see Section~\ref{sec:np} for the definition and details).  
Though the claims~\eqref{split-1} and Theorem~\ref{wsp-and-decomp} look very similar, their nature is rather different, and the last assertion is much more involved: while 
\eqref{split-1}  follows from a mollification argument and the Coarea formula (see Lemma~\ref{lemma-constancy-along-connected-components} and Lemma~\ref{lemma-mutual-weak-Sard-property}), the proof of Theorem~\ref{wsp-and-decomp} requires more subtle technique. 
Indeed, a function $f$ has a~weak Sard property iff there exists a negligible subset $X$ of its critical set $Z$ such that for almost all $t\in \R$ the nontrivial connected components of the level sets $\{f=t\}$ have \emph{empty} intersection with $Z \setminus X$. 
But the Coarea formula allows to control such intersection only up to $\H^1$-negligible subsets!
Moreover, even if such intersection consists of {\it a single point} for a.e. $t$, then $f$ may still fail to have a~weak Sard property, as the example in \cite[Section~4]{ABC_2013_LipSard} shows\,!!
In order to overcome these difficulties, in Section~\ref{section-fine-properties} we establish some new fine properties of the level sets of functions from $W^{1,1}(\R^2) \cap W^{1,2}(\R^2)$.  As a byproduct of our techniques, we show that monotonicity compensates for the lack of regularity for the Sobolev embedding discussed above:
\begin{theorem}\label{monotone-Sobolev-embedding}
	Suppose that $g\in W^{1,1}(\R^2) \cap W^{1,2}(\R^2)$ is bounded and monotone. Then $g$ is continuous, i.e., $g\in C(\R^2)$. In particular, under the assumptions of Theorem~\ref{decomposition-W1p},
	if $p=d=2$, then for all $i$ the monotone components $f_i$ are continuous.
\end{theorem}
In our opinion this result is particularly surprising, since without using any mollifiers we decompose a function, which might be discontinuous, into sum of continuous functions.
Furthermore, we prove that monotone functions from $W^{1,1}(\R^2) \cap W^{1,2}(\R^2)$, which are possibly unbounded, are continuous on all plane except for at most two points, where they can have pole-like singularities.
Quite unexpectedly, the stream function of the vector field, suggested by Nelson, shows that for $f\in W^{1,p}(U)$ with $p<2$ the analog statement does not hold, which indicates sharpness of the result above.

Ultimately, in Theorem~\ref{alt-mono} we prove that a~function $f$ from $W^{1,1}(\R^2)\cap W^{1,2}(\R^2)$ is monotone if and only if for all $t\in \R$ the level set $\{f=t\}$ is connected.
The latter property is a more common definition of monotonicity for continuous functions, see, e.g., \cite[p.~358]{Engelking1989}.

Since the paper deals with a sharp version of the results under minimal regularity assumptions, we necessarily need many technical statements about the subtle properties of level sets of Sobolev functions (which are discontinuous in general, etc.).  Thus, for a reader's convenience, below we give a more detailed "map" for the paper, which is organized as follows. Section~\ref{sec:np} contains some necessary notation and includes certain classical facts from the theory of sets of finite perimeter and functions of bounded variation,
which are used to develop some preliminary results:
\begin{itemize}
\item In Section~\ref{section-coarea} we recall the Coarea formula for BV and Sobolev functions and discuss the relationship between the level sets and the essential boundaries of the upper level sets for Sobolev functions, etc.

\item In Section~\ref{sec:decomp-of-sets} we collect some properties of indecomposable and simple sets, established in \cite{ACMM_2001}. 
In Section~\ref{sec:decomposition-into-monotone-functions} we turn to functions of bounded variation and discuss the definition of monotonicity, introduced in \cite{BT}. 
Then we state a~strengthened  version of Theorem~\ref{decomposition-W1p}~--- Theorem~\ref{thm:BT} concerning the monotone decompositions. 

\item In Section~\ref{sec:vector-measures-induced-by-curves} we focus on the relationship between the simple sets on the plane and curves. In particular, we prove that the essential boundary of a set $E\subset \R^2$ with finite perimeter is an at most countable union of nontrivial closed simple Lipschitz curves, which have $\H^1$-negligible intersections with each other (see Corollary~\ref{structure-of-essential-boundary-2d}).
If $E = \{f>t\}$ for some $f\in W^{1,1}(\R^2)$ and $t\in \R$, then we denote the union of the corresponding curves by $\{f=t\}^*$, and call it the \emph{essential level set of~$f$}.

\item In Section~\ref{sec:reg-val} we discuss the properties of \emph{regular} level sets of of any precisely represented $f\in W^{1,1}(\R^2)$ for a.e. $t\in \R$. In particular, it turns out the level set $\{f>t\}$ has finite perimeter, $\{f=t\}$ and $\{f=t\}^*$ coincide up to $\H^1$-negligible subsets, the gradient $\nabla f$ does not vanish $\H^1$-a.e. on $\{f=t\}$, and, ultimately, its direction agrees with  generalized inner normal to $\{f>t\}$ $\H^1$-a.e. on $\{f=t\}$ (see Theorem~\ref{reg-vals}).
The values $t$ for which the level sets of $f$ have these properties for convenience will be called \emph{regular}. 
\end{itemize}

In section~\ref{sec:parallel-gradients} we study some properties of the regular level sets of
two functions $f,g \in W^{1,1}(\R^2) \cap W^{1,2}(\R^2)$ with $\nabla f \parallel \nabla g$ a.e.
We prove that for a.e. regular value $t$ of $g$ for every curve $\gamma$ from $\{g=t\}^*$ the function $f$ is constant $\H^1$-a.e. on $\gamma$ (Lemma~\ref{lemma-constancy-along-connected-components}). 
In the particular case when the gradients of $f$ and $g$ are concentrated on disjoint subsets, we show in Lemma~\ref{lemma-mutual-weak-Sard-property} that the corresponding essential values of $f$ are not regular (with respect to $f$). This leads, in particular, to \eqref{partial-WSP}. 

In Section~\ref{sec:crp}, we digress from further investigation of the level sets to apply the developed results to the chain rule problem. In particular, we prove Theorem~\ref{chain-rule} and Proposition~\ref{main-I}.

In Section~\ref{section-fine-properties}, we resume the investigation of the level sets of Sobolev functions.
First we show that for monotone functions from $W^{1,1}(\R^2)$ the \emph{pointwise} inclusion $\{f=t\} \subset \{f=t\}^*$ holds for a.e. $t$ (see Theorem~\ref{levelset-subset-of-essential-levelset}).
In general, the converse inclusion does not hold, even if $\nabla f \in L^p(\R^2)$ with $p<2$ (see Remark~\ref{Nelson-reprise}).
However, under the additional assumption that $f\in W^{1,2}(\R^2)$, we prove that $\{f=t\}$ coincides with $\{f=t\}^*$ for a.e. $t\in \R$ (see Proposition~\ref{almost-all-level-sets-are-essential}). In Section~\ref{sec-connected-components}, we study the connected components of the level sets of functions from $W^{1,2}(\R^2)$. 
In particular, we show that the union of all non-trivial connected components of their level sets is Borel (for Lipschitz functions, such a result was obtained
in~\cite{ABC_2013_LipSard}). In Section~\ref{sec-monotonicity-and-continuity}, we prove the equivalence of the two definitions of monotonicity mentioned above
(see Theorem~\ref{alt-mono}).

In Section~\ref{sec:wsp}, we study the interplay between the weak Sard property and a decomposition of a function from $W^{1,1}(\R^2) \cap W^{1,2}(\R^2)$ into monotone components. In particular, we prove Theorem~\ref{wsp-and-decomp}.

In Section~\ref{sec:uniq}, we finally turn to the continuity equation and prove Theorem~\ref{uniqueness-criterea}.
The core of our approach is the decomposition of the stream function into monotone components.
It helps us to avoid the issue that the level sets of functions from $W^{1,2}(\R^2)$ are not necessarily closed.
Besides, Theorem~\ref{wsp-and-decomp} allows us to establish sufficiency of the weak Sard property for uniqueness without using the disintegration theorem.

\section{Notation and preliminaries}\label{sec:np}
For any vector $a \in \R^d$ let $|a|$ denote the Euclidean norm of $a.$
For any function $f\colon E \to \R$ (where $E \subset \R^d$) let $\{f>t\},$ $\{f=t\}$ and $\{f<t\}$ denote respectively the corresponding superlevel (upper level set), level set and sublevel (lower level set) of $f$.
Let $\1_E$ denote the indicator of the set $E.$
Given a Borel vector measure $\mu$ on $\R^d$ let $|\mu|$ denote the total variation of the measure $\mu$.
By $\mu\rest E$ we denote the restriction of $\mu$ to $E,$
i.e., $(\mu \rest E)(A) = \mu(E \cap A)$ for any Borel set $A \subset \R^d$.
For any Borel function $f\colon \R^n \to \R^m$ and any Borel measure $\mu$ on $\R^n$
the image (pushforward) of $\mu$ under $f$ is denoted with~$f_\# \mu$.
If $\mu \ge 0$, then $L^1(\mu)$ denotes the space of functions integrable with respect to $\mu$.
In our notation, the elements of $L^1(\mu)$ are fixed Borel functions defined at every point (not equivalence classes). If two Borel measures, $\mu$ and $\nu$, are mutually singular, we write $\mu \perp \nu$.
If two sets $A$ and $B$ coincide up to a $|\mu|$-negligible set (i.e., $|\mu|(A \triangle B) = 0$),
then we say that $A=B$ modulo $\mu$ (and write $A=B$ $\mod \mu$).
The complement of a set $E\subset \R^d$ will be denoted with $E^c$.
As usual, we say that $\mu$ is concentrated on $E$ if $|\mu|(E^c)=0$.
Given a sequence $\{\mu_k\}_{k\in \N}$ on non-negative Borel measures on $\R^d$, we say that $\mu_k \to \mu$ (as measures), if for any Borel set $B\subset \R^d$ we have $\mu_k(B) \to \mu(B)$ as $k\to \infty$.

Given an open set $U \subset \R^d$ let $C_c^\infty(U)$ denote the set of compactly supported infinitely differentiable functions $\fhi\colon U \to \R.$ Let $C_c^\infty(U)^m$ denote the set of compactly supported infinitely differentiable maps $\bm \fhi \colon U \to \R^m$ with the components $(\fhi_1, \ldots, \fhi_m)$; a similar notation will be used for vector-valued functions from other functional spaces, and we will often omit the superscript $m$ whenever it is clear from the context. The set of Lipschitz functions on the set $E \subset \R^d$ is denoted with $\Lip(E)$. As usual, the $s$-dimensional Hausdorff measure is denoted with $\H^s.$

Working with Sobolev functions~--- especially, with stream functions (Hamiltonians),~--- we {\it always assume} that
the precise representatives are chosen. Recall, that if ${U}\subset\R^d$ is an~open set and $f\in
W^{1,q}_\loc({U})$, then the precise representative $f^*$ is
defined for \textit{all} $x \in {U}$ by
\begin{equation}
\label{lrule}f^*(x)=\left\{\begin{array}{rcl} {\displaystyle
\lim\limits_{r\searrow 0} \dashint_{B(x,r)}{f}(y)\,\dd y}, &
\mbox{ if the limit exists
and is finite,}\\
 0 \qquad\qquad\quad & \; \mbox{ otherwise}
\end{array}\right.
\end{equation}
where the dashed integral as usual denotes the integral mean,
$$
\dashint_{B(x,r)}{f}(y) \, \dd
y=\frac{1}{\L^n(B(x,r))}\int_{B(x,r)}{ w}(y)\,\dd y,
$$
and $B(x,r)=\{y: |y-x|<r\}$ is the open ball of radius $r$
centered at $x$. Henceforth we omit special notation for the
precise representative assuming simply $f^* = f$. Furthermore, recall, that $x\in{U}$ is a Lebesgue point for $f$, if 
\begin{equation}\label{app-lim}
\lim_{r\to 0} \dashint_{B(x,r)}|f(y) - f(x)| \, \dd y = 0.
\end{equation}

The set $S_f$ of all non-Lebesgue points $x\in U$ where this property does not hold is called the \emph{approximate discontinuity set of $f$}. Recall that $S_f$ is an $\L^d$-negligible Borel set for locally integrable functions (see, e.g., \cite[Proposition~3.64]{AFP}).
It is well-known that for Sobolev functions the approximate discontinuity set is even smaller:

\begin{prop}\label{approx-discon-set-of-Sobolev-fn}
	If $f\in W^{1,1}_\loc(U)$, then $\H^{d-1}(S_f) = 0$.
\end{prop}

\begin{proof}
	Without loss of generality let us assume that $U=\R^d$ and $f\in W^{1,1}(\R^d)$;
	the general case can be easily reduced to this setting using some appropriate cutoff functions.
	
	By \cite[Theorem 4.19]{EG2015} there exists a Borel set $E \subset U$ such that
	\begin{enumerate}
		\item $\operatorname{Cap}_1(E) = 0$, where $\operatorname{Cap}_p(E)$ denotes the $p$-capacity of the set $E$; see for instance \cite[Definition 4.10]{EG2015}.
		\item for all $x\in U\setminus E$ there exists finite $f(x)$ and
		\begin{equation}
			\lim_{r\to 0} \dashint_{B_r(x)} |f(y) - f(x)|^{1^*} \, dy = 0,
		\end{equation}
		where $1^* = \frac{d}{d-1}$.
	\end{enumerate}
	Consequently, for all $x\in U\setminus E$
	\begin{equation}
		\lim_{r\to 0} \dashint_{B_r(x)} |f(y) - f(x)| \, dy = 0,
	\end{equation}
	and therefore $S_f \subset E$. On the other hand $\operatorname{Cap}_1(E) = 0$ is equivalent to $\H^{d-1}(E) = 0$ by \cite[Remark after Theorem 4.17]{EG2015}. 
	
	Alternatively, one could use \cite[Theorem 3.78]{AFP} and recall that the approximate jump set for Sobolev functions is negligible (for instance, by \cite[Lemma 3.76]{AFP}).
\end{proof}

In particular, for $d=2$ we conclude that 
\begin{equation}
\H^1(S_f) = 0\quad\mbox{ for any }f\in W^{1,1}_\loc(U).
	\end{equation}

\subsection{Functions of bounded variation and sets of finite perimeter} The set of Lipschitz functions on the set $E \subset \R^d$ is denoted with $\Lip(E)$. As usual, the $s$-dimensional Hausdorff measure is denoted with $\H^s.$
Let $\BV(U)$ denote the space of functions $f\colon U \to \R$ of bounded variation.
Given $f\in \BV(U)$, let $Df$ denote the vector measure associated with the derivative of $f.$
In particular, for any $\bm \fhi \in C_c^\infty(U)$ we have
$$
\int_U f \, \div \bm \fhi \, dx = - \int_U \bm \fhi \, d Df.
$$
If $f$ belongs to the Sobolev space $W^{1,1}(U)$, then $Df = (\nabla f) \L^d,$
where $\L^d$ denotes the Lebesgue measure on $\R^d.$
The total variation of a function $f\in \BV(U)$ will be denoted with $TV(f),$
i.e., $TV(f) = |Df|(U)$.

Recall that, given a (Lebesgue) measurable set $E \subset \R^d$,
\begin{equation*}
P(E) := \sup\left\{\int_E \div \bm \fhi \, dx \; \middle| \; \bm \fhi \in C_c^\infty(\R^d), \; \|\bm \fhi\|_{\infty} \le 1 \right\}
\end{equation*}
is called the \emph{perimeter of $E$.}
From the definition it is clear that for any measurable set $E$ the perimeter $P(E)$
is equal to the total variation of the indicator function $\1_E$ of this set.
If $P(E) < +\infty$, then $E \subset \R^d$ is called \emph{a set of finite perimeter}.
A set $E$ with finite measure has finite perimeter iff $\1_E \in \BV(\R^d)$
(and in this case $P(E) = |D\1_E|(\R^d)$).

For any $\alpha \in [0,1]$ and any measurable $E \subset \R^d$ let
\begin{equation}
E^\alpha := \left\{x\in \R^d \;\middle|\; \lim_{r\to 0} \tfrac{\L^d(B_r(x) \cap E)}{\L^d(B_r(x))} =\alpha \right\},
\label{alpha-density}
\end{equation}
where $B_r(x)$ denotes the open ball with radius $r$ and center $x.$
If $E$ has finite perimeter, then the set
\begin{equation*}
\d^* E := \left\{x\in \R^d \;\middle|\; \exists \bm \nu(x) := \lim_{r\to 0} \tfrac{D\1_E(B_r(x))}{|D\1_E|(B_r(x))}\text{ and } |\bm \nu(x)|=1 \right\}
\end{equation*}
is called the \emph{reduced boundary of $E$}.
The vector field $\bm \nu_E := \bm \nu$ defined above is called the \emph{generalized inner normal to $E$}.
Furthermore, the set $\d^M E := \R^d \setminus (E^0 \cup E^1)$ is called the \emph{essential boundary of $E$}.
Of course, for any measurable $E \subset \R^d$ we have  $\d^M E \subset \d E$.

Now let us collect some standard facts from \cite[\S 2.11]{AFP} about the
approximate tangent spaces.
Let $\G_m$ denote the Grassmanian manifold of $m$-dimensional subspaces of $\R^d$.
Let $\mu$ be a Borel measure on an open set $\Omega \subset \R^d$ and let $x\in \Omega$.
For any $r>0$ let
\begin{equation*}
	\mu_{x,r} := (\Phi_{x,r})_{\#} \mu, \qquad \text{ where } \qquad \Phi_{x,r}(y) := \frac{y-x}{r}.
\end{equation*}
We say that \emph{$\mu$ has approximate tangent space $\pi \in \G_m$} if $r^{-d} \mu_{x,r} \to \H^m \rest \pi$ weakly* as $r\to 0$;
in this case we write $\Tan(\mu, x) = \pi$.
If $E \subset \Omega$ is an $\H^m$-measurable set with $\H^m(E)<\infty$ and $\Tan(\H^m \rest E, x) = \pi$, then $\pi$ is called the \emph{approximate tangent space to $E$ at $x$};
in this case we write $\Tan^m(E, x) = \pi$.

For images of measurable sets under injective Lipschitz maps the approximate tangent spaces exist and coincide with the range of the differential: 

\begin{prop}[{\cite[Proposition 2.88]{AFP}}]\label{prop:tang_smooth}
	Let $\fhi \colon \R^k \to \R^N$ be an injective Lipschitz map and let $D \subset \R^k$ be a $\L^k$-measurable set. Then
	\begin{equation*}
		\Tan(\fhi(D),y) = d\fhi_{\fhi^{-1}(y)}(\R^k)  \quad \text{ for } \H^{k} \text{-a.e. } y \in \fhi(D),
	\end{equation*}
	where $d\fhi_x$ is the usual differential of $\phi$ at $x$.
\end{prop}

For sets with finite perimeter the approximate tangent space to the reduced boundary is determined by the generalized inner normal:

\begin{theorem}[De Giorgi, {\cite[Theorem 3.59 and Eq. (3.60)]{AFP}}] \label{thm:de_giorgi} Let $E\subset\R^d$ be a Lebesgue measurable set of finite perimeter. Then $\d^* E$ is countably $(d-1)$-rectifiable and $\vert D\1_E \vert = \H^{d-1}\rest \d^* E$. In addition, the approximate tangent space to $E$ at $x$ coincide with the orthogonal hyperplane to $\bm\nu_E(x)$ for $\H^{d-1}$-a.e. $x \in \d^* E$, i.e., 
	\begin{equation*}
		\Tan(\d^* E,x) = \bm\nu_E^\perp(x).   
	\end{equation*}
\end{theorem}

The link between the reduced boundary, the essential boundary and the set of points of density $1/2$ is provided by the following theorem (see \cite[Theorem 3.61]{AFP}):
\begin{theorem}[Federer]\label{thm:de_giorgi_federer}
	If $E \subset \R^d$ has finite perimeter, then 
	\begin{equation*}
		\d^* E \subset E^{1/2} \subset \partial^M E
	\end{equation*}
	and $\H^{d-1}(\partial^M E \setminus E^*)=0$.
\end{theorem}

If $E$ has finite perimeter, then by the generalized Gauss-Green theorem (see e.g. \cite{AFP})
\begin{equation}
	D \1_E = \bm \nu_E \H^{d-1} \rest \d^* E,
	\label{Gauss-Green}
\end{equation}
where $\bm \nu_E$ is the generalized inner normal to $E$.
Furthermore, it is well-known that $\d^* E,$ $\d^M E$ and $E^{1/2}$ coincide up to $\H^{d-1}$-negligible sets, so any of these sets can be used in the right-hand side of \eqref{Gauss-Green}.

\subsection{Coarea formula}\label{section-coarea}

One of the main technical tools used in this paper is the Coarea formula for BV and Sobolev functions, see for instance \cite[Theorem 3.40]{AFP} and \cite[Theorem 1.1]{MSZ_2003}.

\begin{theorem}[Coarea formula for BV]
Let $f\in \BV(\R^d)$. 
Then for a.e. $t\in \R$ the set $\{f>t\}$ has finite perimeter and
\begin{equation}
	|Df|(B) = \int_\R |D \1_{\{f>t\}}|(B) \, dt, \qquad Df(B) = \int_\R D \1_{\{f>t\}}(B) \, dt.
\end{equation}
for any Borel set $B \subset \R^d$.
\end{theorem}

Under assumptions of the previous theorem, for any bounded Borel function $\fhi \colon \R^d \to \R$ we also have
\begin{equation}
	\int_{\R^d} \fhi \, d |Df| = \int_\R \int_{\R^d} \fhi \, d |D\1_{\{f>t\}}|  \, dt,
	\qquad
	\int_{\R^d} \fhi \, d Df = \int_\R \int_{\R^d} \fhi \, d D\1_{\{f>t\}}  \, dt.
\end{equation}

\begin{theorem}[Coarea formula for Sobolev functions]
	Let $f\in W^{1,1}(\R^d)$. 
	Then for a.e. $t\in \R$ the set $\{f=t\}$ is countably $\H^{d-1}$-rectifiable and
	\begin{equation}\label{coarea-s}
		\int_{B} |\nabla f| \, dx = \int_\R \H^{d-1}(\{f=t\} \cap B) \, dt
	\end{equation}
	for any measurable set $B \subset \R^d$.
\end{theorem}

Under assumptions of the previous theorem, for any bounded measurable function $\fhi \colon \R^d \to \R$ we also have

\begin{equation}\label{coarea-fhi}
	\int_{B} \fhi |\nabla f| \, dx = \int_\R \int_{\{f=t\}} \fhi \, d \H^{d-1} \, dt.
\end{equation}
By monotone convergence theorems, the last identity is also true if  $\fhi$ is measurable and nonnegative (not necessarily bounded). 

The next two statements (Propositions~\ref{constancy on reduced boundary}--\ref{gen-inner-normal}) are widely known among experts and often used; this is a kind of mathematical folklore, but for the reader's convenience, we present an accurate proof based on the previous theorems in Appendix~\ref{appendix-connectedness-lemma}. 

\begin{prop}\label{constancy on reduced boundary}
	If $f \in W^{1,1}(\R^d)$, then for a.e. $t\in \R$ we have
	\begin{equation}
		\d^M \{f>t\} = \{f=t\}
		\label{f is constant on the reduced boundary of superlevelset}
	\end{equation}
	up to $\H^{d-1}$-negligible subsets.
\end{prop}

\begin{prop}\label{gen-inner-normal}
	If $f\in W^{1,1}(\R^d)$, then for a.e. $t\in \R$ the properties
	\begin{subequations}
		\begin{equation}
			\nabla f \ne 0,
			\label{gradient does not vanish}
		\end{equation}
		\begin{equation}
			\bm \nu_{\{f>t\}} = \frac{\nabla f}{|\nabla f|}
			\label{generalized inner normal for Sobolev function}
		\end{equation}
	\end{subequations}
	hold $\H^1$-a.e. on $\{f=t\}$.
\end{prop}

\subsection{Decomposition of sets with finite perimeter}\label{sec:decomp-of-sets}
A measurable set $E \subset \R^d$ with finite perimeter is called \emph{indecomposable} if it cannot be written as a disjoint union $E = A \sqcup B$ of two measurable sets $A$ and $B$ such that $P(E) = P(A) + P(B).$
For open sets there is a natural criterion of indecomposablity (see \cite[Proposition 2 and Theorem 2]{ACMM_2001}):
\begin{theorem}\label{indecomposablity-criterion-for-open-sets}
	An open set in $U \subset \R^d$ with $\H^{d-1}(\d U)<\infty$ is indecomposable if and only if it is connected.
\end{theorem}
Let us recall some classical results obtained in \cite[Proposition 3 and Theorem 1]{ACMM_2001}.

\begin{prop}
\label{indec-prop}
Let $\{A_i\}_{i\in I}$ be a finite or countable family of sets with finite perimeter and let $A=\bigcup_{i\in I} A_i$. Suppose that $A_i \ne \R^N$ for all $i\in I$
and $\sum_i P(A_i) < \infty.$ Then the following conditions are equivalent:
\begin{itemize}
 \item[(i)] $P(A) \ge \sum_i P(A_i)$;
 \item[(ii)] $P(A) = \sum_i P(A_i)$;
 \item[(iii)] for any $i\ne j$ we have $|A_i \cap A_j| = 0$ and $\H^{d-1}(\d^M A_i \cap \d^M A_j)=0$;
 \item[(iv)] for any $i\ne j$ we have $|A_i \cap A_j| = 0$ and $\bigcup_i \d^M A_i \subset \d^M A$
 (up to $\H^{d-1}$-negligible subsets).
\end{itemize}
If these conditions are fulfilled, then we have also $\d^M A = \bigcup_i \d^M A_i$
(modulo $\H^{d-1}$) and
\begin{equation*}
\H^{d-1}(A^1 \setminus \bigcup_{i\in I} A_i^1) = 0.
\end{equation*}
where $A^1$ is defined in \eqref{alpha-density}.
\end{prop}

\begin{theorem}[Decomposition theorem]\label{thm:ACMM}
Let $E$ be a set with finite perimeter in $\R^d$. Then there exists a unique (up to permutations) at most countable family of pairwise disjoint indecomposable sets $\{E_i\}_{i\in I}$ such that $\L^d(E_i)>0$, $E=\bigcup_{i\in I} E_i$, and $P(E) = \sum_{i} P(E_i)$.
Moreover,
\begin{equation*}
\H^{d-1}(E^1 \setminus \bigcup_{i\in I} E_i^1) = 0,
\end{equation*}
and for any indecomposable $F \subseteq E$ with $\L^d(F)>0$ there exists a unique $j\in I$ such that $F \subset E_j$ (up to $\L^d$-negligible subsets).
\end{theorem}

The sets $E_i$ in the statement of Decomposition theorem are called the \emph{$M$-connected components of~$E$.} Accordingly we denote $CC^M(E):=\{E_i\}_{i\in I}.$
By Proposition \ref{indec-prop} we have $\H^{d-1}(\d^M E_i \cap \d^M E_j) = 0$ whenever $i\ne j.$
Furthermore, if $\meas{E} = +\infty$, then there exists a unique $j\in I$ such that $\meas{E_j} = \infty$
(see \cite[Lemma 2.16]{BonicattoGusev2021}).

The statement of Decomposition theorem can be slightly strengthened with the following simple corollary of Proposition~\ref{indec-prop} (see also Equation (10) in \cite[Remark~1]{ACMM_2001}):

\begin{prop}\label{p-decomposition-th+}
Let $E \subseteq \R^d$ be a set with finite perimeter. Let $CC^M(E) = \{E_i\}_{i\in I}$, where $I$ is at most countable. Then $P(\bigcup_{i\in I_1 \cup I_2} E_i) = P(\bigcup_{i\in I_1} E_i) + P(\bigcup_{i\in I_2} E_i)$ for any disjoint sets $I_1, I_2 \subseteq I$.
\end{prop}

The $M$-connected components of $E^c:= \R^d \setminus E$ which have finite measure are called the \emph{holes of $E$}.
An indecomposable set $E \subset \R^d$ is called \emph{simple} if it has no holes.

For $d>1$ the only simple set with infinite measure is $\R^d$
(modulo $\L^d$). Furthermore, a set $E\subset \R^d$ with finite measure is simple if and only if both $E$ and $E^c$ are indecomposable (see \cite[Proposition~2.17]{BonicattoGusev2021}).

For any indecomposable set $E\subset \R^d$ the \emph{saturation of $E$} (denoted with $\operatorname{sat}(E)$)
is defined as the union of $E$ and its holes. For a generic set $E \subset \R^d$ with finite perimeter
the saturation of $E$ is defined as $\sat(E) := \bigcup_{i\in I} \sat(E_i),$ where $CC^M(E) = \{E_i\}_{i\in I}.$
A set $E \subset \R^d$ with finite perimeter is called \emph{saturated}, if $E = \sat(E).$
Clearly a set $E \subset \R^d$ is simple if and only if it is indecomposable and saturated.

Let us state some properties of holes \cite[Propositions 6 and 9]{ACMM_2001}:

\begin{prop}\label{holes-of-indecomposable-set}
Let $E \subset \R^d$ be an indecomposable set. Then
\begin{enumerate}
 \item[(i)] Any hole of $E$ is a simple set.
 \item[(ii)] $\sat(E)$ is a simple set, $\d^M \sat(E) \subset \d^M E$ (modulo $\H^{d-1}$),
 and $\sat(E)$ has finite measure if $\meas{E} < \infty.$
 \item[(iii)] if $E \subset \sat(F)$, then $\sat(E) \subset \sat(F).$
 \item[(iv)] if $F$ is indecomposable and $\meas{F \cap E} = 0$, then the sets $\sat(E)$ and $\sat(F)$
 are either one a subset of another, or are disjoint.
\end{enumerate}
\end{prop}

\begin{prop}\label{saturation-and-holes}
Let $E \subset \R^d$ be an indecomposable set and let $\{Y_i\}_{i\in I}$ be its holes.
Then we have $E = \sat(E) \setminus \bigcup_{i\in I} Y_i$ and $P(E) = P(\sat(E)) + \sum_{i\in I} P(Y_i)$.
\end{prop}

\subsection{Decomposition of functions of bounded variation}\label{sec:decomposition-into-monotone-functions}

For convenience we will use the following function space (as in \cite{BonicattoGusev2021}):
\begin{equation*}
	\FV(\R^d) :=  \{f \in L^{1^*}(\R^d): V(f)<+\infty\},
\end{equation*}
where $V(f)=V(f, \R^d)$ is the variation of the function $f$ and
\begin{equation*}
	1^* := \begin{cases}
		\frac{d}{d-1}, & d>1 \\
		\infty, & d=1.
	\end{cases}
\end{equation*}

It is easy to see that $\BV(\R^d) \subset \FV(\R^d) \subset \BV_{\loc}(\R^d)$ and both inclusion are strict.
Furthermore, if we endow $\FV(\R^d)$ with the norm $\|f\| := V(f)$
then $\FV(\R^d)$ becomes a Banach space, since $\|f\|_{1^*} \le V(f)$ by the BV embedding theorem (see e.g. \cite[Theorem 3.47]{AFP}).
If $d>1$, then for any $f\in L^{1^*}(\R^d)$ the sets $\{f>t\}$ and $\{f<-t\}$ have finite $\L^d$-measure for any $t>0$,
since $\|f\|_p = \int_0^{+\infty} \meas{\{|f|^p>t\}} \, dt$, $1\le p<+\infty$.

The following definition was introduced by Bianchini and Tonon \cite{BT} for BV functions:

\begin{defn}\label{def-monotone-BT}
	A function $f \in \FV(\R^d)$ is said to be \emph{monotone} if the sets $\{f>t\}$ and $\{f<t\}$ are indecomposable for a.e. $t \in \R$.
\end{defn}

In other words, a function from $\FV(\R^d)$ is monotone, if the sets $\{f>t\}$ and $\{f<-t\}$ are simple for a.e. $t>0$.
There are some other ways to define monotonicity in the case of functions of several variables, see, e.g., \cite{VodGol76} or references at~\cite{BT}, we discuss some equivalent criterion further in Section~\ref{sec-monotonicity-and-continuity}. 
The~next result from our recent article~\cite{GK2025}  is one of the main tools for the present paper. 

\begin{theorem}\label{thm:BT}
	For any $f\in \FV(\R^d)$ there exist at most countable family $\{f_i\}_{i\in I} \subset \FV(\R^d)$
	of monotone functions {\rm(}called \emph{monotone components of $f$}\,{\rm)} such that
	\begin{enumerate}
		\item[(o)] for each $i\in I$ the set $\{|f_i| > 0\}$ has finite measure and either $0\le f_i\le 1$ or $-1\le f_i \le 0$ a.e.;
		\item[(i)] $f = \sum_{i} f_i$ the series converges in $L^{1^*}$\,{\rm)};
		\item[(ii)] $|Df| = \sum_{i} |D f_i|$;
	\end{enumerate}
	If, furthermore, $Df \ll \L^d$, then the functions $f_i$ can be chosen in such a way that additionally
	\begin{enumerate}
		\item[(iii)] $Df_i \ll \L^d$;
		\item[(iv)] if $i\ne j$ then $|Df_i| \perp |Df_j|.$
	\end{enumerate}
\end{theorem}

Really, in~\cite{GK2025} the assertions~(i)--(iv) of the theorem were proved. The simple proof of the additional property~(o) is given below.  
For Lipschitz functions with compact support the corresponding result again with properties (i)--(iv) was mentioned in~\cite{BT}, where it was extended to the BV setting, but without properties (iii) and (iv) (which do not hold in the BV setting in general).
A simpler proof of the result in the BV setting with the properties (i) and (ii) is given in \cite{BonicattoGusev2021}.
Note, that even if $f$ belongs to $W^{1,1}(\R^2)\cap W^{1,2}(\R^2)$, then property~(iv) in Theorem~\ref{thm:BT} does not follow immediately from the other properties.

\begin{proof}[Proof of Theorem~\ref{thm:BT}] Let $f = \sum_{i} f_i$ be a~decomposition of $f$ into the sum of monotone functions~$\{f_i\}_{i\in I} \subset \FV(\R^d)$ satisfying~(i)--(ii) or (i)--(iv) respectively (whose existence was proved in~\cite{GK2025}).
It remains to establish the last property~(o), and the simplest way to do this is to decompose further the functions $f_i$.
Let us fix $i\in I$ and $g:= f_i$.
Since by construction we have either $g \ge 0$ or $g \le 0$ a.e., without loss of generality let us suppose that $g \ge 0$.
Let $\{t_k\}_{k\in \Z}$ be a set of points on $(0,+\infty)$ such that for all $k\in \Z$ we have
\begin{itemize}
	\item $0<t_{k-1}<t_{k}\le t_{k-1}+1$,
	\item $\{g > t_k\}$ is simple (hence $\meas{\{g > t_k\}} < +\infty$).
\end{itemize}
For each $k\in \Z$ let
\begin{equation*}
	\psi_k(t) := \begin{cases}
		t_{k-1}, & t\le t_{k-1}, \\
		t, &t_{k-1} < t < t_{k}, \\
		t_{k}, & t\ge t_{k}.
	\end{cases}
\end{equation*}
Let us decompose $g$ as follows: 
\begin{equation*}
	g(x) = \sum_{k\in \Z} \fhi_k(x), \qquad \fhi_k(x):= \psi_k(g(x)) - t_{k-1}
\end{equation*}
By construction, the functions $\fhi_k$ belong to $\FV(\R^d)$.
By the Coarea formula we have $V(g) = \sum_{k} V(\fhi_k)$.
Further, if $g\in W^{1,1}_\loc(\R^d)$, then by construction of $\psi_k$ we also have $\fhi_k \in W^{1,1}_\loc(\R^d)$.

By definition of $\fhi_k$ we have $0\le \fhi_k \le 1$ and $\{\fhi_k \ne 0\} \subset \{g>t_{k-1}\}$.
The latter implies that $\meas{\{\fhi_k \ne 0\}} < +\infty$.
\end{proof}

Note that if $Df \ll \L^d,$ then by {(ii)} for all $i$ we have $|\nabla f_i| \le |\nabla f|.$
In particular, if $\nabla f \in L^p(\R^d)$ with some $p\in [1,+\infty]$, then for all $i$ we have $\nabla f_i \in L^p(\R^d)$.
Since the sets $\{f_i \ne 0\}$ have finite measures and the functions $f_i$ are bounded, we also have $f_i \in L^p(\R^d)$.
Consequently, if $f\in W^{1,1}(\R^d) \cap W^{1,p}(\R^d)$, then the monotone components $f_i$ belong to $W^{1,1}(\R^d) \cap W^{1,p}(\R^d)$ as well.

Ultimately, if $f$ is compactly supported, then for all $i$ the functions $f_i$ are compactly supported as well
(since $|Df_i| \le |Df| = 0$ outside some ball and $f_i \in L^{1^*}(\R^d)$). Note, that the equality $\|\nabla f\|_1 = \sum_{i} \|\nabla f_i\|_1$ follows from (ii). 

\subsection{Vector measures induced by curves}\label{sec:vector-measures-induced-by-curves}
Let $\gamma \colon [0,1] \to \R^d$ be a continuous function.
If $\gamma$ is Lipschitz, then the curve (parametrized by) $\gamma$ is called \emph{Lipschitz}.
If $\gamma$ is injective on $[0,1)$  with $\gamma(1) \notin \gamma((0,1))$, then the curve $\gamma$ is called \emph{simple}. If $\gamma(0) = \gamma(1)$, then the curve $\gamma$ is called \emph{closed}.
 
Let $\gamma\colon [0,1] \to \R^d$ be a simple (possibly closed) Lipschitz curve such that $\gamma'\ne 0$ a.e.
The inverse map $\gamma^{-1}$ is Borel measurable since $\gamma$ is injective and continuous.
Hence there exists a unit-length Borel vector field $\bm \tau \colon \gamma([0,1]) \to \R^d$
such that
\begin{equation}
\bm\tau(\gamma(s)) = \frac{\gamma'(s)}{|\gamma'(s)|}
\label{canonical-tangent}
\end{equation}
for a.e. $s\in [0,1]$.

By area formula for any $\bm \fhi \in C_c(\R^d)^d$
\begin{equation*}
\int_0^1 \bm \fhi(\gamma(s)) \cdot \gamma'(s) \, ds
=\int_{\R^d} \bm \fhi \cdot \, d\bm \mu_\gamma,
\end{equation*}
where
\begin{equation*}
\bm \mu_\gamma := \bm \tau \H^1\rest \gamma([0,1])
\end{equation*}
is called the \emph{vector measure induced by $\gamma$}
(where $\bm \tau$ is defined in \eqref{canonical-tangent}).

The divergence equation \eqref{div(rho v)=0} can be written as $\dive (\rho \bm \mu) = 0,$
where $\bm \mu = \ve \L^d$. If we replace this measure $\bm \mu$ with a measure induced
by a curve $\gamma,$ then the divergence equation can be solved using
the following result \cite[Theorem 4.9]{BG16}:

\begin{prop}
\label{steady continuity equation along curve}
Let $\gamma\colon [0,1] \to \R^d$ be a closed simple Lipschitz curve such that $\gamma'\ne 0$ a.e. and let $\Gamma:=\gamma([0,1])$.
Let $\rho \in L^1(\H^1\rest \Gamma)$ and let $\nu$ be a Borel measure on $\Gamma$.
Then
\begin{equation}
\div (\rho \bm \mu_\gamma) = \nu
\label{continuity-eq-along}
\end{equation}
(in the sense of distributions) if and only if $\rho \circ \gamma \in L^\infty(\Gamma)$ and
\begin{equation*}
(\rho\circ \gamma)' = \nu \circ \gamma
\end{equation*}
holds in the sense of distributions,
where $\nu \circ \gamma = \gamma^{-1}_\# \nu.$
In particular, $\div (\rho \bm \mu_\gamma) = 0$
if and only if there exists a constant $C\in \R$
such that $\rho\circ \gamma = C$ a.e. on $[0,1]$.
\end{prop}

Now let us turn to the two-dimensional setting ($d=2$), in which the essential boundaries of sets with finite perimeter and also the right-hand side of \eqref{Gauss-Green} can be described more precisely using curves (see Propositions~\ref{generalized-inner-normal-for-simple-set} and~\ref{derivative of indicator of a set of finite perimeter}).

Let $\gamma \colon [0,1]\to \R^2$ be a closed simple curve.
By the Jordan curve theorem the open set $\R^2 \setminus \gamma([0,1])$ has exactly two connected components (with common boundary $\gamma([0,1])$), just one of which is bounded.
Let us denote it with $I_\gamma$ (interior), and let $E_\gamma := \R^2 \setminus \overline{I_\gamma}$ denote the other one (exterior).

It follows from the Jordan curve theorem that if $U \subset \R^2$ is a bounded open set and $\d U$ is a closed simple curve, then $U$ can be recovered from $\d U$.
A similar result holds for sets with finite perimeter, whose essential boundary is a closed simple Lipschitz curve (up to $\H^1$-negligible subsets):

\begin{lemma}\label{Jordan-interior}
	Let $E \subset \R^2$ be a set with strictly positive finite measure and finite perimeter.
	Let $\gamma\colon [0,1] \to \R^2$ be a Lipschitz curve such that $\gamma(0) = \gamma(1)$, $\gamma'\ne 0$ a.e. and $\gamma$ is injective on $[0,1)$. Then
	\begin{equation}
		\label{E-eq-interior}
		E = I_\gamma \mod \L^2
	\end{equation}
	holds if and only if
	\begin{equation}
		\label{essential-boundary-is-curve}
		\d^M E = \Gamma \mod \H^1,
	\end{equation}
	where $\Gamma = \gamma([0,1])$.
\end{lemma}

The above technical assertion is almost evident and can be easily deduced from the well-known facts of real analysis, so we move its proof  to Appendix~\ref{appendix-connectedness-lemma}.

The structure of the essential boundary of a simple set with finite measure was revealed in \cite[Theorem 7]{ACMM_2001}:

\begin{theorem}\label{boundary-of-simple-plane-sets}
	Let $E \subset \R^2$ be a set with finite and strictly positive measure and finite perimeter.
	Then $E$ is simple if and only if there exists a closed simple Lipschitz curve $\gamma\colon [0,1] \to \R^2$ such that \eqref{essential-boundary-is-curve} holds.
\end{theorem}

\begin{proof}
	\emph{Necessity} was proved in \cite[Theorem 7]{ACMM_2001}.
	
	\emph{Sufficiency.} By Lemma~\ref{Jordan-interior} we have $E = I_\gamma \mod \L^2$.
	And since $I_\gamma$ and $E_\gamma$ are open and connected, by Theorem~\ref{indecomposablity-criterion-for-open-sets} both of them are indecomposable.
	Hence $I_\gamma$ (and consequently $E$) is simple.
\end{proof}

If $d=2$ and the set $E$ is simple, then right-hand side of \eqref{Gauss-Green} can be described more precisely:

\begin{prop}\label{generalized-inner-normal-for-simple-set}
If $E \subset \R^2$ is a simple set, then there exists a closed simple Lipschitz curve ${\gamma \colon [0,1] \to \R^2}$ such that 
$\gamma'(s)\ne0$ a.e.,  
\begin{equation}
\d^M E = \gamma([0,1]) \quad \text{{\rm modulo} $\H^1$},
\label{essential-boundary-of-a-simple-set}
\end{equation}
and
\begin{equation}
\bm \nu(\gamma(s)) = \left(\frac{\gamma'(s)}{|\gamma'(s)|}\right)^\perp
\label{generalized normal and tangent vector}
\end{equation}
for a.e. $s\in[0,1],$ where $\bm \nu$ is the generalized inner normal to $E$ and $a^\perp = (-a_2, a_1)$ for any $a=(a_1, a_2).$
In particular,
\begin{equation*}
D \1_E = (\bm \mu_\gamma)^\perp.
\end{equation*}
\end{prop}

The proof of this result is contained in the proof of \cite[Proposition 6.1]{BonicattoGusev2021}, but we present it below in Appendix~\ref{appendix-connectedness-lemma} for completeness and a reader's convenience.
We will refer to the function $\gamma$ given by Proposition~\ref{generalized-inner-normal-for-simple-set}
as to the \emph{canonical parametrization of $\d^M E$} (note that $\gamma$ is not unique). For a generic set with finite perimeter Proposition \ref{generalized-inner-normal-for-simple-set}
can be extended as follows (cf. \cite[Theorem 4]{ACMM_2001}):

\begin{prop}\label{derivative of indicator of a set of finite perimeter}
Let $E \subset \R^2$ be a set of finite perimeter.
Then there exists at most countable family of closed simple Lipschitz curves $\gamma_i \colon [0,1] \to \R^2,$ such that
\begin{equation*}
D \1_E = \sum_{i} (\bm\mu_{\gamma_i})^\perp,
\qquad
|D \1_E| = \sum_{i} |\bm\mu_{\gamma_i}|,
\end{equation*}
and
\begin{equation*}
|\bm \mu_{\gamma_i}| \perp |\bm \mu_{\gamma_j}| = 0
\end{equation*}
whenever $i\ne j.$
\end{prop}
Again, the accurate proof of the last assertion is contained in Appendix~\ref{appendix-connectedness-lemma}.

\begin{defn}\label{generalized-connected-components}
Given a set $E \subset \R^2$ with finite perimeter,
let
\begin{equation*}
\cd^M E := \{\gamma_i\}
\end{equation*}
denote the family of the curves $\gamma_i$ given by Proposition~\ref{derivative of indicator of a set of finite perimeter}.
\end{defn}

We will call the curves $\gamma_i\in \cd^M E$ the \emph{generalized connected components of $\d^M E$}.
From Proposition~\ref{derivative of indicator of a set of finite perimeter} we immediately deduce the following:

\begin{cor}\label{structure-of-essential-boundary-2d}
	For any set $E \subset \R^2$ with finite perimeter we have
	\begin{equation}
		\d^M E = \bigcup_{\gamma \in \cd^M E} \gamma([0,1]) \mod \H^1.
	\end{equation}
	Furthermore, the sets $\gamma([0,1])$ {\rm(}for $\gamma \in \cd^M E${\rm)} are pairwise disjoint, up to $\H^1$-negligible subsets.
\end{cor}

\begin{proof}
	On one hand, we have $|D\1_E| = \H^1 \rest \d^M E$.
	On the other hand, by Propositions~\ref{derivative of indicator of a set of finite perimeter} and~\ref{generalized-inner-normal-for-simple-set}
	we have 
	\begin{equation}
		|D\1_E| = \sum_{\gamma \in \cd^M E} |\bm\mu_{\gamma}| = 
		\sum_{\gamma \in \cd^M E} |\bm\mu_{\gamma}|
		= \H^1 \rest \bigcup_{\gamma \in \cd^M E} \gamma([0,1]),
	\end{equation}
	since for all distinct $\gamma$ and $\theta$ from $\cd^M E$ by Proposition~\ref{derivative of indicator of a set of finite perimeter} we have $\H^1 \rest \gamma([0,1]) \perp \H^1 \rest\theta([0,1])$, and therefore $\gamma([0,1])\cap \theta([0,1]) = \emptyset \mod \H^1$.
\end{proof}

Let us state an immediate consequence of Theorem~\ref{boundary-of-simple-plane-sets}:

\begin{cor}\label{gcc-of-simple-set}
A set $E \subset \R^2$ with strictly positive finite measure and finite perimeter is simple if and only if $\cd^M E$ consists of a single curve.
\end{cor}

\subsection{Regular values of Sobolev functions on the plane}\label{sec:reg-val}
Let us collect some important properties of the level sets $\{f=t\}$ of a function $f\in W^{1,1}(\R^2)$, which hold for a.e. $t\in \R$.

\begin{defn}\label{def-reg}
Let $f\in W^{1,1}(\R^2)$.
A~number $t\in \R$ is called \emph{a~regular value\footnotemark{}}  of~$f$, if
\begin{enumerate}
 \item the set $\{f>t\}$ has finite perimeter;
 \item $\nabla f(x) \ne 0$ for $\H^1$-a.e. $x\in \{f=t\}$;
 \item $\d^M \{f>t\} = \{f=t\}$ modulo $\H^1$ and the equality
$
\bm \nu_{\{f>t\}} = \frac{\nabla f}{|\nabla f|}
$ holds $\H^1$-a.e. on $\{f=t\}$.
\end{enumerate}
Let $\reg f$ denote the set of all regular values of $f.$
For any $t\in \reg f$ the set
\begin{equation*}
\{f=t\}^* := \bigcup_{\gamma \in \cd^M \{f>t\}} \gamma([0,1])
\end{equation*}
will be called the \emph{essential level set of $f$}.
(Here $\cd^M \{f>t\}$ is the set of
the generalized connected components of $\d^M \{f>t\}$, see Definition~\ref{generalized-connected-components}.)
\end{defn}
\footnotetext{The standard definition of the term \emph{regular value}
(that the differential of $f$ is surjective at every preimage of $t$, i.e., in the sense of Morse-Sard) is not used in this text.}

By Corollary~\ref{structure-of-essential-boundary-2d} for any $t\in \reg f$ we have $\{f=t\} = \{f=t\}^* \mod \H^1$.
In particular, the set $\{f=t\}^*$ is empty if and only if $\H^1(\{f=t\})=0$;
in this case we call the regular value $t$ \emph{degenerate}.

\begin{prop}
\label{reg-vals}
If $f \in W^{1,1}(\R^2)$, then a.e. $t\in \R$ is a regular value of $f$.
\end{prop}

\begin{proof}
The item (1) from Definition~\ref{def-reg} holds for a.e. $t$ by the Coarea formula.
The items (2)--(3) hold by Proposition~\ref{constancy on reduced boundary}
and Proposition~\ref{gen-inner-normal}.
\end{proof}

\begin{rem}\label{non-deg-rv-of-monotone-fn}
	In particular, if $f\in W^{1,1}(\R^2)$ is monotone, then for any non-degenerate regular value $t$ the level set $\{f=t\}$ coincides with a closed simple Lipschitz curve (up to $\H^1$-negligible subsets).
	If, in addition, $f\in W^{1,2}(\R^2)$, then this property can be strengthened, as we will see in Proposition~\ref{almost-all-level-sets-are-essential}.
\end{rem}

Further we need also the following useful

\begin{prop}[{\bf on existence of a suitable parametrization}]\label{existence-of-good-parametrization}
If $f\in W^{1,1}(\R^2)$ is monotone, then for a.e. $y\in\reg f$ there exists an injective absolutely continuous parametrization $\gamma \colon [0,L_y) \to \R^2$
of the closed curve $\{f=y\}^*$ such that $\gamma'\ne 0$ a.e. and
\begin{equation}
\label{good-parametrization-f}
\gamma'(s) = \nabla^\perp f(\gamma(s)) \qquad\text{for a.e. } s\in[0,L_y).
\end{equation}
\end{prop}

In the case when $f$ is Lipschitz, such parametrizations are Lipschitz as well, cf. \cite[Lemma 2.11 (ii)]{ABC_2014}.
However, when $f$ is Sobolev, such parametrizations may fail to be Lipschitz since $\nabla f$ may be unbounded. Proposition~\ref{existence-of-good-parametrization} can be easily deduced from the definition of regular values (see also the previous Remark\,) and some well-known classical facts of one-dimensional Real Analysis, see Appendix~\ref{appendix-connectedness-lemma} for the~accurate proof with details.

\section{Sobolev functions with parallel gradients}\label{sec:parallel-gradients}

In contrast to the previous section (where some standard facts from geometric analysis were mostly collected, except for Theorem~\ref{thm:BT}), this part is more interesting in a sense: here we start to develop some basic tools, which allow us to establish further the main results of the paper, namely, Theorem~\ref{chain-rule} in the next Section~\ref{sec:crp} and Theorem~\ref{uniqueness-criterea} in Section~\ref{sec:uniq}  (after the further refine development of the present technique in Sections~\ref{section-fine-properties}--\ref{sec:wsp}). 

Let us start from the proof that if $f,g\in W^{1,1}(\R^2) \cap W^{1,2}(\R^2)$ and $\nabla f \parallel \nabla g$,
then the function $f$ is constant along every generalized connected component of $\d^M \{g>t\}$
for a.e. $t\in\R$ (in the Lipschitz setting some close results were obtained in \cite[\S 4.2 and Lemma 5.1]{BBG2016}): 

\begin{lemma}\label{lemma-constancy-along-connected-components}
	Suppose that for $f, g \in W^{1,1}(\R^2)\cap W^{1,2}(\R^2)$ we have $\nabla f \parallel \nabla g$ a.e.
	Then for a.e. regular value $t$ of $g$ and for every $\gamma\in \cd^M\{g>t\}$ there exists a constant $c=c(t, \gamma)\in \R$
	such that $f(x) = c$ for all $x\in \gamma([0,1])$ that are Lebesgue points of $f$ {\rm(}in particular, for $\H^1$-a.e. $x\in \gamma([0,1])${\rm)}.
\end{lemma}

\begin{proof}
	For $\eps>0$ let $f_\eps = f * \omega_\eps$ denote the convolution of $f$ with the standard mollification kernel $\omega_\eps.$
	If $x\in \R^2$ is a Lebesgue point of $f$, then $f_\eps(x) \to f(x)$ as $\eps \to 0$.
	Since a.e. $x\in \R^2$ is a Lebesgue point of $f$, by the Coarea formula
	for a.e. $t\in \R$ we have that $\H^1$-a.e. $x\in \{g=t\}$ is a Lebesgue point of $f$.
	Hence for a.e. regular value $t$ of $g$ and for all $\gamma \in \cd^M \{g>t\}$ we have
	\begin{equation}
		f_\eps\circ \gamma \to f\circ \gamma\qquad
		\mbox{for }\H^1\mbox{-a.e. } x\in \gamma[0,1]. \label{convergence of mollifications on curve}
	\end{equation}
	as $\eps \to 0$.
	Then by lower semicontinuity of total variation (see e.g. \cite[Remark 3.5]{AFP})
	\begin{equation*}
		\TV(f\circ \gamma) \le \liminf_{\eps\to 0} \TV(f_\eps \circ \gamma).
	\end{equation*}
	Furthermore, if $x_0 = \gamma(s_0)$ and $x = \gamma(s)$ are Lebesgue points of $f$,
	and if $f(x_0) = c$, then
	\begin{equation*}
		|f(x) - c| = \liminf_{\eps\to 0} |f_\eps(\gamma(s)) - f_\eps(\gamma(s_0))| \le \liminf_{\eps\to 0} \TV(f_\eps \circ \gamma).
	\end{equation*}
	Hence
	\begin{equation*}
		\sup_{x} |f(x) - c| \le \liminf_{\eps\to 0} \TV(f_\eps \circ \gamma),
	\end{equation*}
	where the supremum is taken over all Lebesgue points $x$ of $f$ such that $x\in \gamma([0,1])$.
	
	Since
	\begin{equation*}
		\sum_{\gamma \in \cd^M \{g>t\}} \liminf_{\eps\to 0} \TV(f_\eps \circ \gamma)
		\le
		\liminf_{\eps\to 0} \sum_{\gamma \in \cd^M \{g>t\}} \TV(f_\eps \circ \gamma) =: F(t),
	\end{equation*}
	it is sufficient to prove that $F(t) = 0$ for a.e. $t$.
	
	The function $F$ is measurable as the lower limit of the sequence of functions
	\begin{equation*}
		F_\eps(t) = \sum_{\gamma \in \cd^M \{g>t\}} \TV(f_\eps \circ \gamma),
	\end{equation*}
	which are measurable by the Coarea formula.
	Then by Fatou lemma,
	\begin{equation*}
		I:= \int F(t) \, dt
		= \int \liminf_{\eps\to 0} F_\eps(t) \, dt
		\le \liminf_{\eps\to 0} \int F_\eps(t) \, dt.
	\end{equation*}
	Since $f_\eps$ is smooth, for any $\gamma \in \cd^M \{g>t\}$ we have
	\begin{equation*}
		\TV(f_\eps \circ \gamma) = \int_0^1 \left|\frac{d}{ds} f_\eps(\gamma(s))\right|\, ds
		=\int_0^1 |\nabla f_\eps(\gamma(s)) \cdot \gamma'(s)| \, ds =: J_\gamma.
	\end{equation*}
	By \eqref{generalized inner normal for Sobolev function}, \eqref{generalized normal and tangent vector},  and by the area formula we have
	\begin{equation*}
		J_\gamma = \int_0^1 \left|\nabla f_\eps(\gamma(s)) \cdot \frac{\nabla^\perp g(\gamma(s))}{|\nabla g (\gamma(s))|}\right| \cdot |\gamma'(s)| \, ds = \int \frac{|\nabla f_\eps \cdot \nabla^\perp g|}{|\nabla g|} d\H^1 \rest \gamma([0,1]).
	\end{equation*}
	Hence
	\begin{equation*}
		F_\eps(t)
		= \int \frac{|\nabla f_\eps \cdot \nabla^\perp g|}{|\nabla g|} \, d\H^1\rest \{g=t\}.
	\end{equation*}
	
	Therefore, again using the Coarea formula for $g$, we get
	\begin{equation*}
		\int F_\eps(t) \, dt \le \int \int \frac{|\nabla f_\eps \cdot \nabla^\perp g|}{|\nabla g|} \, d\H^1\rest \{g=t\} \, dt
		= \int_{\R^2} |\nabla f_\eps \cdot \nabla^\perp g| \, dx =: I_\eps.
	\end{equation*}
	Since $\nabla f_\eps \to \nabla f$ in $L^2$ as $\eps \to 0$, we have
	\begin{equation*}
		I_\eps \to \int_{\R^2} \nabla f \cdot \nabla^\perp g \, dx =0.
	\end{equation*}
	Thus $I=0$.
	Hence the equality $f(x)=c$ holds for $\H^1$-a.e. $x\in \gamma([0,1])$.
\end{proof}

Let us mention that Lemma~\ref{lemma-constancy-along-connected-components} will be considerably strengthened in Section~\ref{section-fine-properties}. Below we will use the following convenient agreement: 

\begin{defn}
	The constant $c(t,\gamma)$ in Lemma~\ref{lemma-constancy-along-connected-components} is called the \emph{essential value of $f$ on $\gamma$}.
\end{defn}

The condition $\nabla f \parallel \nabla g$ holds in particular when $|Df| \perp |Dg|$.
In this case we can also say that for almost all regular values $t$ of $g$
the essential values of $f$ on each $\gamma \in \cd^M \{g>t\}$ are not regular:

\begin{lemma}\label{lemma-mutual-weak-Sard-property}
	Suppose that for $f, g \in W^{1,1}(\R^2)\cap W^{1,2}(\R^2)$ we have $|Df| \perp |Dg|$.
	Then there exists a negligible set 
	$N \subset \reg g$
	such that for all $t\in \reg g \setminus N$ and 
    for all $\gamma \in \cd^M \{g>t\}$ the essential value of $f$ on $\gamma$
    does not belong to $\reg f$.
In particular,
	\begin{equation}
		\label{almost-no-regular-values}
		f(x) \notin \reg f \quad \text{for } \H^1\text{-a.e. } x\in \{g=t\}.
	\end{equation}
	We also have
	\begin{equation*}
		f_\# (\L^2 \rest \{\nabla g \ne 0\}) \perp \L^1
	\end{equation*}
	and
	\begin{equation}\label{nut-2}
		f_\# |Dg| \perp \L^1.
	\end{equation}
\end{lemma}

\begin{proof}
	By Lemma~\ref{lemma-constancy-along-connected-components}, there exists
	a negligible set $N_g \subset \R$ such that $N_g^c \subset \reg g$ and for all $t\in N_g^c$ and 
	for any $\gamma \in\cd^M \{g>t\}$ there exists a constant $c=c(t,\gamma)$
	such that $f(x) = c(t,\gamma)$ for $\H^1$-a.e. $x\in \gamma([0,1])$.
	
	Let $A$ denote the set of all $t\in N_g^c$ for which there exists $\gamma \in\cd^M \{g>t\}$
	such that $c(t,\gamma) \in \reg f$.
	It is easy to see that $A=A'$, where $A'$ is  set of all $t\in N_g^c$ such that $\H^1(\{g=t\}\cap f^{-1}(\reg f))>0$.
	Indeed, clearly $A \subset A'$. 
	On the other hand, up to $\H^1$-negligible subsets $\{g=t\}$ is the (at most countable) union of the sets $\gamma([0,1])$ for all $\gamma \in \cd^M\{g>t\}$. 
	So, if $s \in A'$, then there exists $\gamma \in \cd^M \{g>t\}$ such that $\H^1(\gamma([0,1])\cap f^{-1}(\reg f))>0$. But $f(x)=c(t, \gamma)$ holds for $\H^1$-a.e. $x\in \gamma([0,1])$, and in particular for $\H^1$-a.e. $x\in \gamma([0,1])\cap f^{-1}(\reg f)$.
	Then there exists $x\in \gamma([0,1])$ such that $f(x) = c(t, \gamma)$ and $f(x) \in \reg f$.
	Hence $A' \subset A$.
	
	By Definition~\ref{def-reg} of regular values, we have $\nabla f(x) \ne 0$ for $\H^1$-a.e. $x\in \{f=c\}$, hence for $\H^1$-a.e. $x\in \gamma([0,1])$.
	Then 
	\begin{equation}
		\H^1(\{\nabla f\ne 0\} \cap \{g=t\}) \ge \H^1(\{\nabla f \ne 0\} \cap \gamma([0,1]))=\H^1(\gamma([0,1]))>0.
	\end{equation}
	Hence, the set $A$ is contained in the set
	\begin{equation*}
		B:= \{t\in \reg g \;|\; \H^1(\{\nabla f\ne 0\} \cap \{g=t\})>0\}.
	\end{equation*}
	By the Coarea formula,
	\begin{equation*}
		\int \H^1(\{\nabla f \ne 0\} \cap \{g=t\}) \, dt = \int_{\{\nabla f\ne 0\}} |\nabla g| \, dx = 0,
	\end{equation*}
	since $|Df| \perp |Dg|.$
	Hence, $\L^1(B) = 0$ and, consequently, $\L^1(A) = 0.$
	
	Let $N:=N_g \cup A.$ 
	Then for any $t\in \reg g \setminus N$ on the set $\{g=t\}^*$ the function $f$ takes only non-regular values, up to $\H^1$-negligible subsets, i.e., the property~\eqref{almost-no-regular-values} is established.
	
	We need to check other two properties. Let $R \subset \reg f$ be a Borel set such that $\L^1(R^c) = 0$ and let $M = f^{-1}(R) \cap \{\nabla g \ne 0\}$. 
	By \eqref{almost-no-regular-values} for a.e. $t\in \R$ we have
	\begin{equation}
		\label{H1Mcapgs}
		\H^1(M\cap \{g=t\}) = 0.
	\end{equation}
	Then by the Coarea formula
	\begin{equation}
		\int_{M} |\nabla g| \, dx = \int_\R \H^1(M\cap \{g=t\}) = 0,
	\end{equation}
	hence $\L^2(M) = 0$. Then
	\begin{equation}
		f_\# (\L^2 \rest \{\nabla g \ne 0\})
		= f_\# (\L^2 \rest \{\nabla g \ne 0\} \setminus M) \perp \L^1,
	\end{equation}
	since $f(\{\nabla g \ne 0\} \setminus M) \subset R^c$ and $\L^1(R^c)=0$.
\end{proof}

Note that \eqref{almost-no-regular-values} implies that for a.e. $s\in \reg f$ and all $t\in \reg g$ we have
\begin{equation}
   	\label{almost-disjoint-levelsets}
   	\{f=s\} \cap \{g=t\} = \emptyset \mod \H^1,
\end{equation}
which has a clear geometric interpretation: almost all level sets of $f$ and $g$ are almost disjoint.
However \eqref{almost-no-regular-values} is slightly stronger than \eqref{almost-disjoint-levelsets}, in particular \eqref{H1Mcapgs} does not follow immediately from~\eqref{almost-disjoint-levelsets} (due to uncountable union of the level sets in the definition of $M$).

Let us mention that we will strengthen Lemma~\ref{lemma-mutual-weak-Sard-property} in Section~\ref{sec:wsp} (see Proposition~\ref{mutual weak Sard property}).

\section{Chain rule for the divergence operator and Nelson's conjecture}\label{sec:crp}

Throughout this section we assume that $f\in W^{1,2}(\R^2)$ and $\ve = \nabla^\perp f$
is compactly supported. (Consequently, $f\in W^{1,1}(\R^2)$ as well.)

\begin{prop}\label{div-decomp}
	Let $\{f_i\}$ be the decomposition of the stream function $f$ into monotone parts.
	Then $\rho\in L^\infty(\R^2)$ solves $\dive(\rho \ve) = 0$
	if and only if for any $i$
	\begin{equation}
		\dive (\rho \nabla^\perp f_i) = 0.
		\label{comp}
	\end{equation}
\end{prop}

\begin{proof}
	Sufficiency follows from linearity of the operator $\ve \mapsto \dive(\rho \ve)$ and passage to the limit. It remains to prove necessity. Let us fix $i$ and set $h:= f_i$ and $g:=f-h$.
	Since $\ve \in L^2(\R^2)$, we can write the distributional formulation of \eqref{div(rho v)=0}
	using compactly supported test functions from $W^{1,2}(\R^2).$
	In particular, we can consider test functions of the form
    \begin{equation}
		\phi(x) = \fhi(x) \psi(h(x)),
		\label{test-form-22}
\end{equation}	
	where $\fhi$ and $\psi$ are some compactly supported smooth functions. Since $\nabla^\perp f \cdot \nabla h = 0$ a.e., for any such test function $\phi$ we have
	\begin{equation}
		\int \rho \nabla^\perp h \cdot (\nabla \fhi) \psi(h)  \, dx
		+ \int \rho \nabla^\perp g \cdot (\nabla \fhi) \psi(h) \, dx = 0.
		\label{disint-div-eq}
	\end{equation}

	By the Coarea formula the first term in \eqref{disint-div-eq} can be written as
	\begin{equation}
		\int \rho \nabla^\perp h \cdot (\nabla \fhi) \psi(h) \, dx
		= \int a(t) \psi(t) \, dt,
	\end{equation}
	where
	\begin{equation*}
		a(t):=\int \rho \bm \tau \cdot \nabla \fhi
		\, d\H^1 \rest \{h = t\},\qquad \ \bm \tau:=\frac{\nabla^\perp h}{|\nabla h|}
	\end{equation*}
	Let us consider the second term in \eqref{disint-div-eq}.
	Let $\mu:= \rho \nabla^\perp g \cdot \nabla \fhi \L^2.$
	By the change of variables formula
	\begin{equation*}
		\int \rho \nabla^\perp g \cdot (\nabla \fhi) \psi(h) \, dx
		= \int_{\R^2} \psi(h) \, d\mu
		= \int_\R \psi(t) d (h_\# \mu)(t).
	\end{equation*}
	Therefore, \eqref{disint-div-eq} takes the form
	\begin{equation}
		\int a(t) \psi(t) \, dt + \int \psi(t) d\, (h_\# \mu)(t) = 0.
		\label{tmargin}
	\end{equation}
	Passing to the limit we note that \eqref{tmargin} holds for any compactly supported continuous  function $\psi.$ Hence by Riesz theorem
	\begin{equation}
		a \L^1 + h_\# \mu = 0.
		\label{Riesz-corollary}
	\end{equation}
	Since $|Dh| \perp |Dg|$, 
	by  Lemma~\ref{lemma-mutual-weak-Sard-property} 
	we have $h_\# |Dg| \perp \L^1.$
	By definition of $\mu$ we also have $|\mu| \ll |D g|.$
	Hence, $h_\# |\mu| \perp \L^1$ and, consequently, $h_\# \mu \perp \L^1.$
	In particular,  the measures $a \L^1$ and $h_\# \mu$ are mutually singular. Therefore, \eqref{Riesz-corollary}
	holds if and only if $a \L^1 = 0$ and $h_\# \mu = 0.$
	Hence, $a(t)=0$ for a.e. $t\in \R.$
	
	The above argument can be repeated for any test function $\fhi$ from some countable dense subset $D$ of $C_0^\infty(\R^2)$.
	Hence, there exists a negligible set $N \subset \R$ (which does not depend on $\fhi$)
	such that for all $t\in \R \setminus N$ and for all $\fhi \in D$ the equality $a(t)=0$ holds.
	Passing to the limit, for any such $t$ we also get the equality $a(t)=0$ for all $\fhi \in C_0^\infty(\R^2).$
	Then, from the identity $\int a(t) \, dt = 0$, by the Coarea formula we obtain \eqref{comp}.
\end{proof}

\begin{prop}\label{div-monotone-no-rhs}
	Suppose that $f\in W^{1,2}(\R^2)$ is monotone (and compactly supported).
	Then $\rho\in L^\infty(\R^2)$ solves
	\begin{equation}
		\div (\rho \nabla^\perp f) = 0
		\label{steady-contin}
	\end{equation}
	if and only if for a.e. $t$ the function $\rho$ is constant $\H^1$-a.e. on $\{f=t\}$.
\end{prop}

\begin{proof} We follow ideas from the previous proof. Again, we consider the distributional formulation of~\eqref{steady-contin} with test functions of the form~\eqref{test-form-22}, 
that means 
\begin{equation}
		\int \rho \nabla^\perp f \cdot (\nabla \fhi) \psi(f)  \, dx=0
		\label{disint-div-eq-new}
	\end{equation}
	for any compactly supported smooth functions $\fhi, \psi$.  (Note, that here we essentially use the assumption~$f\in W^{1,2}(\R^2)$.) 
	As we  already saw, by Coarea formula for any such pair $\fhi, \psi$ we have
	\begin{equation*}
		\int \rho \nabla^\perp f \cdot (\nabla \fhi) \psi(f) \, dx
		= \int \int \rho \bm \tau \cdot \nabla \fhi \, d\H^1 \rest \{f = t\} \psi(t) \, dt.
	\end{equation*}
	Therefore,  again by Riesz theorem \eqref{disint-div-eq-new} holds iff 
	\begin{equation*}
		\div (\rho \bm \tau \H^1 \rest \{f=t\}) = 0
	\end{equation*}
	for a.e. $t\in\R$.  (Strictly saying, here we need to consider a~countable dense family of functions $\fhi$.) By Proposition \eqref{steady continuity equation along curve}, the above equation holds if and only if $\rho$ is constant $\H^1$-a.e. on $\{f=t\}$.
\end{proof}

Since any compactly supported $\ve \in L^p(\R^2)$ with $p\ge 2$ belongs to $L^2(\R^2)$,
as a corollary of the two Propositions above we immediately obtain the second part of Theorem~\ref{chain-rule}.

To prove the first part of Theorem~\ref{chain-rule},
one can consider the following minor modification of the vector field constructed by E. Nelson (cf. \cite{Aizenman_1978}):

\begin{prop}\label{ex-n}
	There exists a compactly supported vector field $\ve \colon \R^2\to \R^2$
	such that $\dive \ve = 0$, furthermore, $\ve \in L^p(\R^2)$ for any $p\in[1,2)$, 
	and $\ve$ does not have the chain rule property.
\end{prop}

\begin{figure}[h]
	\begin{subfigure}[h]{0.37\linewidth}
	\begin{tikzpicture}[scale=2,every node/.style={scale=1}]
		\tikzmath{
			\sqrtthree=sqrt(3);
		}
		\fill[color=green!20] (0,0) -- (1/\sqrtthree,1) -- (0,2);
		\fill[color=green!20, pattern color=green!80!black!30!white, pattern=north east lines] (0,0) -- (1/\sqrtthree,1) -- (0,2);
		\fill[color=red!20] (0,0) -- (\sqrtthree,1) -- (0,2) -- (1/\sqrtthree,1);
		\fill[color=green!20, pattern color=red!80!black!30!white, pattern=north west lines] (0,0) -- (\sqrtthree,1) -- (0,2) -- (1/\sqrtthree,1);
		
		\draw[->, very thin] (-0.2,0) -- (2,0) node[above] {$x$};
		\draw[->, very thin] (0,0) -- (0,2.4) node[right] {$y$};
		\node[below] at (0,0) {$0$};
		\node[below] at (\sqrtthree,0) {$\sqrt{3}$};
		\node[above left] at (0,1) {$1$};
		\node[above left] at (0,2) {$2$};
		
		\draw[thin, densely dotted] (\sqrtthree,0) -- (\sqrtthree,1);
		\node[below] at (1/\sqrtthree,0) {$1/\sqrt{3}$};
		\draw[thin, densely dotted] (1/\sqrtthree,0) -- (1/\sqrtthree,1);
		\draw[thin, densely dotted] (0,1) -- (\sqrtthree,1);
		
		\begin{scope}[decoration={
				markings,
				mark=at position 0.5 with {\arrow{latex}}}
			]
			\tikzmath{
				let \muls{one} = 0.2; let \muls{two} = 0.45; let \muls{three} = 0.8;
			}
			\foreach \x in {one,two,three}{
				\tikzmath{
					\tgc=\muls{\x}/sqrt(3);
					\y=\muls{\x}*100;
					\xtgc=(sqrt(3) - \tgc)/(1+sqrt(3)*\tgc);
				}
				\draw[postaction={decorate}] (0,2)--(\tgc,1);
				\draw[postaction={decorate}] (0,2)--(\xtgc,1);
				\draw[postaction={decorate}] (\tgc,1)--(0,0);
				\draw[postaction={decorate}] (\xtgc,1)--(0,0);
			}
		\end{scope}
	\end{tikzpicture}
	\caption{The vector field $\ve$ and the density $\rho$. The sets $\{\rho=+1\}$ and $\{\rho=-1\}$ are hatched with green and red respectively.\label{stream-function1}}
	\end{subfigure}
	\qquad
	\begin{subfigure}[h]{0.45\linewidth}
		\begin{tikzpicture}[scale=2,every node/.style={scale=1}]
			\tikzmath{
				\sqrtthree=sqrt(3);
			}
			\draw[->, very thin] (-1.05*\sqrtthree,0) -- (1.05*\sqrtthree,0) node[above] {$x$};
			\draw[->, very thin] (0,0) -- (0,2.4) node[right] {$y$};
			\node[below] at (0,0) {$0$};
			\node[above left] at (0,1) {$1$};
			\node[above left] at (0,2) {$2$};
			\draw[thin, densely dotted] (-\sqrtthree,1) -- (\sqrtthree,1);
			\node[below] at (1/\sqrtthree,0) {\vphantom{$1/\sqrt{3}$}};
		\begin{scope}
			\tikzmath{
				let \cols{one} = red; let \cols{two} = green!40!black; let \cols{three} = blue;
				let \muls{one} = 0.9; let \muls{two} = 0.6; let \muls{three} = 0.3;
				let \alpha{one} = 30; let \alpha{two} = 60; let \alpha{three} = 90;
			}
			\foreach \idx in {one,two,three}{
				\tikzmath{
					\x=\muls{\idx}*sqrt(3);
					let \col=\cols{\idx}!\alpha{\idx}!white;
				}
				\draw[\col] (0,2)--(\x,1)--(0,0)--(-\x,1)--cycle;
			}
		\end{scope}
		\end{tikzpicture}
		\caption{Stream function $f$ of $\ve$. The level sets corresponding to different values of $f$ are depicted with different colors and intensities. \label{stream-function2}}
	\end{subfigure}
	\caption{Nelson's example (a minor modification).}
	\label{fig:no-chain-rule}
\end{figure}

\begin{proof}
	Let us construct the stream function $f\colon \R^2 \to \R$ as follows.
	If $x\ge0$ and $0\le y\le 1$, then we set
	\begin{equation}\label{Nelson-stream-function}
		f(x,y):=\begin{cases}
			\frac{\pi}{3}-\arctg \frac{x}{y}, & y\ge x/\sqrt{3}>0, \\
			0, & y<x/\sqrt{3}.
		\end{cases}
	\end{equation}
	For $x>0$ and $y\in (1,2)$ we extend the function $f$ by symmetry with respect to $y=1$:
	$f(x,y) := f(x,2-y)$. We also extend $f$ by symmetry with respect to $x=0,$
	i.e., $f(x,y):=f(-x,y)$ if $x<0$ and $y\in(0,2).$
	Ultimately, we set $f(x,y)=0$ if $y\notin(0,2)$.
	
	The vector field $\ve$ is then defined as $\ve = \nabla^\perp f$.
	If $x>0$ and $x<y<1$, then
	\begin{equation*}
		\ve(x,y) =
		\begin{bmatrix}
			-\d_y f\\
			\d_x f
		\end{bmatrix}
		=
		-
		\begin{bmatrix}
			\frac{x}{x^2+y^2} \\
			\frac{y}{x^2+y^2}
		\end{bmatrix},
		\qquad\ |\ve| = \frac1{\sqrt{{x^2} + {y^2}}}.
	\end{equation*}
	Therefore $\ve \in L^p(\R^2)$ (with $p>1$) iff $p<2$. 
	
	Next for any $0\le\alpha<\beta<+\infty$ we define $\rho_{\alpha, \beta}(x,y)$ as follows: if $x>0$ and $0<y<1$, then we set
	\begin{equation*}
		\rho_{\alpha,\beta}(x,y) :=
		\begin{cases}
			1, &\alpha y < x < \beta y, \\
			0, & \text{otherwise}.
		\end{cases}
	\end{equation*}
	Further, for $x>0$ and $1<y<2$ we extend $\rho_{\alpha,\beta}$ by symmetry  with respect to $y=1$, namely 
	$\rho_{\alpha,\beta}(x,y):=\rho_{\alpha,\beta}(x,2-y)$. Finally, we put $\rho_{\alpha,\beta}(x,y)=0$ for $x<0$ or $y\notin (0,2)$. 
	One can easily check that 
	\begin{equation*}
		\dive(\rho_{\alpha, \beta} \ve) = (\arctg \beta - \arctg \alpha) (\delta_{(0,2)} - \delta_{(0,0)}),
	\end{equation*}
	where $\delta_{(\xi, \eta)}$ is the Dirac delta at $(\xi, \eta)$.
	
	Now consider $\rho:= \rho_{0, 1/\sqrt{3}}-\rho_{1/\sqrt{3},\sqrt{3}}$. By linearity of the divergence operator
	we have $\dive (\rho \ve) = 0$, but $\rho^2 = \rho_{0,\sqrt{3}}$, and, consequently, $\dive(\rho^2 \ve) = \frac{\pi}{3}(\delta_{(0,2)} - \delta_{(0,0)}) \ne 0.$
\end{proof}

\begin{rem}
	By Propositions \ref{removable-singularity} and \ref{constancy-on-connected-sets},
	the situation described in the example above (see Figure~\ref{fig:no-chain-rule})
	is not possible for $f\in W^{1,2}(\R^2)$ (i.e., when $\ve \in L^2(\R^2)$).
\end{rem}

We conclude this section with the proof of Proposition~\ref{main-I}.

\begin{proof}[Proof of Proposition~\ref{main-I}]\label{Nelson-conj-2d}
	By \cite[section 4]{ABC_2013_LipSard}, there exists a Lipschitz function $f\colon \R^2 \to \R$
	not having the weak Sard property. 
	One can choose a smooth cutoff function $\zeta$ such that $\zeta \cdot f$ still does not have the weak Sard property (it is sufficient that $\zeta=1$ on a sufficiently large ball).
	Replacing $f$ with $\zeta \cdot f$, if necessary, we can achieve that $\nabla f$ is compactly supported.
	Then $\ve := \nabla^\perp f$ is divergence-free and has compact support. 
	By \cite[Theorem 4.7]{ABC_2014}, since $f$ does not have the weak Sard property, there exists a nontrivial
	bounded weak solution $\rho$ to \eqref{continuity-equation}--\eqref{continuity-equation-initial-condition}
	with zero initial condition. Since $\ve$ is compactly supported and $\rho|_{t=0}\equiv0$, we easily obtain that $\rho(t, \cdot)$
	is uniformly compactly supported for a.e. $t$ as well, and, consequently, $t\mapsto \|\rho(t, \cdot)\|_{L^2(\R^2)}$
	belongs to $L^\infty$. Then by \cite[Theorem 1.1]{Panov2015} the operator $\mathsf{A}(\rho) = \ve \cdot \nabla \rho$ is not essentially skew adjoint.
\end{proof}

\section{New fine properties of level sets}
\label{section-fine-properties}

We have already seen that for any Sobolev function $f\in W^{1,1}(\R^2)$
for a.e. $t\in \reg f$ the level sets $\{f=t\}$ coincide with the essential level sets $\{f=t\}^*$ up to $\H^1$-negligible subsets. 
In this section we study their relationship more carefully.

\begin{theorem}\label{levelset-subset-of-essential-levelset}
Suppose that $f\in W^{1,1}(\R^2)$ is monotone.
Then for a.e. regular value $t$ of $f$ the pointwise inclusion 
\begin{equation}\label{f-eq-t-subset-f-eq-t-star}
\{f=t\} \subset \{f=t\}^*
\end{equation}
holds.
\end{theorem}

\begin{proof}
Without loss of generality we may assume that $f$ is non-negative.
We claim that for a.e. regular value $t>0$ of $f$ the following properties hold:
\begin{enumerate}
 \item[(i)] $t$ is a Lebesgue point of the function $t\mapsto \H^1(\{f=t\})$.
 \item[(ii)] the $\L^2$-measure of $\{f=t\}$ is zero;
 \item[(iii)] the function $t\mapsto \L^2(\{f>t\})$ is continuous at $t$.
\end{enumerate}
Indeed, (i) follows immediately from the integrability of $t\mapsto \H^1(\{f=t\})$
(which we have by the Coarea formula).
Since the sets $\{f=t\}$ are disjoint for different values of $t$,
the set $\{t\in \R : \L^2(\{f=t\})>0\}$ is at most countable, so the property (ii) follows.
Ultimately, the property (iii) follows from the property (ii) by
continuity of Lebesgue measure and monotonicity of $t\mapsto \L^2(\{f>t\})$.

Take $b:= \inf\{t\le+\infty : f\le t \text{ a.e.}\}$.
By assumption that $f$ is precisely represented, $f\le b$ everywhere in $\R^2$. Hence,  it is sufficient to study the regular values $t\in (0,b)$ such that (i)--(iii) hold.
For any such value $t$ by the definition of $b$ we have $\L^2(\{f>t\})>0$ and $\L^2(\{f<t\})>0$ (since $f\in L^1(\R^2)$).

Let $\gamma_t \colon [0,1] \to \R^2$
denote the canonical parametrization of $\d^M \{f>t\}$. We can assume, without loss of generality, that $|\gamma_t'(s)| = \H^1(\{f=t\})$ for a.e. $s\in [0,1]$ (i.e., up to a scaling factor this is a natural parametrization of the curve).

By property (i) there exists a sequence $\{t_n^+\}_{n\in \N}$ of regular values of $f$ such that $t_n^+ \to t +0$
and $\H^1(\{f=t_n^+\}) \to \H^1(\{f=t\})$ as $n\to \infty$.
By Arzela--Ascoli theorem, passing to a subsequence if necessary, we may assume that $\gamma_{t_n^+}$
converges uniformly to some continuous function $\gamma\colon [0,1]\to \R^2$.
Let $\Gamma:= \gamma([0,1]).$

Consider the vector measures $\bm \mu_n := D \1_{\{f>t_n^+\}}$.
Since $|\bm \mu_n|(\R^2) = \H^1(\{f=t_n^+\})$ are uniformly bounded,
after passing to a subsequence we may assume that $|\bm \mu_n|$ converge weak* to some measure $\lam$ as $n\to \infty$.
By construction, $\lam$ is concentrated on $\Gamma$.
By lower semicontinuity of total variation
\begin{equation}
\label{tv-est}
\lam(\R^2) \le \H^1(\{f=t\}).
\end{equation}
On the other hand, by property (iii) the functions $\1_{\{f>t_n^+\}}$ converge in $L^1$ to $\1_{\{f>t\}}$ as $n\to \infty$.
Therefore, $\bm \mu_n$ converge weak* to $\bm \mu:= D \1_{\{f>t\}}$ as $n\to\infty$.
Then $|\bm \mu| \le \lam$ (see e.g. \cite[Proposition 1.62]{AFP}), hence
$0\le (\lam - |\bm \mu|)(\R^2) = \lam(\R^2) - \H^1(\{f=t\}) \le 0$ by \eqref{tv-est}.
Therefore, $\lam = |\bm \mu|=\H^1 \rest  \{f=t\}$. Since $\lam$ is concentrated on $\Gamma$, we have thus proved that
$\{f=t\} = \Gamma$ modulo $\H^1$. Then $\Gamma = \gamma_t([0,1])$.
In particular, the curves $\{f=t_n^+\}^*$
converge to the curve $\{f=t\}^*$ in the Hausdorff metric.

In the same way one can construct a sequence $\{t_n^-\}$ of regular values of $f$ such that $t_n^- \to t-0$ and
the curves $\{f=t_n^-\}^*$ converge to $\{f=t\}^*$ in the Hausdorff metric.

Suppose that $x\in \{f=t\} \setminus \gamma_t([0,1])$.
Then there exist $r>0$ and $N\in\N$ such that $B_r(x) \cap \{f=t_n^\pm\}^* = \emptyset$ for all $n>N$.
Let us fix such $n$.
Since $D\1_{\{f>t_n^\pm\}}$ is concentrated on $\{f=t_n^{\pm}\}^*$,
the function $\1_{\{f>t_n^\pm\}}$ is constant a.e. on $B_r(x)$, hence
either $B_r(x) \subset \{f>t_n^{\pm}\}$, or $B_r(x) \subset \{f\le t_n^\pm\}$ (modulo $\L^2$).
Because of equality $f(x)=t$ and by the assumption that $f$ is precisely represented, the inclusion $B_r(x) \subset \{f>t_n^+\}$ (modulo $\L^2$) is not possible.
Hence, $B_r(x) \subset \{f\le t_n^+\}$ and similarly $B_r(x) \subset \{f\ge t_n^-\}$ (modulo $\L^2$).
Therefore, for a.e. $y\in B_r(x)$ and for all $n>N$
we have $t_n^- \le  f(y) \le t_n^+$. Passing to the limit as $n\to \infty$, we get $f(y)=t$, hence $B_r(x) \subset \{f=t\}$ modulo $\L^2$.
This contradicts (ii).
\end{proof}

\begin{rem}
	The monotonicity assumption is essential for the inclusion \eqref{f-eq-t-subset-f-eq-t-star}. 
	If $f$ is not monotone, then even the weaker inclusion
	\begin{equation}
		\{f=t\}\setminus S_f  \subset \{f=t\}^* \label{essential-level-set-vs-standard-one-weak}
	\end{equation}
	may fail for some set of values of $t$ with positive measure (even if $f$ is Lipschitz), see Remark~\ref{importance-of-monotonicity} below.
	(Here $S_f$ is the approximate discontinuity set of $f$,  see the comments after~\eqref{app-lim}.)
\end{rem}

\begin{rem}[Nelson's example revisited]\label{Nelson-reprise}
The inclusion \eqref{f-eq-t-subset-f-eq-t-star} may be strict for all values of $t$ from some interval, even if the function $f$ is monotone and has $\nabla f \in L^p(\R^2)$ for all $p\in [1,2)$.
In particular, this happens for the stream function~$f$ from the Nelson's example, given by \eqref{Nelson-stream-function}:
for all $t \in (-\pi/3, \pi/3)$ the level sets $\{f=t\}$ are contained in the stripe $0<y<2$,
while the essential level sets $\{f=t\}^*$ always include the points $(0,0)$ and $(0,2)$, see Figure~\ref{stream-function2}. 
\end{rem}

As we will see, the phenomenon described in Remark~\ref{Nelson-reprise} does not occur for $\nabla f\in L^{p}_\loc$ with $p\ge 2$.
But first let us prove some auxiliary results.

\begin{lemma}\label{removable-singularity}
	Let $U\subset \R^2$ be an open set and $f\in W^{1,2}_\loc(U)$.
	Suppose that $x\in U$, $t\in \R$, and for any~$\eps > 0$ there exists $\delta > 0$ such that $B_\delta(x) \subset U$
	and for a.e. $r\in (0,\delta)$ there exists a~point $y_r\in\R^2$
	such that $|y_r - x| = r$ and $|f(y_r) - t| < \eps$.
	Then
	\begin{equation*}
		\lim_{r\to 0} \dashint_{B(x,r)} |f(y) - t| \, dy = 0,
	\end{equation*}
	i.e., $x$ is a Lebesgue point of $f$ and $f(x) = t$.
\end{lemma}

\begin{proof}
	Let
	\begin{equation*}
		S_r := \{y \in \R^2 \;|\; |x-y| = r\}
	\end{equation*}
	denote the circle with radius $r$ and center $x$.
	Let $\eps > 0$. Let $\delta$ be given by the assumptions of the lemma.
	Then by standard properties of Sobolev functions, for a.e. $r\in(0,\delta)$  the following assertions are fulfilled:
	\emph{all points $y\in S_r$ are the Lebesgue points of~$f$ and $\H^1$-almost all points $y\in S_r$ are the Lebesgue points of~$\nabla f$, moreover, $f$ is absolutely continuous on~$S_r$, 
	and its tangential derivative with respect to~$S_r$ coincides with $\nabla f(y)\cdot\bm\tau$ for $\H^1$-almost all $y\in S_r$, where $\bm\tau$ is the~corresponding tangential vector to~$S_r$}.

	Thus, for a.e. $r\in(0,\delta)$ and for all points $y \in S_r$ we have
	\begin{equation*}
		|f(y) - t| < |f(y) - f(y_r)| + \eps \le \int_{S_r} |\nabla f| \, d\H^1 + \eps.
	\end{equation*}
	Let $\tilde B_r(x) := B_r(x) \setminus B_{r/2}(x)$. Then for any $r\in (0,\delta)$
	\begin{equation}
		\begin{aligned}\label{wrk3}
			\int_{\tilde B_r(x)} |f(y) - t| \, dy
			&= \int_{r/2}^{r} \left( \int_{S_\rho} |f(y) - t| \, d\H^1 \right) \, d\rho \\
			&\le \int_{r/2}^{r}\left(\int_{S_\rho} |\nabla f| \, d\H^1 + \eps\right)   2\pi \rho\, d\rho
			\le 2\pi r \int_{\tilde B_r(x)} |\nabla f| \, dy + \eps \L^2(\tilde B_r(x)).
		\end{aligned}
	\end{equation}
	By Cauchy--Bunyakovsky inequality we have
	\begin{equation*}
		\int_{\tilde B_r(x)} |\nabla f| \, dy \le \left(\int_{\tilde B_r(x)} |\nabla f|^2 \, dy\right)^{1/2} \left(\int_{\tilde B_r(x)} 1 \, dy\right)^{1/2}
	\end{equation*}
	By the absolute continuity of the Lebesgue integral, for any $\eps > 0$ there exists $\delta_1 \in (0,\delta)$
	so that for all $r\in (0,\delta_1)$ we have $\left(\int_{B_r(x)} |\nabla f|^2 \, dy\right)^{1/2} < \eps$.
	Since $ r = \frac{2}{\sqrt{3\pi}} \sqrt{\L^2(\tilde B_r(x))}$, by \eqref{wrk3} we have
	\begin{equation}
		\int_{\tilde B_r(x)} |f(y) - t| \, dy \le \left(1 + 4\sqrt{\frac{\pi}{3}}\right) \eps \L^2(\tilde B_r(x)).
		\label{telescope}
	\end{equation}
	Replacing in \eqref{telescope} $r$ with $r_k:=2^{1-k}r$ and summing for all $k\in \N$ we obtain
	\begin{equation*}
		\int_{B_r(x)} |f(y) - t| \, dy \le \left(1 + 4\sqrt{\frac{\pi}{3}}\right) \eps \L^2(B_r(x)).
		\qedhere
	\end{equation*}
\end{proof}

Lemma~\ref{removable-singularity} implies that if $f\in W^{1,2}(\R^2)$ is constant on some connected set $C\subset \R^2$, then $f$ takes the same value on the closure of $C$:

\begin{theorem}\label{constancy-on-connected-sets}
	Let $U \subset \R^2$ be an open set and suppose
	that $f\in W^{1,2}_\loc(U)$. Then the following properties hold:
	
	\begin{enumerate}
		\item[(i)] Let $t\in \R$ and let $C \subset U$ be a connected set with $\diam C > 0$.
		If $f(x)=t$ for $\H^1$-a.e. $x\in C$, then any $x\in U\cap \overline{C}$
		is a Lebesgue point of $f$ and we have $f(x) = t$.
		
		\item[(ii)] Let $\gamma \colon [0,1] \to U$ be a closed simple Lipschitz curve, $s\in \R$, 
		and $f(x)=t$ for $\H^1$-a.e. $x\in \gamma([0,1])$. Then for all $s\in [0,1]$ the point $\gamma(s)$ is a Lebesgue point of $f$
		and $f(\gamma(s)) = t$.
		
		\item[(iii)] If $U=\R^2$ and, in addition, $f\in W^{1,1}(\R^2)$, then
		for all $t\in \reg f$ all points of the set $\{f=t\}^*$
		are Lebesgue points of $f$ and 
		\begin{equation}
			\{f=t\}^*\subset \{f=t\}.
		\end{equation}
	\end{enumerate}
\end{theorem}

\begin{proof}
	(i). Let $x\in U\cap \overline{C}$.
	Since $C$ is connected and $\diam C >0$, there exists $\delta>0$
	such that for any $r\in (0, \delta)$ the circle
	$S_r(x)= \{y \in \R^2 \;|\; |x-y| = r\}$
	intersects $C$ at some point $y_r$.  Let $\widetilde S_f$ denote the set of all such $y_r$ where $f(y_r)\ne t$.  By the theorem assumptions, $\H^1(\widetilde S_f)=0$. Since the~$\H^1$-measure of the orthogonal projection of a set does not exceed the $\H^1$-measure of the original set (see e.g. \cite[Proposition 2.49]{AFP}), we have $\L^1(N)=0$, where $N=\{r\in (0,\delta) : y_r\in \widetilde S_f\}$. In other words, for almost all $r\in (0, \delta)$ the circle
	$S_r (x)$
	intersects $C$ at some point $y_r$ with $f(y_r)=t$. Then 
	by Lemma~\ref{removable-singularity}
	we have $f(x) = t$. Clearly,~(i) implies (ii),  and (ii) implies (iii).
\end{proof}

Nelson's example \eqref{Nelson-stream-function} shows that the assumption $\nabla f \in L^2_\loc(U)$ cannot be replaced with a weaker assumption $\nabla f \in L^p_\loc(U)$ with $p<2$.

We have already mentioned in the Introduction, that the level sets of functions from $W^{1,2}(\R^2)$ are not necessarily closed.
However, their nontrivial connected components are:

\begin{cor}\label{cor-connected-components-are-closed}
	Let $f\in W^{1,2}_\loc(\R^2)$, then all  the nontrivial connected components of the level sets of $f$ are closed. 
\end{cor}

\begin{proof}
	Let $t\in \R$ and let $C$ be a connected component of $\{f=t\}$ with $\diam C>0$.
	By Theorem~\ref{constancy-on-connected-sets} we have $\overline{C} \subset \{f=t\}$.
	Since closure of a connected set is connected, by maximality of $C$ among the connected subsets of $\{f=t\}$ we have $\overline{C} \subset C$. Hence $C = \overline{C}$.
\end{proof}

Now we can easily improve Lemma~\ref{lemma-constancy-along-connected-components}:

\begin{theorem}\label{constancy-along-connected-components}
	Suppose that $f, g \in W^{1,1}(\R^2) \cap W^{1,2}(\R^2)$ and $\nabla f \parallel \nabla g$ a.e.
	Then for a.e. regular value $t$ of $g$ and for every $\gamma\in \cd^M\{g>t\}$ there exists a constant $c=c(t,\gamma)\in \R$
	such that for all $x\in \gamma([0,1])$ we have $f(x) = c$
	and $x$ is a Lebesgue point of $f$.
\end{theorem}

\begin{proof}
	By Lemma~\ref{lemma-constancy-along-connected-components} there exists a negligible set $N \subset \R$ such that $N^c \subset \reg g$ and for all $t\in N$
	for any $\gamma \in \cd^M \{g>t\}$ there exists $c \in \R$ such that
	$\H^1$-a.e. $x\in \gamma([0,1])$ is a Lebesgue point of $f$
	and the equality $f(x)=c$ holds.
	Then by Theorem~\ref{constancy-on-connected-sets}
	any $x\in \gamma([0,1])$ is a Lebesgue point of $f$ and $f(x)=c$.
\end{proof}

Furthermore, for $W^{1,2}$-functions we can refine Theorem~\ref{levelset-subset-of-essential-levelset} (a further refinement will given in Theorem~\ref{monotonicity-implies-continuity}):

\begin{prop}
\label{almost-all-level-sets-are-essential}
Let $f\in W^{1,1}(\R^2) \cap W^{1,2}(\R^2)$ be a monotone function.
Then for a.e. $t\in \reg f$ we have
\begin{equation}
\{f=t\}^* = \{f=t\}
\quad\text{and}\quad
\{f=t\} \cap S_f = \emptyset,
\label{essential-level-set-vs-standard-one}
\end{equation}
where $S_f$ is the approximate discontinuity set of $f$.
\end{prop}

\begin{proof}
By Theorem~\ref{levelset-subset-of-essential-levelset} for a.e. $t\in \reg f$ we have the inclusion
\begin{equation}\label{reggv3}
	\{f=t\} \subset \{f=t\}^*.
\end{equation}
By Theorem~\ref{constancy-on-connected-sets}(iii) for any such $t$
we have the opposite inclusion
\begin{equation}
	\{f=t\}^* \subset \{f=t\}.
\end{equation}
and, furthermore, $\{f=t\}^* \cap S_f = \emptyset$. 
\end{proof}

\begin{cor}\label{E-T-is-measurable}
Let $f\in W^{1,1}(\R^2) \cap W^{1,2}(\R^2)$ be a monotone  function.
Then there exists a negligible set $N\subset \R$ such that
for any Borel set $T \subset \reg f \setminus N$
the set $E_T := \bigcup_{t\in T} \{f=t\}^*$ is measurable.
\end{cor}

\begin{proof}
By Proposition~\ref{almost-all-level-sets-are-essential} there exists a negligible set $N\subset \R$
such that \eqref{essential-level-set-vs-standard-one}  holds for all $t\in \reg f \setminus N$.
Let $T \subset \reg f \setminus N$ be a Borel set. Then by \eqref{essential-level-set-vs-standard-one}
\begin{equation*}
E_T = f^{-1}(T).\end{equation*}
The set $f^{-1}(T)$ is measurable, since $T$ is Borel. 
\end{proof}

The Sobolev regularity exponent $p=2$ in Theorem~\ref{constancy-along-connected-components} and Proposition~\ref{almost-all-level-sets-are-essential} is sharp:

\begin{rem}[Nelson's example revisited once again]
	If $f\in W^{1,1}(\R^2)\cap W^{1,p}(\R^2)$ with $p<2$, then the essential level sets $\{f=t\}^*$ may intersect the approximate discontinuity set $S_f$ for some \emph{non-negligible} set of values $t$, in contrast to \eqref{essential-level-set-vs-standard-one}.
	This is demonstrated by the stream function $f$ from the Nelson's example \eqref{Nelson-stream-function}: the origin~$O$ belongs to $\{f=t\}^*$ for all $t\in (0,\frac{\pi}{3})$, and at the same time $O$ belongs to $S_f$. Indeed, for any $h\in (0,\frac{\pi}{3})$ the density of $\{f>h\}$ at $O$ belongs to the interval~$(0,1)$, and this implies that the identity~\eqref{app-lim} at $O$ does not hold, as one can show using, e.g., \cite[Proposition 3.65]{AFP}.
\end{rem}

\subsection{Connected components of the level sets}\label{sec-connected-components}

\begin{prop}\label{essential-level-sets-and-connected-components}
	Let $f\in W^{1,1}(\R^2) \cap W^{1,2}(\R^2)$.
	Then for a.e. $t\in \reg f$ we have $\{f=t\}^* = E_t^*$,
	where $E_t^*$ denotes the union of all connected components $C$ of the set $\{f=t\}$ such that $\H^1(C) >0$.
	
	Furthermore, for a.e. $t\in \reg f$ the corresponding curves $\gamma([0,1])$ {\rm(}with $\gamma \in \cd^M \{f>t\}${\rm)} are precisely the connected components of $E_t^*$; in particular, they are pairwise disjoint.
\end{prop}

In order to prove Proposition~\ref{essential-level-sets-and-connected-components}, let us recall that a set $Y\subset \R^d$ is called a \emph{triod} if there exist continuous curves $\gamma_{i} \colon [0,1] \to \R^d$, $i\in\{1,2,3\}$, such that $\gamma_1(0) = \gamma_2(0) = \gamma_3(0)$, the sets $\gamma_i((0,1])$ are pairwise disjoint and $Y = \bigcup_{i=1}^3 \gamma_i([0,1])$.
The following result is a corollary of \cite[Theorem 1]{Moore1928}, with an independent proof given in \cite[Lemma 2.15]{ABC_2013_LipSard}; see also \cite[Lemma 1]{Kronrod1950} for a related result:
\begin{lemma}\label{triod-lemma}
	Let $\mathscr{F}$ ba a family of pairwise disjoint triods in $\R^2$.
	Then $\mathscr{F}$ is at most countable.
\end{lemma}

\begin{proof}[Proof of Proposition~\ref{essential-level-sets-and-connected-components}]
	By Theorem~\ref{constancy-on-connected-sets}~(iii) there exists a negligible set $N_1 \subset \R$ such that $N_1^c \subset \reg f$ and for all $t\in N_1^c$
	the inclusion~$\{f=t\}^* \subset \{f=t\}$ holds. Since for any $\gamma \in \cd^M \{f>t\}$ we have $\H^1(\gamma([0,1]))>0$, we also have $\{f=t\}^* \subset E_t^*$.
	Now let us prove that for a.e. $t\in N_1^c$ the opposite inclusion holds.
	
	Let $N_2 \subset \R$ denote the family of all $t\in \R$ such that $\{f=t\}$ contains a triod $Y_t$. 
	By the definition of level set, the family $\{Y_t\}_{t\in N_2}$ is disjoint.
	Then by Lemma~\ref{triod-lemma} the set $N_1$ is at most countable.
	Let $N:=N_1\cup N_2$ and suppose that $t\in N^c$.
	
	Let $C$ be a connected component of $\{f=t\}$ with $0<\H^1(C)<\infty$.
	Let $x\in C$ and suppose that $x$ belongs to $\Gamma := \gamma([0,1])$
	for some $\gamma \in \cd^M \{f>t\}$.
	Since $C \subset \{f=t\}^*$ modulo $\H^1$ (see Definition~\ref{def-reg}), 
	such $\gamma$ exists for $\H^1$-a.e. $x\in C$.
	Since the connected sets $C$ and $\Gamma$ intersect and both of them are subsets of $\{f=t\}$, we have $\Gamma \subset C$.
	Let us prove that $C \subset \Gamma$.
	Suppose that, on the contrary, $y\notin\Gamma$ for some~$y\in C$.  
	Since $C$ is arcwise connected (see e.g. \cite[Lemma 3.12]{Falconer1985}),
	there exists an injective continuous function $\theta\colon[0,1]\to \R^2$
	such that $\theta(0)=x,$ $\theta(1)=y$ and $\theta([0,1])\subset C$. Let $m:=\sup\{s\in[0,1] : \theta(s) \in \Gamma\}$.
	By~continuity of $\theta$ we have $\theta(m)\in \Gamma$ (since $\Gamma$ is compact).
	By construction, $0<m<1$.
	Then, by definition of $m$, the sets $\Gamma$ and $\theta((m,1])$ are disjoint.
	Hence the set $\{f=t\}$ contains a triod,
	which is impossible since~$t\notin N_2$.
	Thus $C \subset \Gamma$, therefore, by the previous inclusion, $C=\Gamma$.
	Hence $C \subset \{f=t\}^*$ and, consequently, $E^*_t \subset \{f=t\}^*$.
	
	It remains to note that for any connected component $C$ of $E_t^*$
	for any $x\in C$ there exists at most one $\gamma \in \cd^M \{f>t\}$ such that $x\in \gamma([0,1])$. Suppose that $\theta \in \cd^M \{f>t\}$ is another such curve.
	Then by the above argument $\theta([0,1]) = C = \gamma([0,1])$.
	This clearly contradicts Corollary~\ref{structure-of-essential-boundary-2d}, according to which any two distinct curves $\gamma, \theta \in \cd^M\{f>t\}$ are disjoint up to $\H^1$-negligible subsets.
\end{proof}

In order to discuss the weak Sard property in the Sobolev setting, we need to generalize the Corollary~\ref{E-T-is-measurable}.

\begin{prop}\label{Borel-prop}
	If $f\in W^{1,2}(\R^2)$, then the union $E^*$ of all connected components $C$ of the $f$-level sets with $\H^1(C)>0$
	is Borel.
\end{prop}

\begin{proof}
Recall, that by our agreement, $f$ is a~precisely represented function (see Section~\ref{sec:np}\,).
	For any $\ell > 0$ let $M_\ell$ denote the union of all connected components $C$
	of the level sets of $f$ such that $\diam(C) \ge \ell$.
	Let us prove that $M_\ell$ is closed.
	
	Let $x_k \in M_\ell$ and $x_k\to x$ as $k\to \infty$, $k\in\N$.
	Put $t_k:= f(x_k)$ and let $\tilde C_k$ denote the connected component of the level set $\{f = t_k\}$
	such that $x_k\in \tilde C_k$ and $\diam \tilde C_k\ge \ell$. By Theorem~\ref{constancy-on-connected-sets}, every $\tilde C_k$ is a~closed set consisting of Lebesgue points of~$f$. 
	
	Put $R=2\ell$ and $B=\overline{B_R(x)}$. Denote by $C_k$ the connected component of the set $\tilde C_k\cap B$ containing $x_k$.
	Then by construction $C_k$ is a compact connected set in $B$ and $\diam C_k\ge \ell$ as well.
	
	By elementary arguments (related to the trace theorems for Sobolev mappings), the inequality 
	$$\int\limits_B\bigl(f^2+|\nabla f|^2\bigr) \, dx<\infty$$
	implies the uniform boundedness of~$t_k$. 
	Indeed, since there exists two points $x,y\in C_k$ with $|x-y|\ge \ell$, the length of the orthogonal projection of $C_k$ on some line is at least $\ell$. 
	Without loss of generality suppose that this line is parallel to $Ox_1$ and the projection of $C_k$ on $Ox_1$ contains $[0,\ell]$. 
	Let $Q = (0,\ell) \times (a,b)$, where $a$ and $b$ are such that $[0,\ell]\times \{a,b\} \cap B = \emptyset$. 
	Then for any $x_1\in [0,\ell]$ the segment $\{x_1\}\times [a,b]$ intersects $C_k$ at some point $(x_1,\tilde x_2)$.
	Let $g$ denote the trace of $f$ on the upper side of $Q$.
	Then
	\begin{align}
		|\ell t_k| &\le \int_0^\ell |t_k - g(x_1)| \, dx_1 + \int_0^\ell |g(x_1)| \, dx_1
		= \int_0^\ell |f(x_1, \tilde x_2) - g(x_1)| \, dx_1 + \int_0^\ell |g(x_1)| \, dx_1\\
		&\le \int_0^\ell \int_a^b |\d_{x_2} f| \, dx_2 \, dx_1 + \int_0^\ell |g(x_1)| \, dx_1 \le c\|f\|_{W^{1,1}(Q)},
	\end{align}
	where $c$ is some constant related to the standard estimate of the trace.
	Hence the sequence $\{t_k\}$ is bounded.
	Passing to a~subsequence, if necessary, we may assume, that the sequence $t_k$
	converges to some $t\in \R$ as $k\to \infty$.
	
	Let $\mathscr F$ denote the family of all nonempty compact sets $C \subset B$.
	We endow $\mathscr F$ with the Hausdorff metric.
	By the well-known Blaschke selection theorem
	(see, e.g., \cite[Theorems 3.16]{Falconer1985}), $\mathscr F$ is a compact metric space.
	Let $\mathscr F_c$ denote the family of all $C \in \mathscr F$ that are connected.
	Then $\mathscr F_c$ is closed with respect to the~convergence in the Hausdorff metric
	(see e.g. \cite[Theorems 3.18]{Falconer1985}).
	Hence, passing to a subsequence, we may assume that
	the sequence $C_k$ converges in the Hausdorff metric
	to some connected compact set $C_0$ with $\diam(C_0) \ge \ell$.
	Since $x_k \in C_k$ for all $k\in \N$ and $x_k \to x$ as $k\to \infty$, we have $x\in C_0$.
	It remains to prove that $t = f(z)$ for any $z\in C_0$. 
	
	Indeed, since $z$ is a limiting point for some $C_k\ni z_k\to z$ and $\diam(C_k) \ge \ell$,  for any $r\in (0, \ell/2)$
	there exists $N\in \N$ such that for all $k > N$
	the circle $S_r := \{y\in \R^2 : |y-z| = r\}$ intersects $C_k$.
	Hence, for any $\eps > 0$ on any such circle we may find a~point~$\xi\in C_k$ 
	such that $|f(\xi) - t| = |t_k - t| < \eps$.
	Then by Lemma~\ref{removable-singularity} we have $f(z) = t$.
\end{proof}

\begin{rem}
	For a generic function from $W^{1,1}(\R^2) \cap W^{1,2}(\R^2)$, the <<gap>> between the union of the level sets and the union of the essential level sets~$E^*$ is a subset of the critical set $\{\nabla f=0\}$ modulo $\L^2$. This follows directly from the Coarea formula applied to the complement $\R^2\setminus E^*$. \end{rem}

\subsection{Monotonicity and continuity}\label{sec-monotonicity-and-continuity}
In fact, Proposition~\ref{almost-all-level-sets-are-essential} can be strengthened even further:

\begin{theorem}\label{monotonicity-implies-continuity}
	Let $f\in W^{1,1}(\R^2) \cap W^{1,2}(\R^2)$ be a monotone function.
	Then there exist at most one point $x_{+} \in \R^2$
	such that $f(x_{+}) = +\infty$
	and at most one point $x_{-} \in \R^2$
	such that $f(x_{-}) = -\infty$.
	Furthermore, for any $x\in \R^2 \setminus \{x_{+}, x_{-}\}$
	the function $f$ is continuous at $x$, i.e., $f\in C(\R^2 \setminus \{x_{+}, x_{-}\})$. In particular, $S_f \subset \{x_{+}, x_{-}\}$. If, in addition, $f$ is bounded, then we have $f\in C(\R^2)$.
\end{theorem}

\begin{proof}
	Without loss of generality assume that $f \ge 0$. Denote $M=\sup f$.  
	By Proposition~\ref{almost-all-level-sets-are-essential},
	there exists a measurable set $T \subset (0,M)$ such that $T \subset \reg f$,
	$(0,M) \setminus T$ is negligible, and for all $t\in T$ we have $\{f=t\} = \{f=t\}^*\ne\emptyset$.
	For any $t\in T$ let $I_t$ is the bounded connected component of $\R^2 \setminus \{f=t\}^*$, given by Jordan curve theorem, and let $E_t$ denote the corresponding unbounded connected component. By Lemma~\ref{Jordan-interior} we have 
	\begin{equation}
		\label{I-and-E}
		I_t = \{f>t\} \mod \L^2
		\quad\text{and}\quad
		E_t = \{f\le t\} \mod \L^2.
	\end{equation}
	
	For any $x\in \R^2$  let us define
	\begin{equation}
		\overline{f}(x) := \begin{cases}\sup\,\{t \in T : x \in I_t \}, \ \ & \mbox{ if }f(x)>0;\\
		                 0,\ \ & \mbox{ if }f(x)=0.
		                 \end{cases}
	\end{equation}
		Now let $t\in T$ be such that $t<\overline{f}(x)$.
	Since $I_t$ is open, there exists $r>0$ such that $B_r(x) \subset I_t$.
	Hence, by \eqref{I-and-E} we have $f>t$ a.e. on $B_r(x)$.
	Then by the definition of precise representative $f\ge t$ \emph{everywhere} on $B_r(x)$.
	Consequently, $f(x) \ge \overline{f}(x)$ everywhere. 
	
	Suppose now that $s\in T$ is such that $\overline{f}(x) < s$.
	Then either $x\in \{f=s\}^*=\{f=s\}$ or $x \in E_s$.
	Slightly decreasing, if necessary, the value $s$, we can exclude the first option (but keep the inequality $\overline{f}(x) < s$).
	So without loss of generality $x\in E_s$.
	Since $E_s$ is open, there exists $\rho>0$ such that $B_\rho(x) \subset E_s$.
	Hence by \eqref{I-and-E} we have $f\le s$ a.e. on $B_\rho(x)$.
	Then, by the definition of precise representative, $f\le s$ \emph{everywhere} on $B_\rho(x)$.
	Consequently $f(x) \le \overline{f}(x)$, and, therefore, $f(x) = \overline{f}(x)$ everywhere on~$\R^2$.
	
	In the case $0<\overline{f}(x) < +\infty$ for any $y \in B_r(x) \cap B_\rho(x)$ we have
	\begin{equation}
		t \le f(y) \le s,
	\end{equation}
	where $t,s,r, \rho$ are from the above procedure.
	Since $t$ and $s$ may be chosen arbitrarily close to $\overline{f}(x)$,
	this implies that $f$ is continuous at $x$. The same argument works for the case $f(x)=0$ with an~evident simplification (i.e., we need only the~one-sided estimate from above for this case). 
	
	Suppose now that $\overline{f}(x) = \overline{f}(y) = +\infty$ for two distinct points $x$ and $y$. Then for all $t \in T$ we have $x,y \in I_t$, i.e., $x$ and $y$ are inside the bounded domain surrounded by the simple Jordan curve~$\{f=t\}^*$. 
	Then $\H^1(\{f=t\}^*) \ge  |x-y|$ and hence $\int_T \H^1(\{f=t\}) \, dt = +\infty$, which is impossible by Coarea formula.
\end{proof}

As we already mentioned, there exist some other definitions of monotonicity for functions from $\R^d$ to~$\R$. In one of them, which is common for continuous functions, the level sets $\{f=t\}$ are required to be connected for all $t\in \R$ (see e.g. \cite[p.~358]{Engelking1989}).
In \cite[Proposition 1]{BT} it was proved that for a compactly supported Lipschitz function $f$ this definition of monotonicity is equivalent to Definition~\ref{def-monotone-BT}.
Using Theorem~\ref{monotonicity-implies-continuity}, it is easy to generalize this result.
Let us state some auxiliary results.

\begin{lemma}[{\cite[Lemma 2.17]{ABC_2013_LipSard}}]
	\label{triod-free-closed-connected-set-with-finite-length}
	Let $E\subset \R^d$ be a closed, connected set with $\H^1(E)<\infty$.
	If~$E$ contains no triods, then it is a simple curve, possibly closed.
	More precisely, there exists a Lipschitz function $\gamma\colon[0,1]\to\R^d$, which is injective on $[0,1)$ and satisfying $E = \gamma([0,1])$.
\end{lemma}

\begin{lemma}\label{connected-open-set-with-connected-exterior}
	Suppose that the space~$\R^d$ is represented as a disjoint union $\R^d = A \sqcup B \sqcup C$, where $A$ and $B$ are connected open sets with common boundary $C$, and $d\ge 2$. Then $C$ is connected as well.
\end{lemma}

In \cite{CKL2011}\footnote{We would like to thank Roman Dribas for suggesting this reference.} it was proved that for any open set $U \subset \R^d$ ( $d\ge 2$) the boundary $\d U$ is connected if and only if~$U^c$ is connected. This result immediately implies Lemma~\ref{connected-open-set-with-connected-exterior}, since the closure of any connected open set in $\R^d$ is connected. We also give a self-contained proof of Lemma~\ref{connected-open-set-with-connected-exterior} in the Appendix~\ref{appendix-connectedness-lemma}.

Now we are in a position to prove the equivalence of the two definitions of monotonicity:

\begin{theorem}\label{alt-mono}
	Let $f$ be a~bounded function from $W^{1,1}(\R^2) \cap W^{1,2}(\R^2)$. Then $f$ is monotone if and only if for all nonzero $t\in \R$ the level set $\{f=t\}$ is connected.
\end{theorem}

\begin{proof}
	\emph{Necessity.} Without loss of generality assume that $f \ge 0$.
	Denote $M=\sup f$. If $M=0$, then $f\equiv0$ and nothing to prove. Now suppose $M>0$. Then by the theorem assumptions, $0< M<+\infty$. By Proposition~\ref{almost-all-level-sets-are-essential} for a.e. $t\in (0,M)$ the level set $\{f=t\}$ coincides with the curve $\{f=t\}^*$, hence it is connected.
	It remains to prove that $\{f=t\}$ is connected for \emph{all} $t\in (0,M]$, Denote $I:=(0,M)$.
	
	By Theorem~\ref{monotonicity-implies-continuity} we have $f\in C(\R^2)$, hence for all $t\in I$ the sets $\{f<t\}$ and $\{f>t\}$ are open. By the definition of monotonicity, for a.e. $t\in I$ these sets are indecomposable, hence connected (being open) by Theorem~\ref{indecomposablity-criterion-for-open-sets}.
	We claim that they are connected for \emph{all} $t\in I$.
	Suppose that, for instance, $\{f>t\}$ is disconnected for some $t\in I$.
	Then there exist disjoint open sets $U$ and $V$ such that $\{f>t\} \subset U \sqcup V$
	and the sets $\{f>t\} \cap U$ and $\{f>t\} \cap V$ are nonempty.
	By continuity of $f$, there exists $\delta>0$ such that for all $r\in (0,\delta)$ the sets $\{f>t+r\}\cap U$ and $\{f>t+r\}\cap V$ are nonempty. 
	Since $\{f>t+r\} \subset \{f>t\}\subset U \sqcup V$, if follows that the sets $\{f>t+r\}$ are disconnected.
	This is a contradiction, since for a.e. $r>0$ they are connected.
	
	Since for all $t\in I$ the sets $\{f>t\}$ and $\{f<t\}$ are connected and open, by Lemma~\ref{connected-open-set-with-connected-exterior} $\{f=t\}$ is connected as well. It remains to check the connectedness of the limiting level set $\{f=M\}$. Take an~increasing sequence or regular values $t_k<M$ such that $t_k\to M$ and the level set $\{f=t_k\}$ coincides with the simple Jordan curve $\{f=t_k\}^*$. Then by construction $E_k:=\{f\ge t_k\}$
	coincides with the compact connected set surrounded by this simple Jordan curve. Again by construction
	$$E_{k+1}\subset E_k, \qquad \{f=M\}=\bigcap\limits_{k\in\N}E_k.$$
Then the connectedness of the~set $\{f=M\}$ follows from the well-known fact that intersection of a~countable decreasing family of compact connected sets is compact and connected as well. 
	 
	 	\emph{Sufficiency.}
	Let $0\ne t\in \R$ and suppose that $\diam \{f=t\} > 0$.
	Since $\{f=t\}$ is connected, it is also closed by Corollary~\ref{cor-connected-components-are-closed}.
	(If $\diam\{f=t\}=0$, then $\{f=t\}$ is a single point or empty, thus also closed.)
	Hence the level sets of $f$ are closed.
	
	By the Coarea formula,  for a.e. $t\in \R$ we have $\H^1(\{f=t\}) < \infty$.
	Furthermore, by Lemma~\ref{triod-lemma} for a.e. $t\in \R$ the level set $\{f=t\}$ has no triods. Therefore, for a.e. $t\in \R$ the level set $\{f=t\}$ is a closed connected set with finite length and without triods. Then by Lemma~\ref{triod-free-closed-connected-set-with-finite-length} for all such $t$
	the level set $\{f=t\}$ is a closed simple Lischitz curve.
	But for a.e. $t$ we also have $\d^M \{f>t\} = \{f=t\} \mod \H^1$,
	so the essential boundaries of $\{f>t\}$ are also simple Lipschitz curves up to $\H^1$-negligible subsets. Hence by Theorem~\ref{boundary-of-simple-plane-sets} for a.e. $t$ the set $\{f>t\}$ is simple if $t>0$ and the set $\{f<t\}$ is simple if $t<0$. Thus, $f$ is monotone.
\end{proof}

\begin{rem}
The condition~$t\ne 0$ in assumptions of Theorem~\ref{alt-mono} is essential, for the corresponding counterexample one can take a~monotone function $f\in C^\infty(\R^2)\cap W^{1,1}(\R^2) \cap W^{1,2}(\R^2)$ of type 
$f(x,y)=\omega(x)e^{-y^2}$, where $0\ne\omega\in C^\infty_0(\R)$ is a nonnegative even function
nonincreasing on $[0,+\infty)$. By construction, $f>0$ on the axis $O_y=\{(0,y)\in\R^2: y\in\R\,\}$. 
 \end{rem}

\section{New features of the weak Sard property}\label{sec:wsp}

If $f\in C^2(\R^2)$, then by the classical Sard theorem $f(Z)$ is negligible,
where $Z=\{x\in \R^2 : \nabla f(x) = 0\}$ is the critical set of $f$.
Consequently, we have
\begin{equation}
	f_\# (\1_{Z} \cdot \L^2) \perp \L^1. \label{RSP}
\end{equation}
This property is well-defined also for Sobolev functions (for which the critical set $Z$ is understood as the set where the distributional derivative of $f$ vanishes), see also \cite{DribasGolovnevGusev}:
\begin{defn}\label{RSP-def}
	Let $Z$ denote the critical set of a function $f \in W^{1,1}(\R^2)$.
	If \eqref{RSP} holds, then we say that $f$ has the relaxed Sard property.
\end{defn}
A weaker version of this property was introduced in \cite{ABC_2014}:

\begin{defn}\label{WSP-old}
Let $f\in W^{1,1}(\R^2) \cap W^{1,2}(\R^2)$ and let~$E^*$ denote the union of all connected components $C$ of the level sets of $f$ such that $\H^1(C)>0$.
Let $Z$ be the critical set of $f$, i.e., the set of all points $x\in \R^2$ where either $f$ is not differentiable or $\nabla f(x)=0$. The~function~$f$ is said to have the \emph{weak Sard property}, if 
\begin{equation}
f_\# (\1_{Z\cap E^*} \cdot \L^2) \perp \L^1. \label{WSP}
\end{equation}
\end{defn}

For Lipschitz functions the left-hand side of \eqref{WSP} is well-defined, since the set $E^*$ is Borel, see \cite[Proposition 6.1]{ABC_2013_LipSard}.
By Proposition~\ref{Borel-prop} the set $E^*$ is measurable even if $f\in W^{1,2}(\R^2)$.
Examples of Lipschitz function which do not have weak Sard property were constructed in \cite[Section 5]{ABC_2013_LipSard}.

\begin{rem}\label{on-WSP-def}
The property \eqref{WSP}
is strictly weaker than \eqref{RSP}:
An example of Lipschitz function for which \eqref{WSP} holds but \eqref{RSP} does not hold
was constructed in \cite[Proposition 5.1]{Bon17a}
(a similar example in the class $C^{1+\alpha}$ is constructed in \cite{DribasGolovnevGusev}). There is a certain similarity with the famous counterexample by S.V.~Konyagin~\cite{Konyagin93}, who constructed a~$C^1$-smooth function $f:\R^2\to\R$ such that 
for any $t\in[-1,1]$ the preimage $f^{-1}(t)$ can {\it not} be covered by at most countable union of rectifiable curves, i.e., the~exceptional sets of $1$-measure zero are not negligible there as well. 
However, we will see in Corollary~\ref{pushforward-monotone-WSP} that
the properties \eqref{WSP} and \eqref{RSP} are equivalent if $f$ is monotone.
\end{rem}

\begin{rem}\label{importance-of-monotonicity}
In general \eqref{essential-level-set-vs-standard-one-weak} does not hold without the monotonicity assumption,
even if $f$ is Lipschitz. Indeed, consider the function $f$ mentioned in Remark~\ref{on-WSP-def} satisfying~\eqref{WSP}, but not~\eqref{RSP}.
By Proposition~\ref{essential-level-sets-and-connected-components}
for a.e. $t\in \R$ we have $\{f=t\}^* = E^*_t$.
Let $N$ denote the set of all such $t$ for which
the intersection $Z \cap \{f=t\} \setminus E^*_t$ contains a Lebesgue point of $f$.
If $\L^1(N) = 0$,  then $f_\# (\1_{Z\setminus E^*}\L^2) \perp \L^1$.
Then by \eqref{WSP} we have $f_\# (\1_Z \L^2) \perp \L^1$, which is a contradiction.
\end{rem}

In order to simplify the work with the weak Sard property within the framework of Sobolev spaces, we introduce the following definition (cf. \cite[Remark 4.3]{BBG2016}):

\begin{defn}
Let $f\in W^{1,1}(\R^2)$ and suppose that $M\subset \R^2$.
We say that \emph{the level sets of $f$ almost avoid the set $M$},
if there exist an $\L^2$-negligible set $X\subset \R^2$ and an $\L^1$-negligible set $N \subset \R$
such that for all regular values $t$ of $f$ satisfying $t\notin N$ we have
\begin{equation*}
\{f=t\}^* \cap (M\setminus X) = \emptyset.
\end{equation*}
\end{defn}

The definition above allows one to give an equivalent reformulation of the weak Sard property, which does not require explicitly the measurability of the set~$E^*$:

\begin{prop}\label{oldWSP=newWSP}
Let $f\in W^{1,1}(\R^2) \cap W^{1,2}(\R^2)$.
Then the level sets of $f$ almost avoid a measurable set $M\subset \R^2$ if and only if $f_\# (\1_{M\cap E^*} \L^2) \perp \L^1$.
Hence the level sets of $f$ almost avoid the critical set of $f$
if and only if $f$ has the weak Sard property.
\end{prop}

\begin{proof}
\emph{Necessity.} Let $N_1 \subset \R$ and $X\subset \R^2$ be negligible sets such that
for all $t\in N_1^c$ we have
\begin{equation*}
\{f=t\}^* \cap M \setminus X = \emptyset.
\end{equation*}
By Proposition~\ref{essential-level-sets-and-connected-components}, there exists a negligible set $N_2\subset \R$ such that $\{f=t\}^* = E_t^*$ for all $t\in N_2^c$.
Let $N:= N_1 \cup N_2$. Then for all $t\in N^c$ we have
\begin{equation*}
E_t^* \cap M  = \{f=t\}^* \cap M \subset X
\end{equation*}
hence $(E^* \cap M)\setminus f^{-1}(N) \subset X$, and, consequently, $f_\# (\1_{E^*\cap M} \L^2)$ is concentrated on $N$.

\emph{Sufficiency.} Let $N_1\subset \R$ be a negligible set on which $f_\#(\1_{M\cap E^*} \cdot \L^2)$ is concentrated. 
By Proposition~\ref{essential-level-sets-and-connected-components}, there exists a negligible set $N_2\subset \R$ such that $\{f=t\}^* = E_t^*$ for all $t\in N_2^c$.
Let $N:=N_1\cup N_2$.
Then the set $X:= (M\cap E^*)\setminus f^{-1}(N)$ is $\L^2$-negligible 
(since $f_\# (\1_X \L^2) = 0$)
and $(M\cap E^*) \setminus X \subset f^{-1}(N)$. Hence
for all $t\in N^c$ we have
\begin{equation*}
\{f=t\}^*\cap M\setminus X = E_t^* \cap (M\cap E^*) \setminus X \subset E_t^* \cap f^{-1}(N) \subset \{f=t\} \cap f^{-1}(N)= \emptyset.
\end{equation*}
\end{proof}

From \eqref{essential-level-set-vs-standard-one} and Proposition~\ref{oldWSP=newWSP} we get
\begin{cor}
\label{pushforward-monotone-WSP}
Let $f\in W^{1,1}(\R^2) \cap W^{1,2}(\R^2)$.
If $f$ is monotone, then $f$ has the weak Sard property~\eqref{WSP} if and only if \eqref{RSP} holds.
\end{cor}

Let $f,g \in W^{1,1}(\R^2) \cap W^{1,2}(\R^2)$ be such that
$|Df| \perp |Dg|$. 
Then $|Dg|$ is concentrated on $\{\nabla f=0\}$, in other words, 
$\{\nabla g \ne 0\} \subset \{\nabla f = 0\}$ modulo $\L^2$.
By Lemma~\ref{lemma-mutual-weak-Sard-property} we also have
\begin{equation}
	f_\# (\1_{\{\nabla g \ne 0\}}\L^2) \perp \L^1. \label{weak-mutual-wsp}
\end{equation}
Now we are going to strengthen it using the fine properties of the level sets, developed in the previous section.

\begin{prop}\label{mutual weak Sard property}
Let $f,g \in W^{1,1}(\R^2) \cap W^{1,2}(\R^2)$ with $|Df| \perp |Dg|$. 
Then there exists a negligible set $N\subset \R$
such that for all $x\in E_g := \bigcup_{t\notin N} \{g=t\}^*$
we have $f(x)\notin \reg f.$
In particular,
the level sets of $f$ almost avoid the set $E_g$. 
Consequently, the level sets of $f$ almost avoid the set $\{\nabla g\ne0\}$ and 
\begin{equation}\label{nut-1}
	f_\# (\1_{E_g}\L^2) \perp \L^1.
\end{equation}
\end{prop}

\begin{proof}
By Theorems \ref{constancy-along-connected-components} and \ref{constancy-on-connected-sets}, there exists
a negligible set $N_g \subset \R$ such that $N_g^c \subset \reg g$ and for all $t\in N_g^c$ and 
for any $\gamma \in\cd^M \{g>t\}$ there exists a constant $c=c(t,\gamma)$
such that $f(x) = c(t,\gamma)$ on $\gamma([0,1])$. (We emphasize, that the last equality holds {\it everywhere} on $\gamma([0,1])$, and {\it not only almost everywhere}, as in the claim of previous Lemma~~\ref{lemma-mutual-weak-Sard-property}.)
Let $A$ denote the set of all $t\in N_g^c$ for which there exists $\gamma \in\cd^M \{g>t\}$
such that $c(t,\gamma) \in \reg f$. By Lemma~\ref{lemma-mutual-weak-Sard-property} we have $\L^1(A) = 0.$
Let $N:=N_g \cup A.$ Then on the set $E_g:=\bigcup_{s\notin N} \{g=t\}^*$ the function $f$ takes only non-regular values. But the set of non-regular values of~$f$ has $1$-measure zero (by Proposition~\ref{reg-vals}\,). This implies, that the level sets of $f$ almost avoid the set~$E_g$ and also the assertion~\eqref{nut-1}. Finally, by Coarea formula the residual set $E_g\setminus \{\nabla g\ne0\}$ has the~$\L^2$-measure zero. Therefore, the level sets of $f$ almost avoid the set~$\{\nabla g\ne 0\}$ as well.  
\end{proof}

\begin{prop}\label{avoidance-criterea}
Suppose that $M\subset \R^2$, $f\in W^{1,1}(\R^2) \cap W^{1,2}(\R^2)$, and let $\{f_i\}_{i\in I}$ be a decomposition of $f$ given by Theorem~\ref{thm:BT}.
Then the level sets of $f$ almost avoid $M$ if and only if for all $i\in I$
the level sets of $f_i$ almost avoid $M$.
\end{prop}

\begin{proof}
\emph{Necessity.}
Fix $i \in I$ and let $g:=f_i$.
Let $N_f\subset \R$ and $X\subset \R^2$ be negligible sets such that $N_f^c\subset\reg f$ and for all $s\in \R \setminus N_f$
we have $\{f=s\}^* \cap M \setminus X = \emptyset$.
Without loss of generality we may assume,  that $N_f$ is a~$G_\delta$-set and hence $f^{-1}(N_f)$ is measurable.

By Theorem \ref{constancy-along-connected-components} and by monotonicity of $g$, there exists a negligible set $N_g\subset \R$ such that for all $t\in \reg g \setminus N_g$ there exists a constant $c=c(t)$ satisfying the inclusion 
\begin{equation*}
 \{g=t\}^* \subset \{f=c(t)\}.
\end{equation*}

Suppose that $\{g=t\}^*$ intersects $M\setminus X$. If $c(t)\notin N_f$, then also $\{f=c(t)\}^*$ intersects $M\setminus X$, which is not possible.
Hence $c(t) \in N_f$ and $\{g=t\}^* \subset f^{-1}(N_f)$. Then $\H^1(f^{-1}(N_f) \cap \{g=t\})>0$.
Thus
\begin{equation*}
A:= \{t\in \reg g \setminus N_g \;:\; \{g=t\}^* \cap M\setminus X \ne \emptyset\}
\end{equation*}
is a subset of
\begin{equation*}
B:= \{t \in \reg g \setminus N_g \;:\; \H^1(f^{-1}(N_f) \cap \{g=t\})>0\}.
\end{equation*}

By the Coarea formula for $g$ we have
\begin{equation*}
\int_{\R} \H^1(f^{-1}(N_f) \cap \{g=t\}) \, dt
=\int_{f^{-1}(N_f)} |\nabla g| \, dx \le \int_{f^{-1}(N_f)}|\nabla f| \, dx,
\end{equation*}
since $|\nabla g| \le |\nabla f|$.
On the other hand, by the Coarea formula for $f$ we have
\begin{equation*}
\int_{f^{-1}(N_f)}|\nabla f| \, dx = \int_{N_f} \H^1(\{f=t\}) \, dt = 0.
\end{equation*}
Hence,
\begin{equation*}
\int_{\R} \H^1(f^{-1}(N_f) \cap \{g=t\}) \, dt = 0.
\end{equation*}
This means that $\L^1(B)=0$ and, consequently, $\L^1(A) = 0$.

\emph{Sufficiency.}
For every $i\in I$ let $N_i\subset \R$ and $X_i \subset \R^2$ be negligible sets such that
$N_i^c \subset \reg f_i$ and for all $s\in N_i^c$ we have $\{f_i=s\} \cap (M\setminus X_i) = \emptyset$. Without loss of generality we may assume, that  $N_i$ is a~$G_\sigma$-set and hence the~preimage $f_i^{-1}(N_i)$ is measurable.
Put $X:= \bigcup_{i\in I} X_i$.

For every $i\in I$ by Theorem~\ref{constancy-along-connected-components} 
there exists a negligible set $Y_i \subset \R$
such that $Y_i^c \subset \reg f$ and for all $t\in Y_i^c$ and for all $\gamma \in \cd^M\{f>t\}$ there exists $c_i=c_i(t, \gamma)$ such that $f_i=c_i$ on $\gamma([0,1])$. Let $Y:= \bigcup_{i\in I} Y_i$.

Take $t\in Y^c$ and suppose that $\{f=t\}^*$ intersects $M\setminus X$.
Then there exists $\gamma \in \cd^M\{f>t\}$ such that $\Gamma:=\gamma([0,1])$ intersects $M\setminus X$.
If $c_i=c_i(t,\gamma) \in N_i^c$, then $\{f_i = c_i\}^*$ intersects $M\setminus X,$ which is not possible.
Hence $c_i\in N_i$. Thus $\Gamma \subset G$, where
\begin{equation*}
G:= \bigcap_{i\in I} f_i^{-1}(N_i)
\end{equation*}
Since $\H^1(\Gamma)>0$, we also have $\H^1(G\cap \{f=t\}^*)>0$.
Therefore, the set
\begin{equation*}
A := \{t \in Y^c : \{f=t\}^* \cap M\setminus X \ne \emptyset\}
\end{equation*}
is a subset of the set
\begin{equation*}
B := \{t \in \reg f : \H^1(G\cap\{f=t\})>0\}.
\end{equation*}
By the Coarea formula for all $i\in I$ we have
\begin{equation*}
\int_{G} |\nabla f_i|\, dx \le \int_{N_i} \H^1(G \cap \{f_i = s\}) \, ds = 0,
\end{equation*}
hence $\int_{G}|\nabla f|\, dx = \sum_{i\in I}\int_{G} |\nabla f_i| \, dx = 0.$
On the other hand, by the Coarea formula
\begin{equation*}
\int_{G} |\nabla f|\, dx = \int_{B} \H^1(G\cap\{f=t\}) \, dt.
\end{equation*}
Since $\H^1(G\cap\{f=t\})>0$ for all $t\in B$, we have $\L^1(B)=0$.
Consequently, $\L^1(A)=0$.
\end{proof}

\begin{rem}\label{roles}
We emphasize, that the Coarea formula itself is not sufficient to prove the assertions of Proposition~\ref{avoidance-criterea}. The essential role here is played also by Theorem~\ref{constancy-along-connected-components}, which fails in general for functions from~$W^{1,p}(\R^2)$ with $p<2$. 
\end{rem}

\begin{rem}\label{avoidance-of-countable-union}
Suppose that $f\in W^{1,1}(\R^2)$ with $\nabla f \in L^2(\R^2)$ and $\{M_i\}_{i\in I}$ is an at most countable family of sets $M_i \subset \R^2$. If for all $i\in I$ the level sets of $f$ almost avoid $M_i$,
then the level sets of $f$ almost avoid $M:=\bigcup_{i\in I} M_i$.
\end{rem}

Indeed, let $N_i\subset \R$ and $X_i \subset \R^2$ be negligible sets such that
all $t\in N_i^c$ are regular values of $f$ and for all such $t$ we have $\{f=t\}^* \cap M\setminus X_i = \emptyset$.
Then the sets $N:=\bigcup_{i\in I} N_i$ and $X:=\bigcup_{i\in I} X_i$ are negligible,
all $t\in N^c$ are regular values of $f$, and for all such $t$ we have $\{f=t\}^* \cap M\setminus X = \emptyset$.

\begin{prop}\label{WSP-for-monotone-components}
Suppose that $f\in W^{1,1}(\R^2)$ with $\nabla f \in L^2(\R^2)$.
Let $\{f_i\}_{i\in I}$ be a decomposition of $f$ given by Theorem~\ref{thm:BT}.
Then $f$ has the weak Sard property if and only if for all $i\in I$ the function $f_i$ has the weak Sard property.
\end{prop}

\begin{proof}
\emph{Necessity.}
Suppose that the level sets of $f$ almost avoid the set $\{\nabla f = 0\}$.
Then by Proposition~\ref{avoidance-criterea} for all $i\in I$ the level sets of $f_i$
almost avoid the set $\{\nabla f = 0\}$. Let us fix $i\in I$.

By Proposition~\ref{mutual weak Sard property} for all $j\in I \setminus\{i\}$
the level sets of $f_i$ almost avoid the set $\{\nabla f_j \ne 0\}$.
Therefore, by Remark~\ref{avoidance-of-countable-union} the level sets of $f_i$
almost avoid the set
\begin{equation*}
\{\nabla f = 0\} \cup \bigcup_{j\ne i} \{\nabla f_j \ne 0\}
=\left(\{\nabla f\ne 0\} \cap \bigcap_{j\ne i} \{\nabla f_j = 0\}\right)^c
=\left(\{\nabla f_i \ne 0\}\right)^c = \{\nabla f_i = 0\},
\end{equation*}
since $|\nabla f| = \sum_{i\in I} |\nabla f_i|$.

\emph{Sufficiency.}
Suppose that for all $i\in I$ the level sets of $f_i$ almost avoid the set $\{\nabla f_i=0\}$.
Since $|\nabla f| = \sum_{i\in I} |\nabla f_i|$, we have $\{\nabla f = 0\} = \bigcap_{i\in I} \{\nabla f_i = 0\}$. Hence, for all $i$ the level sets of $f_i$ almost avoid the set $\{\nabla f = 0\}$.
By Proposition~\ref{avoidance-criterea} this implies that the level sets of $f$ almost avoid $\{\nabla f = 0\}$.
\end{proof}

\section{Uniqueness of weak solutions of the continuity equation}\label{sec:uniq}

As in the previous section, we consider a compactly supported divergence-free vector field ${\ve\colon \R^2 \to \R^2}$
such that $\ve \in L^2(\R^2)$.
In this section we assume, that $f\in W^{1,2}(\R^2)$ is the stream function of~$\ve$
and $\{f_i\}_{i\in I}$ is a decomposition of $f$ into monotone $W^{1,2}$ functions given by Theorem~\ref{thm:BT}.
Recall, that by our agreement $f$ and each of $f_i$ are precisely represented.

\subsection{Proof of Theorem~\ref{uniqueness-criterea}: sufficiency of the weak Sard property.}
\label{un-sect-1}

First we prove that if $f$ has the weak Sard property, then \eqref{div(rho v)=0}
with a source term can be splitted into similar equations with the vector fields $\nabla^\perp f_i$
(this important fact implies a certain locality property for the divergence operator, cf. \cite[Proposition 6.1]{BBG2016}):

\begin{prop}\label{div-wsp-rhs}
Suppose that $f$ has the weak Sard property.
Then $R \in L^\infty(\R^2)$ and $Q \in L^\infty(\R^2)$ satisfy
\begin{equation}
\label{divergence-equation-with-rhs}
\dive(R \nabla^\perp f) = Q.
\end{equation}
if and only if for all $i\in I$
\begin{equation}
\label{equation-for-nonzero-components}
\dive(R \nabla^\perp f_i) = \1_{\{\nabla f_i \ne 0\}} Q
\end{equation}
and
\begin{equation}
\label{equation-for-zero-component}
Q=0 \quad\text{a.e. on } \quad \{\nabla f = 0\}.
\end{equation}
\end{prop}

\begin{proof}
By Proposition~\ref{WSP-for-monotone-components} we know that for all $i\in I$
the function $f_i$ has the weak Sard property.

Then arguing as in the proof of Proposition~\ref{div-decomp}
we deduce that \eqref{equation-for-nonzero-components} holds for all $i$.
Indeed, \eqref{disint-div-eq} now takes the form
\begin{equation}
\int R \nabla^\perp h \cdot (\nabla \fhi) \psi(h)  \, dx
+ \int R \nabla^\perp g \cdot (\nabla \fhi) \psi(h) \, dx =-  \int Q \fhi \psi(h) \, dx
\label{disint-div-eq2}
\end{equation}
for the same $h:= f_i$ and $g:=f-h$.
We decompose the right hand side of~\eqref{disint-div-eq2} using the identity $1 = \1_{\{\nabla h = 0\}} + \1_{\{\nabla h \ne 0\}}$ a.e.
And similar to the proof of Proposition~\ref{div-decomp} we introduce the measure
\begin{equation*}
\mu = R \nabla^\perp g \cdot \nabla \fhi \cdot \L^2 + Q \fhi \1_{\{\nabla h = 0\}} \cdot \L^2 =: \mu_1 + \mu_2
\end{equation*}
(since $\ve$ is compactly supported, we can actually restrict $\L^2$ to some set with finite measure, so $\mu$ is finite).
Then \eqref{disint-div-eq2} can be written as
\begin{equation*}
\int R \nabla^\perp h \cdot (\nabla \fhi) \psi(h)  \, dx + \int_{\{\nabla h \ne 0\}} Q \fhi \psi(h)\, dx
+ \int \psi(h) \, d\mu = 0.
\end{equation*}
As before, by \eqref{nut-2} we have $h_\# |Dg| \perp \L^1.$
Hence $h_\# \mu_1 \perp \L^1$.
On the other hand, by the weak Sard property $h_\# \mu_2 \perp \L^1$ (see Corollary~\ref{pushforward-monotone-WSP}).
Then we can obtain the required~\eqref{equation-for-nonzero-components} just following the last part of the proof of Proposition~\ref{div-decomp} taking now 
\begin{equation*}
a(t):=\int \left(R \frac{\nabla^\perp h}{|\nabla h|} \cdot \nabla \fhi
+\frac{Q\fhi}{|\nabla h|}\right) \, d\H^1 \rest \{h = t\}.
\end{equation*}
Finally, summing the corresponding equations \eqref{equation-for-nonzero-components} for all $i$,  we obtain $\dive (R \nabla^\perp f) = \1_{\{\nabla f \ne 0\}} Q$. Then \eqref{divergence-equation-with-rhs} implies \eqref{equation-for-zero-component}.
\end{proof}

\begin{prop}\label{div-wsp-rhs-mono}
Suppose that $f\in W^{1,2}(\R^2)$ is monotone.
Then functions $R \in L^\infty(\R^2)$ and $Q \in L^\infty(\R^2)$ satisfy
\begin{equation}
\label{steady-contin-rhs}
\dive(R \nabla^\perp f) = \1_{\{\nabla f \ne 0\}}Q.
\end{equation}
iff for a.e. ${y}\in \reg f$ we have
\begin{equation}
\label{div-along-curve}
(R\circ \gamma)' = a \cdot Q\circ \gamma,
\end{equation}
where $\gamma\colon [0,1]\to \R^2$ is a canonical parametrization of $\{f={y}\}^*$
and
\begin{equation}
\label{a-def}
a := \frac{|\gamma'|}{|\nabla f|\circ \gamma}
\end{equation}
belongs to $L^1([0,1])$.
\end{prop}

\begin{proof}
We argue as in the proof of Proposition~\ref{div-monotone-no-rhs}. By Coarea formula for any compactly supported smooth functions $\fhi, \psi$
\begin{equation*}
\int R \nabla^\perp f \cdot (\nabla \fhi) \psi(f) \, dx
= \int \int R \bm \tau \cdot \nabla \fhi \, d\H^1 \rest \{f = {y}\} \psi({y}) \, d{y}.
\end{equation*}
and
\begin{equation*}
\int \1_{\{\nabla f\ne 0\}} Q \fhi \psi(f) \, dx
= \int \int \frac{Q}{|\nabla f|} \fhi \, d\H^1 \rest \{f = {y}\} \psi({y}) \, d{y}.
\end{equation*}
Hence \eqref{steady-contin-rhs} holds if and only if for a.e. $t$
\begin{equation*}
\div (R \bm \tau \H^1 \rest \{f={y}\}) = \frac{Q}{|\nabla f|}\H^1 \rest \{f = {y}\}.
\end{equation*}
Let $\gamma$ be a canonical parametrization of $\{f={y}\}^*$.
By the area formula, we have
\begin{equation*}
\gamma_\# \left(\frac{|\gamma'|}{|\nabla f|\circ \gamma}\cdot (Q\circ \gamma) \L^1\right) = \frac{Q}{|\nabla f|}\H^1 \rest \{f={y}\}.
\end{equation*}
Hence by Proposition~\ref{steady continuity equation along curve} we obtain
\eqref{div-along-curve}.

It remains to note that by the area formula
\begin{equation*}
\int_0^1 a(s)\, ds \le \|Q\|_\infty \int\limits_{\{f={y}\}} \frac{1}{|\nabla f|} \, d\H^1,
\end{equation*}
which is finite (for a.e. ${y}$) by the Coarea formula.
\end{proof}

Below denote by $\T$ the 1-dimensional torus (or, equivalently, the~unit circle).

\begin{lemma}\label{1d-renorm}
Suppose that $T>0$, $a\in L^1(\T)$, and
$\rho \colon [0,T]\times \T\to \R$ is a bounded measurable solution to
\begin{equation}
\label{1d-transport}
\d_t (a(s) \rho(t,s)) + \d_s \rho(t,s) = 0
\end{equation}
with bounded initial condition $\rho|_{t=0} = \rho_0$.
Then for any~$\beta \in C^1(\R)$ the function $\beta(\rho)$ solves \eqref{1d-transport}
with the initial condition $\beta(\rho_0)$.
\end{lemma}

\goodbreak

\begin{proof}
Indeed, mollifying $\rho$ with respect to the variable $t$,
we get
\begin{equation}\label{mol1}
\d_t \bigl(a(s) \rho_\eps(t,s)\bigr) + \d_s \bigl(\rho_\eps(t,s)\bigr).
\end{equation}
By construction, $\partial_t\rho_\eps\in L^\infty([0,T]\times \T)$. Then, from equation~\eqref{mol1} we conclude, that $\partial_s\rho_\eps\in L^1([0,T]\times \T)$. 
Hence, $\rho_\eps\in W^{1,1}([0,T]\times \T)$. Therefore, the distributional equation \eqref{mol1} can be understood in the classical sense (pointwise a.e. in~$[0,T]\times \T$\,). 
Then $\beta(\rho_\eps)\in W^{1,1}([0,T]\times \T)$ as well, and by usual chain rule property we get
\begin{equation*}
\d_t (a \beta(\rho_\eps)) + \d_s \beta(\rho_\eps) = 0.
\end{equation*}
It remains to pass to the limit as $\eps \to 0$.
\end{proof}

\begin{rem} Propositions~\ref{div-wsp-rhs-mono} and \ref{1d-renorm} illustrate, that the main obstacle for the renormalization is the critical set $\{\nabla f=0\}$. The above Proposition~\ref{div-wsp-rhs} allows us to remove this critical set from the consideration, and this plays the key role in the arguments below. 
\end{rem}

Now we are going to show that the weak Sard property of $f$ is necessary and sufficient
for $\ve = \nabla^\perp f$ to have the renormalization property (and consequently for uniqueness of bounded weak solutions of the Cauchy problem for the continuity equation).
First we prove the following analog of Proposition \ref{div-wsp-rhs}:

\begin{prop}\label{reduction-of-CE-to-mono}
Suppose that $f$ has the weak Sard property.
Then $\rho \in L^\infty(0, T; L^\infty(\R^2))$ solves \eqref{continuity-equation}
with the initial condition $\rho_0$ if and only if:
\begin{enumerate}
 \item for any $i\in I$ the function $\rho_i := \rho \cdot \1_{\{\nabla f_i \ne 0\}}$
 solves
 \begin{equation*}
  \d_t \rho_i + \dive (\rho_i \nabla^\perp f_i) = 0
 \end{equation*}
 with the initial condition $\rho_i|_{t=0}=\rho_0 \cdot \1_{\{\nabla f_i \ne 0\}}$;
 \item for a.e. $t$ and we have $\rho(t, \cdot )\1_{\{\nabla f =0\}} = \rho_0 \1_{\{\nabla f =0\}}$.
\end{enumerate}
\end{prop}

\begin{proof}
For any $\psi \subset C_c^\infty([0,T))$ let us define
\begin{equation}\label{R-and-Q}
 R(x):=\int_0^T \rho(t,x) \psi(t)\,dt, \qquad
 Q(x):=\psi(0) \rho_0(x) + \int_0^T \rho(t,x) \psi'(t)\, dt.
\end{equation}
Then \eqref{continuity-equation}--\eqref{continuity-equation-initial-condition} holds if and only if
\begin{equation}
\label{RQ}
\dive(R \nabla^\perp f) = Q
\end{equation}
for all such $\psi \subset C_c^\infty([0,T))$.
By Proposition~\ref{div-wsp-rhs}
for each such $\psi$ the above equation holds if and only if
for all $i\in I$ we have
\begin{equation}
\label{RQ-i}
\dive(R \nabla^\perp f_i) = \1_{\{\nabla f_i \ne 0\}} Q
\end{equation}
and
\begin{equation*}
\1_{\{\nabla f = 0\}}Q=0.
\end{equation*}
The latter equality by arbitrariness of $\psi$ is equivalent to $\d_t (\1_{\{\nabla f = 0\}}\rho)=0$.
And again by arbitrariness of $\psi$ identity \eqref{RQ-i} is equivalent to
$\d_t \rho_i + \dive (\rho_i \nabla^\perp f_i) = 0$ for all $i\in I$.
\end{proof}

\begin{prop}\label{renormalization-mono}
If $f$ is monotone and has the weak Sard property, then $\ve$ has the renormalization property
{\rm(}and uniqueness holds for bounded weak solutions to \eqref{continuity-equation}--\eqref{continuity-equation-initial-condition}{\rm)}.
\end{prop}

\begin{proof} Let $\rho \in L^\infty(0, T; L^\infty(\R^2))$ solves \eqref{continuity-equation}
with the initial condition $\rho_0$.
Let $\Psi \subset C_c^\infty([0,T))$ denote a countable set of
test functions which is dense with respect to $L^\infty$ norm.
For each $\psi \in \Psi$ let $R$ and $Q$ be defined by \eqref{R-and-Q}.
Then by definition of weak solution we have \eqref{RQ}.
Note, that by item (2) of Proposition~\ref{reduction-of-CE-to-mono} we have $Q = 0$ a.e. on $\{\nabla f = 0\}$.
Then by Proposition~\ref{div-wsp-rhs-mono} from \eqref{RQ} we deduce that
for a.e. ${y} \in \reg f$ we have
\begin{equation*}
(R\circ \gamma)' = a \cdot Q\circ \gamma,
\end{equation*}
where $\gamma\colon [0,1) \to \R^2$ is a canonical parametrization of $\{f = {y}\}^*$ and $a$ is given by \eqref{a-def}.

Since $\Psi$ is countable, we may assume that the above equation holds for all $\psi \in \Psi$.

Let $\tilde R:= R\circ \gamma$ and $\tilde Q := Q\circ \gamma$.
Then for all $1$-periodic smooth test functions $\fhi$ we have
\begin{equation*}
\int_0^1 (\tilde R(s) \fhi'(s) \,ds + a(s)\tilde Q(s) \fhi(s)) \, ds = 0.
\end{equation*}
Denote $\tilde \rho(t,s):= \rho(t, \gamma(s))$. Then the last equality can be rewritten in the form
\begin{equation*}
\int_0^1\tilde \rho_0(s) a(s) \psi(0) \fhi(s)\,ds+\int_0^1 \left(\int \tilde \rho(t,s)\psi(t) \fhi'(s)  + \tilde \rho(t,s) a(s) \psi'(t) \fhi(s)\right) \, dt\,ds = 0.
\end{equation*}
By density of $\Psi$ this means that
\begin{equation*}
\d_t (a(s) \tilde \rho) + \d_s \tilde \rho = 0
\quad\text{ and } \quad \tilde \rho|_{t=0} = \tilde \rho_0.
\end{equation*}
By Lemma~\ref{1d-renorm}, the function $\beta(\rho)$ solves
$\d_t (a(s) \beta( \tilde \rho)) + \d_s \beta(\tilde \rho)$
with the initial condition $\beta(\tilde\rho) |_{t=0} = \beta(\tilde \rho_0)$.
Then by Proposition~\ref{div-wsp-rhs-mono} we conclude,
that for any $\psi \in \Psi$ the functions
\begin{equation*}
 R_\beta(x):=\int_0^T \beta(\rho(t,x)) \psi(t)\,dt, \qquad
 Q_\beta(x):=\psi(0) \beta(\rho_0(x)) + \int_0^T \beta(\rho(t,x)) \psi'(t)\, dt
\end{equation*}
solve $\dive (R_\beta \nabla^\perp f) = Q_\beta$.
By arbitrariness of $\psi$, this implies that
$\beta(\rho)$ solves \eqref{continuity-equation}
with the initial condition $\beta(\rho_0)$.
Thus, $\ve = \nabla^\perp f$ has the renormalization property.
\end{proof}

\begin{proof}[Proof of Theorem~\ref{uniqueness-criterea} {\rm(}sufficiency of {\rm(}i{\rm)} for {\rm(}ii{\rm)} and {\rm(}iii{\rm)}{\rm)}.]
Let $\beta\in C^1(\R)$. Without loss of generality we may assume that 
\begin{equation}
\label{beta00}
\beta(0)=0
\end{equation}
(since shift by a constant is not essential in the proof of the renormalization property). 
Now suppose that 
$\rho \in L^\infty((0,T)\times \R^2)$ solves \eqref{continuity-equation}--\eqref{continuity-equation-initial-condition}.
By Proposition~\ref{reduction-of-CE-to-mono} we have
\begin{equation*}
\d_t (\1_{\{\nabla f =0\}} \rho) = 0
\end{equation*}
and for all $i\in I$
the function $\rho_i = \1_{\{\nabla f_i \ne 0\}} \rho$ solves
\begin{equation*}
\d_t (\rho_i) + \dive (\rho_i \nabla^\perp f_i) = 0.
\end{equation*}
Then
\begin{equation*}
\d_t (\1_{\{\nabla f = 0\}}\beta(\rho)) = 0
\end{equation*}
and, since $f_i$ are monotone and have the weak Sard property,
by Proposition~\ref{renormalization-mono} we have
\begin{equation*}
\d_t (\beta(\rho_i)) + \dive (\beta(\rho_i) \nabla^\perp f_i) = 0
\end{equation*}
and $\beta(\rho_i)|_{t=0} = \beta(\rho_0)\1_{\{\nabla f_i \ne 0\}}$.
Note, that by construction and by~\eqref{beta00} we have
\begin{equation*}
\beta(\rho_i) = \beta(\rho) \1_{\{\nabla f_i \ne 0\}}.
\end{equation*}
Then again, using Proposition~\ref{reduction-of-CE-to-mono},
we conclude that $\beta(\rho)$
solves the continuity equation with the initial condition $\beta(\rho_0)$.
Hence, $\ve$ has the renormalization property. Consequently, the~uniqueness of bounded weak solutions
to \eqref{continuity-equation}--\eqref{continuity-equation-initial-condition} holds as well.

\end{proof}

In order to apply the sufficiency part of Theorem~\ref{uniqueness-criterea}, recall one of the strongest versions of the Morse--Sard theorem in the plane proved
by J.~Bourgain, M.~Korobkov and J.~Kristensen \cite{BKK13}.
\begin{theorem}[\cite{BKK13}]\label{MS-plane}
Let $f\in \BV_2(\R^2)$, i.e., the second distributional derivative of $f$ is a Radon measure. Then $f$ has the strict Sard property, i.e., the range of the $f$-critical set\footnotemark{}
has $1$-measure zero.
\end{theorem}
Note, that for the locally Lipschitz functions $f\in\BV_2(\R^2)$ the strong Sard property was proved in \cite{PZ06}. An~independent proof of the weak Sard property for this case was given also in~\cite[Theorem 8.4]{BG16}).
See the recent survey \cite{FKK24} for different extensions of this result to multidimensional cases, etc.

\footnotetext{The critical set $Z_f=Z_1\cup Z_2\cup Z_3$ for $BV_2$ functions consists of three parts:
$Z_1=\{x\in\R^2:\nabla f(x)$ is well-defined and $\nabla f(x)=0\}$; \ $Z_2=\{x\in\R^2:$ both <<half-plane>> values of $\nabla f$ are well-defined and at least one of them is zero$\}$; \
$Z_3$ consists of <<bad points>>, where even half-plane values of $\nabla f$ are not well-defined, see \cite{BKK13} for details.}

In the two-dimensional setting
Theorems~\ref{uniqueness-criterea} and \ref{MS-plane} provide an independent proof
of the well-known result of L.~Ambrosio \cite{Ambrosio2004}:

\begin{cor}
Let $\ve\in \BV(\R^2)$ be a compactly supported divergence-free vector field.
Then $\ve$ has the renormalization property RP$_\infty$ and uniqueness holds for bounded weak solutions to \eqref{continuity-equation}--\eqref{continuity-equation-initial-condition}.
\end{cor}

\begin{rem}
The results on the sufficiency of the Sard's property allow for further amplifications and refinements, e.g., in Appendix~\ref{uni-uni} we 
show that if $\ve$ is bounded and the stream function of $\ve$ has the weak Sard property, then the~uniqueness of weak solutions holds not only in the class of bounded solutions, but also in a slightly wider class $L^\infty(0,T; L^1(\R^2))$.
\end{rem}

\subsection{Proof of Theorem~\ref{uniqueness-criterea}: necessity of the weak Sard property.}
\label{un-sect-3}

Before turning to the necessity part of Theorem~\ref{uniqueness-criterea},
let us recall that a family of Borel measures $\{\mu_y\}_{y\in \R}$ on $\R^d$ is Borel
if for any Borel set $A\subset \R^d$ the function $y\mapsto \mu_y(A)$ is Borel.

\begin{theorem}[{{Disintegration theorem; see, e.g., \cite[Example~10.4.11]{bogachev}}}]
Let $\nu$ be a finite Borel measure on $\R^d$
and let $f\colon \R^d \to \R$ be a Borel function.
Let $\mu$ be a Borel measure on $\R$ such that $f_\# \nu \ll \mu$.
Then there exists a Borel family $\{\nu_y\}_{y\in \R}$
such that
$\nu_y$ is concentrated on $\{f=y\}$ for $\mu$-a.e. $y$ and
\begin{equation*}
\nu = \int_{\R} \nu_y \, d\mu(y)
\end{equation*}
in the sense that $\nu(A) = \int_{\R} \nu_y(A)\, d\mu(y)$
for all Borel sets $A\subset \R^d$.
If $\{\tilde \nu_y\}_{y\in \R}$ is another such family, then $\tilde \nu_y = \nu_y$ for $\mu$-a.e. $y$.
\end{theorem}
The family $\{\nu_y\}$ given by the theorem above is called \emph{the disintegration of $\nu$ with respect to $f$}.

\begin{proof}[Proof of Theorem~\ref{uniqueness-criterea} {\rm(}necessity of {\rm(}i{\rm)} for {\rm(}ii{\rm)} and {\rm(}iii{\rm)}{\rm)}.]

\ {\sc Step 1.} We start with a reduction to the case of monotone stream function.
Let $f = \sum_{i\in I} f_i$ be a decomposition of $f$ into monotone $W^{1,2}$ functions given by Theorem~\ref{thm:BT}.
Suppose that $f$ does not have the weak Sard property.
Then by Proposition~\ref{WSP-for-monotone-components} there exists $i\in I$
such that $g:= f_i$ does not have the weak Sard property. Put $h := f - g$.
By Proposition~\ref{mutual weak Sard property}, there exists a negligible set $N_1 \subset \R$
such that $\{\nabla h \ne 0\} \subset g^{-1}(N_1)$ modulo $\L^2$.
Let $N_2 \subset \R$ be a negligible set given by Corollary~\ref{E-T-is-measurable} for the function $g$, and 
let $V \subset \reg g \setminus (N_1 \cup N_2)$ be a $\sigma$-compact set with negligible complement
(such set $V$ exists by the inner regularity of the Lebesgue measure).
By Corollary~\ref{E-T-is-measurable}, the set
\begin{equation}
\label{def-E}
E:= \bigcup_{y\in V} \{g=y\}^*
\end{equation}
is measurable. By the choice of $V$ and Proposition~\ref{mutual weak Sard property} we also have
\begin{equation}
\{\nabla h\ne 0\} \cap E = \emptyset \mod \L^2.
\label{v-hat-zero-on-E}
\end{equation}

Suppose that $\rho \in L^\infty((0,T)\times \R^2)$ vanishes outside of $E$, i.e., for a.e. $t\in (0,T)$ we have
\begin{equation*}
\{\rho(t, \cdot) \ne 0\} \subset E.
\end{equation*}
Let $\mathbf b$ be a Borel vector field such that $\mathbf b = \nabla^\perp g$ a.e.
By \eqref{v-hat-zero-on-E} we have $\rho \nabla^\perp f = \rho \mathbf b$ a.e.
Then $\rho$ solves \eqref{continuity-equation} if and only if
\begin{equation}
\d_t \rho + \dive (\rho \mathbf b) = 0. \label{CE-component}
\end{equation}
Thus, it is sufficient to construct a nontrivial solution $\rho$ of \eqref{CE-component} which vanishes outside of $E$.
This solution can be constructed in the same way as in the proof of \cite[Theorem 3.10]{ABC_2014},
which was stated for the bounded vector fields.
In the remaining part of the present proof we reproduce the essential steps from \cite[Theorem 3.10, Step 2]{ABC_2014}
and show that in our setting we do not need to assume that $\mathbf b$ is bounded.

\medskip
{\sc Step 2.} Now we foliate \eqref{CE-component} into a family of equations on the level sets.
Let $N_3\subset \R$ be a negligible Borel set on which the singular part of $g_\# (\L^2 \rest E)$ is concentrated.
Replacing $V$ with $V\setminus N_3$, if necessary, we may assume that $g_\# (\L^2 \rest E) \ll \L^1$.
By the disintegration theorem, there exists a measurable family of Borel measures $\{\nu_y\}_{y\in V}$
such that for a.e. $y$ the measure $\nu_y$ is concentrated on $\{g=y\}$ and
\begin{equation}\label{L2-at-b=0}
\L^2 \rest (E \cap \{\mathbf b=0\}) = \int_V \nu_y \, dy.
\end{equation}

On the other hand, by the Coarea formula
\begin{equation}
\L^2 \rest (E \cap \{\mathbf b\ne 0\}) = \int_V \frac{1}{|\mathbf b|} \H^1 \rest \{g=y\}\, dy
\label{disintegration-of-Lebesgue-measure-on-E}
\end{equation}
(recall that for a.e. $y\in V$ we have $\{g=y\}^* = \{g=y\}$).
Thus the disintegration of $\L^2 \rest E$ with respect to $g$ has the form
\begin{equation}\label{disint-1}
\L^2 \rest E = \int_V \mu_y \, dy, \qquad \text{where} \qquad
\mu_y = \nu_y + \frac{1}{|\mathbf b|} \H^1 \rest \{g=y\} \qquad \text{for a.e. $y\in V$}.
\end{equation}
Since $\nu_y$ is concentrated on $\{\mathbf b = 0\}$ and the inequality $\mathbf b \ne 0$ holds $\H^1$-a.e. on $\{g=y\}$,
we have
\begin{equation}
\nu_y \perp \H^1 \rest \{g=y\}\qquad\mbox{ for a.e. }y.
\label{nu-y-perp-H1}
\end{equation}

Let us write the distributional formulation of \eqref{CE-component} with the test function of the form
$\Phi(t,x) = \psi(g(x_1, x_2)) \fhi(t, x_1,x_2)$:
\begin{equation*}
\int_V \psi(y) \int_0^T \int_{\{g=y\}} (\d_t \fhi + \mathbf b \cdot \nabla \fhi) \rho \, d\mu_y \, dt \, dy = 0.
\end{equation*}
Hence (by arbitrariness of $\psi$) $\rho$ solves \eqref{CE-component} if and only if for a.e. $y$
\begin{equation}
\d_t (\rho \mu_y) + \dive (\rho \mathbf b \mu_y) = 0
\label{CE-on-the-level-set}
\end{equation}
(in the sense of distributions). So far, we have proved the following important analog of Proposition~\ref{div-wsp-rhs-mono}:

\underline{\bf Foliation claim}. {\sl Suppose that $\rho \in L^\infty((0,T)\times \R^2)$ vanishes outside of $E$ defined by~\eqref{def-E}. Then~$\rho$ is a~solution to \eqref{CE-component} iff $\rho$ is a~solution to~\eqref{CE-on-the-level-set} for a.e. $y\in V$. }

\medskip
{\sc Step 3.}
By Step 1 the function $g$ does not have the weak Sard property, hence the measure in both sides of \eqref{L2-at-b=0} {\it is nontrivial}. 
In other words, there exists a Borel set $H \subset V$ with positive measure such that $\nu_y\ne0$ for any $y \in H$. 
For any such $y$ let us construct a nontrivial solution $\rho=\rho_y$ to \eqref{CE-on-the-level-set} with zero initial condition.

Let $\gamma \colon [0, L) \to \R^2$ be an injective absolutely continuous parametrization of $\{g=y\}^*$
with $\gamma'\ne 0$ a.e., which satisfies \eqref{good-parametrization-f}, i.e., 
\begin{equation}
\label{good-parametrization}
\gamma'(s) = \nabla^\perp g(\gamma(s)) \qquad\text{for a.e. } s\in[0,L).
\end{equation}
For convenience, we denote $\L := \L^1\rest[0,L)$.
Then by the area formula
\begin{equation*}
\gamma_\# \left(\L\right) =
\gamma_\# \left(\frac{|\gamma'|}{|\mathbf b(\gamma)|}\L\right) =
\frac{1}{|\mathbf b|} \H^1 \rest \{g=y\}.
\end{equation*}
Let $\lambda_y$ be a Borel measure on $[0,L)$ such that $\gamma_\# \lambda_y = \nu_y$.
Then
\begin{equation}
\gamma_\# \left(\L + \lambda_y\right) = \mu_y.
\label{mu-y-preimage}
\end{equation}
By \eqref{nu-y-perp-H1} we have
\begin{equation}
\lambda_y \perp \L.
\label{lambda-y-perp-L1}
\end{equation}

For any function $\psi=\psi(t,x)$ let us denote $\widetilde{\psi}(t,s) := \psi(t, \gamma(s))$.
We consider $\gamma$ and $\widetilde{\psi}$ as $L$-periodic functions with respect to the variable $s$.
If~$\widetilde{\rho}$ solves
\begin{equation}
\d_t \bigl(\widetilde{\rho} \cdot (1 + \lambda_y)\bigr) + \d_s \widetilde{\rho} = 0
\label{CE-one-the-level-set-param}
\end{equation}
in the sense of distributions, then $\rho$ solves \eqref{CE-on-the-level-set}.
Indeed, writing the distributional formulation of \eqref{CE-one-the-level-set-param} with the test function of the
form $\Phi(t,s) = \fhi(t, \gamma(s))$ and using \eqref{good-parametrization} we have
\begin{equation*}
\int\int \widetilde{\d_t \fhi} \widetilde{\rho} \, d(\L+ \lambda_y) \, dt
+ \int\int \widetilde{\mathbf b} \cdot \widetilde{(\nabla_x \fhi)} \tilde \rho \, d\L \, dt.
\end{equation*}
Since $\nu_y$ is concentrated on $\{\mathbf b=0\}$, $\lambda_y$ is concentrated on $\{\widetilde{\mathbf b} = 0\}$.
Therefore, the equality above can be written as
\begin{equation*}
\int \int \left(\widetilde{(\d_t \fhi)} + \widetilde{\mathbf b} \cdot \widetilde{\nabla_x \fhi})\right) \tilde \rho
\, d (\L + \lambda_y) \, dt = 0.
\end{equation*}
Changing the variables in the inner integral by \eqref{mu-y-preimage}, we obtain the distributional formulation of \eqref{CE-on-the-level-set}.

It was proved in \cite[Lemma 4.5]{ABC_2014} that for any nontrivial measure $\lambda_y$,
which is singular with respect to $\L$, the equation \eqref{CE-one-the-level-set-param}
has a nontrivial bounded Borel solution $\theta_y$ with zero initial data.
Since $\gamma$ is a homeomorphism between a circle of length $L$ and the essential level set $\{g=y\}^*$,
the function
\begin{equation*}
\rho_y(t,x) :=
\begin{cases}
\theta_y(t, \gamma^{-1}(x)), & x\in \{g=y\}^*, \\
0 & \text{otherwise}
\end{cases}
\end{equation*}
is Borel and, furthermore, $\widetilde{\rho_y} = \theta_y$
solves \eqref{CE-one-the-level-set-param}.
Then, by the previous paragraph, $\rho_y$ solves \eqref{CE-on-the-level-set}.

We have proved that there exists a Borel set $H \subset V$ with positive measure
such that for any $y \in H$ there exists a~nontrivial bounded Borel solution $\rho_y$ of \eqref{CE-on-the-level-set} with zero initial data.

\medskip
{\sc Step 4.} Now our goal is to <<glue together>> the solutions $\rho_y$ constructed in the previous step in such a way that the resulting function $\rho$ is measurable.

By Lusin's theorem there exists a compact $K \subset H$ with positive measure such that the function $y \mapsto \mu_y(\R^2)$ is continuous on $K$, hence there exists a constant $C>0$ such that $\mu_y(\R^2) \le C$ for all $y\in K$. Furthermore, the compact $K$ can be chosen in such a way that
\begin{equation}
	|\mathbf b| \in L^1(\mu_y)
\end{equation}
for \emph{all} $y\in K$. Indeed, by the coarea formula
\begin{equation*}
	\int_\R \int_{\R^2} |\mathbf b| \, d\mu_y \, dy = \int_{\R^2} |\mathbf b| \, dx \le C \|\mathbf b\|_2 < \infty,
\end{equation*}
since $\mathbf b$ has compact support. Hence $\mathbf b\in L^1(\mu_y)$ for a.e. $y$,
and it remains to remove from $K$ an open set with small measure.

Let us fix some natural number $m$.  Denote by $\mathscr{M}_m$ the set of all Borel measures $\eta$ on the <<spacetime>> $[0,T) \times \R^2$ such that 
\begin{equation}
	|\eta|(\R^2) \le mCT.\label{mu-bound}
\end{equation}
We endow $\mathscr{M}_m$ with the weak-star topology.
This topology is metrizable, since $\mathscr{M}_m$ is a norm-bounded subset of the~dual to the space of bounded continuous functions on $[0,T)\times \R^2$. Moreover, $\mathscr{M}_m$ endowed by this metric is a~complete separable metric space, i.e., it is {\it a~Polish space}. 
Furthermore,  we consider $X_m:= K \times \mathscr{M}_m$ as a product of metric spaces, i.e., as a~complete metric space as well. 
 
For each $y \in K$ denote
\begin{equation}\label{def-zeta}
	\zeta_y := \L^1 \rest [0,T) \otimes \mu_y,
\end{equation}
where $\otimes$ denotes the product of measures. It is easy to see that $\zeta_y \in \mathscr{M}_m$.
Following \cite{ABC_2014}, let us prove that the set
\begin{equation*}
	X^*_m:=\{(y,\eta) \in X_m : |\eta| \le m \zeta_y \}
\end{equation*}
is a Borel subset of $X_m$.
Indeed, $(y,\eta) \in X^*_m$ if and only if for each $\fhi \in C_c([0,T)\times \R^2)$ we have $\Lambda_\fhi(y, \eta) \le 0$, where
\begin{equation*}
	\Lambda_\fhi(y, \eta) := \int \fhi \, d\eta - m \int |\fhi| \, d\zeta_y.
\end{equation*}
The first term is a continuous function of $\eta$ (by the definition of the topology on $\mathscr{M}_m$), and the second term is a Borel function of $y$, since the family $[0,1] \mapsto \mu_y$ is Borel.
Hence for each $\fhi \in C_c([0,T)\times \R^2)$ the functional $\Lambda_\fhi$ is Borel on $X_m$.
Let $D$ denote some countable dense set in $C_c([0,T)\times \R^2)$.
Then clearly $X^*_m = \{(y,\eta) \in X_m : \Lambda_\fhi(y,\eta) \le 0 \ \forall \fhi \in D\}$.
Hence the set $X^*_m$ is Borel.

Let $\mathscr{C}_m \subset X^*_m$ be the set of all pairs $(y, \eta)$ with $y\in K$ and $\eta \in \mathscr{M}_m$ such that $\eta = \rho \cdot \zeta_y$,
where $\rho$ is a Borel solution of \eqref{CE-on-the-level-set} with zero initial data
satisfying $|\rho| \le m$. 
We also denote
\begin{equation*}
	\mathscr{C}_m^* := \{(y, \eta)\in \mathscr{C}_m : \eta \ne 0\}.
\end{equation*}

If $\mathbf b$ was bounded, then we could use the following results from \cite{ABC_2014}:

\emph{Proposition A (\cite[Proposition 7.1]{ABC_2014}).}
The set $\mathscr{C}_m$ is a Borel subset of $X^*_m$.

\emph{Proposition B (\cite[Corollary 7.2]{ABC_2014}).}
Let $G_m$ denote the projection of $\mathscr{C}_m^*$ on $K$, i.e. the image of $\mathscr{C}_m^*$ under the map $\pi(y,\eta) = y$.
Then $G_m$ is Lebesgue measurable and
there exist
a Borel set $G_m' \subset G_m$ and
a Borel family $\{\eta_y : y\in G_m'\}\subset \mathscr{M}_m$
such that
$\L^1(G_m \setminus G_m') = 0$
and
$(y, \eta_y)\in \mathscr{C}_m^*$
for all $y\in G_m'$.

We claim that the above results holds in our setting, even though $\mathbf b$ may be unbounded.
Indeed, in the proof of Proposition A
relied on the fact that for any $\fhi\in C_c^\infty([0,T);\R^2)$
for any bounded Borel vector field $\mathbf w$ the functional
$\Lambda_{\mathbf w, \fhi}\colon X_m \to \R$ given by
\begin{equation*}
	\Lambda_{\mathbf w, \fhi}(y,\eta) = \int (\d_t \fhi + \mathbf w \cdot \nabla_x \fhi) \, d\eta
\end{equation*}
is Borel. 
In our setting $\mathbf w = \mathbf b$ may be unbounded, but $\mathbf b\in L^1(\mu_y)$ for all $y\in K$.
For any $n\in \N$ let
\begin{equation*}
	\mathbf b_n := \begin{cases}
		\mathbf b, & \text{if } |\mathbf b| \le n\\
		0, & \text{otherwise}.
	\end{cases}
\end{equation*}
Then $\mathbf b_n \to \mathbf b$ pointwise as $n\to \infty$.
By our definitions for every~$(y,\eta) \in X^*_m$ we have $|\eta| \le m \zeta_y$,
hence by the dominated convergence theorem
$\Lambda_{\mathbf b_n, \fhi}(y,\eta) \to \Lambda_{\mathbf b, \fhi}(y,\eta)$ as $n\to\infty$ .
Thus $\Lambda_{\mathbf b}$ is still Borel on $X^*_m$, being a pointwise limit of Borel functionals.
This shows that Proposition A holds in our setting.
Consequently, Proposition~B also holds in our setting, since it depends on $\mathbf b$
through Proposition A only (indeed, Proposition~B is a simple consequence of Proposition~A and the classical {\it von Neumann's selection theorem}, cf. \cite[Corollary~5.5.8]{Sriv98}).

We claim that the number $m$ can be chosen in such a way that $G_m$ has a~strictly positive measure. Indeed, 
\begin{equation*}
	\bigcup_{m=1}^\infty G_m = \bigcup_{m=1}^\infty \pi(\mathscr{C}_m^*) = \pi\left(\bigcup_{m=1}^\infty \mathscr{C}_m^*\right) = K,
\end{equation*}
since for any $y\in K$ there exists a nontrivial bounded Borel solution $\rho_y$ of \eqref{CE-on-the-level-set} with zero initial data, and for any such solution there exists $m\in \N$ such that $\rho_y \in \mathscr{C}_m^*$.

The rest of the proof of Theorem~\ref{uniqueness-criterea} directly follows \cite[Theorem 3.10]{ABC_2014}.
Namely, we fix $m\in \N$ in such a way that $G_m$ (and, consequently, $G_m'$) has strictly positive measure.
Let $\{(y, \eta_y)\}$ be given by Proposition B.
Let $\rho_y$ be such that $\eta_y = \rho_y \zeta_y$.
Since $|\eta_y| \le m \zeta_y$,
the measure $$\eta:= \int_{G_m'} \eta_y \, dy$$ is absolutely continuous with respect to $\L^1\rest[0,T) \otimes \L^2$.
Let $\rho=\rho(t,x)$ denote the density of $\eta$ with respect to $\L^1\rest[0,T) \otimes \L^2$
(without loss of generality we may assume that $\rho$ is Borel).
By the disintegration theorem for a.e. $y\in G_m'$ and a.e. $t\in[0,T)$ we have
\begin{equation*}
	\rho(t,x) = \rho_y(t,x)
\end{equation*}
for $\mu_y$-a.e. $x\in \R^2$. 
Thus for a.e.  $y \in G_m'$ the function $\rho$
solves \eqref{CE-on-the-level-set} with zero initial data. By construction (see also \eqref{disint-1}, \eqref{def-zeta}\,) the Borel function~$\rho$ vanishes outside of~$E$, moreover, for a.e. $y\in V\setminus G_m'$ we have  $\rho(t,x) = 0$ for $\mu_y$-a.e. $x\in \R^2$. Then by Foliation Claim of Step~2 $\rho$ 
is a~solution to \eqref{CE-component}, and, consequently, to~(\ref{continuity-equation}).

For any $y \in K$ we have $\int_0^T \int_{\R^2} |\rho_y| \, d\mu_y \, dt > 0$, since $\eta_y$ is nontrivial.
Consequently,
\begin{equation*}
	\int_0^T \int_{\R^2} |\rho(t,x)| \, dx \, dt = \int_{G_m'} \int_0^T \int_{\R^2} |\rho| \, d\mu_y \, dt \, dy > 0,
\end{equation*}
since $G_m'$ has strictly positive measure.
Hence $\rho$ is nontrivial.

The solution $\rho$ constructed above clearly is not renormalized, since
if it was, it must have been trivial.
\end{proof}

\medskip

{\bf Data availability statement.} {\sl Data sharing not applicable to this article as no datasets were generated or analyzed during the current study.}

\medskip

{\bf Conflict of interest statement}. {\sl The authors declare that they have no conflict of interest.}

\medskip

\section{Acknowledgements}\label{sec:ack}
The work of N.G. was supported by the RSF project 24-21-00315.
N.G. is grateful to Stefano Bianchini and Eugeny Panov for the discussion of the present work.
The authors are also grateful to the anonymous referee for the constructive criticism of the previous version of this text, which was very useful.

\appendix
\section{Uniqueness of integrable solutions}\label{uni-uni}

As in the previous sections, here we consider a compactly supported divergence-free vector field ${\ve\colon \R^2 \to \R^2}$.
In this section we show that if $\ve$ is bounded and the corresponding stream function~$f$ has the weak Sard property, then the~uniqueness of weak solutions holds not only in the class of bounded solutions, but also in a slightly wider class $L^\infty(0,T; L^1(\R^2))$. 
First of all, since $f$ is compactly supported, we can establish the following version of Proposition~\ref{div-wsp-rhs}:

\begin{prop}\label{div-wsp-rhs-plus}
Suppose that compactly supported  $f$ has the weak Sard property and $f\in W^{1,\infty}(\R^2)$.
Then functions $R \in L^1(\R^2)$  and $Q \in L^1(\R^2)$ satisfy
\begin{equation}
\label{divergence-equation-with-rhs}
\dive(R \nabla^\perp f) = Q
\end{equation}
iff for all $i\in I$ the identities \eqref{equation-for-nonzero-components}\,--\,\eqref{equation-for-zero-component} hold.
\end{prop}

Indeed, in order to use $\psi(h)$ as a test function in the distributional formulation of\ \eqref{steady-contin-rhs}, it is sufficient to prove that the product $R |\nabla f| |\nabla(\psi \circ h)|$ is integrable, which is immediate since $|\nabla f|$ and, consequently, $|\nabla h|$ are essentially bounded.

A similar version of Proposition~\ref{div-wsp-rhs-mono} holds.
For convenience, we formulate it using suitable parametrizations from Proposition~\ref{existence-of-good-parametrization} generated by $\nabla f$:

\begin{prop}\label{div-wsp-rhs-mono-plus}
Suppose that compactly supported $f\in W^{1,\infty}(\R^2)$ is monotone and has the weak Sard property.
Then functions $R \in L^{1}(\R^2)$, $q\ge1$, and $Q \in L^1(\R^2)$ satisfy \eqref{steady-contin-rhs}
iff for a.e. $y\in \reg f$ we have
\begin{equation}
\label{div-along-curve}
(R\circ \gamma)' = Q\circ \gamma,
\end{equation}
where $\gamma\colon [0,L_y]\to \R^2$ is an absolutely continuous parametrization of $\{f=y\}^*$ satisfying 
\eqref{good-parametrization-f} and the functions $R\circ \gamma$ and $Q\circ \gamma$ belong to $L^1[0,L_y]$.
\end{prop}

The proof of the last proposition repeats verbatim the proof of Proposition~\ref{div-wsp-rhs-mono} with some evident simplifications, since now we have a~suitable parametrization satisfying~\eqref{good-parametrization-f}. Furthermore, since for the above parametrizations $a \equiv 1$ (where $a$ is given by \eqref{a-def}), it is easy to see that Lemma~\ref{1d-renorm} holds in a wider class:

\begin{lemma}\label{1d-renorm-plus}
Suppose that $T>0$, $\ell > 0$, $I = (\R/\ell\Z)$ and
$\rho \in L^\infty(0,T; L^1(I))$ solves
\begin{equation}
\label{1d-transport-plus}
\d_t \rho + \d_s \rho = 0
\end{equation}
with the initial condition $\rho|_{t=0} = \rho_0 \in L^1(I)$.
Then for any~$\beta \in C^1(\R)$ with $\beta(0)=0$ and $\beta' \in L^\infty(\R)$ the function $\beta(\rho)$ solves \eqref{1d-transport-plus} with the initial condition $\beta(\rho_0)$.
\end{lemma}

Thus we have the following version of Theorem~\ref{uniqueness-criterea}:

\begin{theorem}\label{aaa4}
Let $\ve \in L^\infty(\R^2)$ be a compactly supported divergence-free vector field
and let $f\in W^{1,\infty}(\R^2)$ be a stream function of $\ve$.
If $f$ has the weak Sard property, then uniqueness holds for \eqref{continuity-equation}--\eqref{continuity-equation-initial-condition} in the class $\rho \in L^\infty(0,T; L^1(\R^2))$.
Furthermore, any solution in this class is renormalized.
\end{theorem}

\begin{rem}
In fact, uniqueness holds in the class $L^\infty(0,T; L^q(\R^2))$ for any $q\in [1,+\infty]$.
Indeed, since $\ve$ is compactly supported, any solution $\rho\in L^\infty(0,T; L^q(\R^2))$ of \eqref{continuity-equation} with zero initial condition vanishes outside of some fixed ball in $\R^2$, and therefore belongs to $L^1(0,T; L^1(\R^2))$. Hence it vanishes by above Theorem~\ref{aaa4}.
\end{rem}

\section{Proofs of some standard technical facts}\label{appendix-connectedness-lemma}

This section contains the proof of some simple technical facts that are widely known and often used, and thus belong to a kind of mathematical folklore. However, as often happens, it is sometimes difficult to provide a simple precise reference to the formulation of these facts in classical textbooks. Therefore, for the reader's convenience, we present the corresponding accurate proofs here. These proofs are collected in the appendix, as they in no way claim to have any novelty.

We start from the proof of the connectedness Lemma~\ref{connected-open-set-with-connected-exterior}, which is essentially the same as the one given in \cite{Govc13}.

\begin{proof}[Proof of Lemma~\ref{connected-open-set-with-connected-exterior}]
	If $C$ is disconnected, then it can be written as a disjoint union of two nonempty closed sets, $C_1$ and~$C_2$.
	
	Let us define the map $\fhi\colon \R^d \to (-\pi, \pi]$ by
	\begin{equation}
		\fhi(x) = (\1_A(x) + \1_{C_2}(x) - \1_B(x))\frac{\dist(x, C_1)}{\dist(x, C_1) + \dist(x, C_2)} \cdot \pi.
	\end{equation}
	It is clear that $\fhi$ is well-defined and
	\begin{equation}
		\fhi(C_1) = \{0\}, \quad
		\fhi(A) \subset (0, \pi), \quad
		\fhi(C_2) = \{\pi\}, \quad
		\fhi(B) \subset (-\pi, 0).
	\end{equation}
	Let $S^1$ denote the unit circle on the complex plane.
	Then the map $f\colon \R^d \to S^1$ given by
	\begin{equation}
		f(x) = e^{i\fhi(x)}
	\end{equation}
	is continuous.
	Let us fix $x\in C_1$ as the basepoint of $\R^d$, and let us fix $1$ as the basepoint of $S^1$. 
	Let $f_*\colon \pi_1(\R^d, x) \to \pi_1(S^1, 1)$ denote the homomorphism between fundamental groups of the corresponding spaces, induced by $f$.
	Since $\pi_1(\R^d, x)$ is trivial, it is sufficient to prove that $f_*(\pi_1(\R^d, x))$ is nontrivial.
	
	Let us fix $y\in C_2$.
	Since $C_1 \cap C_2 = \emptyset$, there exists $r>0$ such that $B_r(x) \cap C_2 = B_r(y) \cap C_1 = \emptyset$, and furthermore $\fhi(B_r(x))\subset (-\pi/4, \pi/4)$ and $\fhi(B_r(y))\subset [-\pi, -3\pi/4] \cup [3\pi/4, \pi)$ (because $f$ is continuous).
	Since $C$ is the common boundary of $A$ and $B$, there exist points $x_A \in B_r(x)\cap A$, $x_B \in B_r(x)\cap B$, $y_A \in B_r(y)\cap A$ and $y_B \in B_r(y) \cap B$.
	Since $A$ and $B$ are connected, there exist curves $\gamma_A\colon [0,1]\to A$ and $\gamma_B \colon [0,1]\to B$ such that $\gamma_A(0) = x_A$, $\gamma_A(1)=y_A$, $\gamma_B(0) = y_B$ and $\gamma_B(1) = x_A$.
	Line segment $\theta_{AB}$ connecting $y_A$ with $y_B$ intersects $C_2$ (but not $C_1$), and similarly line segment $\theta_{BA}$ connecting $x_B$ with $x_A$ intersects $C_1$ (but not~$C_2$).
	Let $\gamma$ denote the composition of $\gamma_A$, $\theta_{AB}$, $\gamma_B$ and $\theta_{BA}$. Then it is easy to see that $f \circ\gamma$ is a loop on $S^1$ which is not homotopic (relative to $1$) to a constant path.
	Hence $f_*(\pi_1(\R^d, x))$ is nontrivial.
\end{proof}

\begin{proof}[Proof of Proposition~\ref{constancy on reduced boundary}]
	By \cite[Lemma 3.2]{MSZ_2003} there exist disjoint measurable sets $E_i \subset \R^d$
	such that $\R^d \setminus \bigsqcup_i E_i$ is negligible and for any $i\in \N$ the function $f$ is Lipschitz on $E_i$. Let $N \subset \R^d$ be a negligible Borel set such that for any $i$
	all points of $\tilde E_i := E_i \setminus N$ are density points of $E_i.$
	
	By Coarea (for BV)
	\begin{equation*}
		\int |D \1_{\{f>t\}}|(N) \, dt = |Df|(N) = 0
	\end{equation*}
	hence for a.e. $t\in\R$ the measure $D \1_{\{f>t\}}$ is concentrated on $\bigsqcup_i \tilde E_i.$
	Let us fix such $t$.
	
	By Theorem~\ref{thm:de_giorgi_federer}, for any set $E \subset \R^d$ of finite perimeter the measure $D\1_E$ is concentrated on the set~$E^{1/2}$. 
	Let $x\in \{f>t\}^{1/2} \cap \tilde E_i.$
	Since $\{f>t\}$ has density $1/2$ at $x$ and $\tilde E_i$ has density $1$ at $x,$
	for any $\eps > 0$ the ball $B_\eps(x)$ contains points $y,z \in \tilde E_i$ such that $f(y) > t$
	and $f(z) \le t.$ Then by continuity of $f$ on $\tilde E_i$ we have $f(x) = t.$
	Hence for a.e. $t \in \R$ the measure $D \1_{\{f>t\}}$ is concentrated on $\bigsqcup_i \tilde E_i \cap \{f=t\}.$
	
	For any test functions $\fhi,\psi\in C_0^\infty(\R^d)$ now we can write
	\begin{equation*}
		\int \fhi(x) \psi(f(x)) d |Df|(x)
	\end{equation*}
	using the Coarea formula for BV and for Sobolev functions:
	\begin{equation*}
		\int \int \fhi(x) \psi(f(x)) d |D \1_{\{f>t\}}|(x) \, dt
		= \int \int \fhi(x) \psi(f(x)) d\H^1\rest \{f=t\}(x) \, dt
	\end{equation*}
	Since $D \1_{\{f>t\}}$ is concentrated on $\{f=t\},$ we can rewrite the equality above as
	\begin{equation*}
		\int \left(\int \fhi(x) d|D \1_{\{f>t\}}|(x)\right) \, \psi(t) dt
		= \int \left(\int \fhi(x) d\H^1\rest \{f=t\}(x)\right) \, \psi(t) \, dt
	\end{equation*}
	By arbitrariness of $\fhi$ and $\psi$ we conclude that
	\begin{equation*}
		|D \1_{\{f>t\}}| = \H^1\rest \{f=t\}
	\end{equation*}
	for a.e. $t$. By \eqref{Gauss-Green} we have $|D\1_{\{f>t\}}| = \H^1 \rest \d^M \{f>t\},$
	so \eqref{f is constant on the reduced boundary of superlevelset} follows.
\end{proof}

\begin{proof}[Proof of Proposition~\ref{gen-inner-normal}]
	By the Coarea formula
	\begin{equation*}
		0=\int_{\R^2} \1_{\{\nabla f =0\}} |\nabla f| \, dx
		=\int \H^1(\{\nabla f=0\}\cap \{f=t\}) \, dt,
	\end{equation*}
	hence $\H^1(\{\nabla f=0\}\cap \{f=t\}) = 0$ for a.e. $t$.
	
	Next, for any $x\in \R^d$ we have the following equality:
	\begin{equation}
		f(x) = \int_\R \chi_t(x) \, dt,
		\label{layer-cake}
	\end{equation}
	where
	\begin{equation}
		\chi_t(x) = \1_{\{f>t\}}(x) - \1_{(-\infty,0)}(t).
	\end{equation} 
	(If $f$ is non-negative, then simply $\chi_t(x) = \1_{[0,+\infty)}(t)\cdot\1_{\{f>t\}}(x)$.)
	
	For a.e. $t\in \R$ we have
	\begin{equation}
		\label{derivative-of-generalized-superlevel}
		D \chi_t = \bm\nu_{\{f>t\}} \H^1 \rest \{f=t\}.
	\end{equation}
	Indeed, for a.e. $t$ the set $\{f>t\}$ has finite perimeter, hence $D \chi_t = \bm\nu_{\{f>t\}} \H^1 \rest \d^* \{f>t\}$ by \eqref{Gauss-Green}. Furthermore, by Proposition~\ref{constancy on reduced boundary} for a.e. $t$ we have $\d^* \{f>t\} = \{f=t\}$ $\mod \H^1$.
	
	Let $\bm\nu\colon \R^d \to \R^d$ be a Borel vector field such that
	\begin{equation}
		\bm\nu(x)=
		\begin{cases}
			\frac{\nabla f(x)}{|\nabla f(x)|}
			&\text{for a.e. $x\in \{\nabla f \ne 0\}$;}\\
			0
			&\text{otherwise.}
		\end{cases}
	\end{equation}
	
	Let $\bm\fhi \in C^1_c(\R^d; \R^d)$. First, by the Coarea formula for Sobolev map, we have
	\begin{equation}
		-\int_{\R^2} f \dive \bm \fhi \, dx = \int_{\R^2} \bm \fhi \cdot \nabla f \, dx
		=\int_\R \int_{\{f=t\}} \bm\fhi \cdot \bm \nu \, d\H^1 \, dt
	\end{equation}
	Second, by \eqref{layer-cake}, Fubini theorem and \eqref{derivative-of-generalized-superlevel}
	we have
	\begin{equation}
		-\int_{\R^2} f \dive \bm \fhi \, dx = -\int_{\R} \int \chi_t \dive \bm \fhi \, dx \, dt
		=\int_\R \langle D \chi_t, \bm\fhi \rangle \, dt
		=\int_\R \int_{\{f=t\}} \bm\fhi \cdot \bm \nu_{\{f>t\}} \, d\H^1 \, dt.
	\end{equation}
	Hence
	\begin{equation}
		\label{grad-vs-generalized-inner-normal}
		\int_\R \int_{\{f=t\}} \bm\fhi \cdot \bm \nu \, d\H^1 \, dt
		=\int_\R \int_{\{f=t\}} \bm\fhi \cdot \bm \nu_{\{f>t\}} \, d\H^1 \, dt
	\end{equation}
	Passing to the (pointwise) limit, by dominated convergence we deduce that
	\eqref{grad-vs-generalized-inner-normal} holds for any bounded Borel vector field $\bm \fhi\colon \R^d \to \R^d$. In particular, after substituting $\bm \fhi = \bm \nu$ we obtain
	\begin{equation}
		\int_\R \int_{\{f=t\}} (|\bm \nu|^2 - \bm \nu \cdot \bm \nu_{\{f>t\}}) \, d\H^1 \, dt = 0. \label{coarea-for-normals}
	\end{equation}
	For a.e. $t$ and for $\H^1$-a.e. $x \in \{f=t\}$ we have 
	$
	|\bm \nu(x)| = |\bm \nu_{\{f>t\}}(x)| = 1.
	$
	Hence we can rewrite \eqref{coarea-for-normals} as 
	\begin{equation}
		\int_\R \int_{\{f=t\}} (|\bm \nu| \cdot |\bm \nu_{\{f>t\}}| - \bm \nu \cdot \bm \nu_{\{f>t\}}) \, d\H^1 \, dt = 0. \label{coarea-for-normals-2}
	\end{equation}
	
	By Cauchy--Bunyakovsky inequality the integrand in \eqref{coarea-for-normals-2} is non-negative. Since the integral \eqref{coarea-for-normals-2} vanishes, the integrand vanishes as well:
	for a.e.~$t$ we have $\bm \nu(x) =\bm \nu_{\{f>t\}}(x)$ for $\H^1$-a.e. $x\in \{f=t\}$.
\end{proof}

\begin{proof}[Proof of Lemma~\ref{Jordan-interior}]
	\emph{Necessity.}
	Let us prove that $\d^M I_\gamma = \Gamma \mod \H^1$.
	Since $I_\gamma$ and $E_\gamma$ are open, it is clear that $I_\gamma \subset I_\gamma^1$ and $E_\gamma \subset I_\gamma^0$, hence
	\begin{equation}
		\d^M I_\gamma \subset \Gamma.
	\end{equation}
	On the other hand, by \cite[Lemma~4]{ACMM_2001} we have $\H^1(\Gamma) = P(I_\gamma)$,
	while $P(I_\gamma) = \H^1(\d^M I_\gamma)$ by~\eqref{Gauss-Green}.
	Hence $\H^1(\Gamma \setminus \d^M I_\gamma) = \H^1(\Gamma) - \H^1(I_\gamma) = 0$ and therefore \eqref{essential-boundary-is-curve} holds.
	
	\emph{Sufficiency.}
	For any $x\in \R^2 \setminus \gamma([0,1])$ there exists $r>0$ such that $B_r(x) \subset I_\gamma$.
	By the relative isoperimetric inequality (see e.g. \cite[Eq. (3.43)]{AFP}) there exists a constant $C>0$ such that
	\begin{equation}
		\min(\L^2(B_r(x) \cap E), \L^2(B_r(x) \setminus E)) \le C \H^1(\d^M E \cap B_r(x))^2.
	\end{equation}
	By \eqref{essential-boundary-is-curve} the right-hand side is zero, since $B_r(x) \cap \gamma([0,1]) = \emptyset$. So either $B_r(x) \subset E \mod \L^2$, or $B_r(x) \subset E^c \mod \L^2$. In the first case we set $g(x):=1$, and in the second case we set $g(x):=0$.
	
	We claim that $g$ is constant on $I_\gamma$ (and on $E_\gamma$).
	Indeed, $\{g=1\}$ and $\{g=0\}$ are open subsets of $\R^2 \setminus \gamma([0,1])$.
	Let $x, y\in I_\gamma$ and let $\theta \colon[0,1] \to I_\gamma$
	be a continuous curve such that $\theta(0) = x$ and $\theta(1) = y$.
	Covering $\theta([0,1])$ by finitely many open balls, lying in $I_\gamma$,
	we note that on the adjacent balls $g$ takes the same value. Hence $g(x) = g(y)$.
	The same argument shows that $g$ is constant on $E_\gamma$.
	
	Suppose that $g=1$ on $E_\gamma$. Then $E_\gamma \subset E \mod \L^2$, which is impossible since $\L^2(E) < \infty$. Thus $g=0$ on $E_\gamma$ (recall that on $\R^2 \setminus \gamma([0,1])$ the function $g$ takes only the values 0 and 1).
	If $g=0$ on $I_\gamma$, then the set of all density points of $E$ is negligible, and hence $E$ is negligible, which is not possible by the assumptions. Thus $g=1$ on $I_\gamma$.
	Therefore $E = I_\gamma \mod \L^2$, and hence $\d^M E = \d^M I_\gamma = \Gamma \mod \H^1$, as we have already proved in the previous part (necessity).
\end{proof}

\begin{proof}[Proof of Proposition~\ref{generalized-inner-normal-for-simple-set}]	
	Let $\gamma$ be a Lipschitz parametrization of $\d^* E$ given by Theorem~\ref{boundary-of-simple-plane-sets}.
	(Recall that by Theorem~\ref{thm:de_giorgi_federer} the sets $\d^M E$ and $\d^* E$ are equal up to $\H^1$-negligible subsets.)
	
	On the one hand, by De Giorgi's Theorem \ref{thm:de_giorgi}, for $\mathscr H^1$-a.e. $x \in \d^* E$ we have 
	\begin{equation}\label{eq:tan_de_giorgi}
		\Tan(\d^* E, x) = \mathsf{span} (\nu^\perp_E(x)) 
	\end{equation}
	where $\nu_E(x)$ is the generalized inner normal to $E$.
	
	On the other hand, since $\d^* E = \gamma([0,1]) \mod \H^1$  by Proposition \ref{prop:tang_smooth} for $\H^1$-a.e. $x\in \d^* E$ we have
	\begin{equation}\label{eq:tan_non_de_giorgi}
		\Tan(\gamma([0,1]), x) = \mathsf{span}(\gamma'(\gamma^{-1}(x))). 
	\end{equation}
	Since the approximate tangent space is a one dimensional vector space and since $\nu_E(x)$ is unit vector for $\H^1$-almost every $x \in \d^*E$, equalities \eqref{eq:tan_de_giorgi} and \eqref{eq:tan_non_de_giorgi} force that for $\H^1$-a.e. $x \in \d^*E$
	\begin{equation}\label{eq:equality_x}
		\nu^\perp_E(x) = \sigma(x) \frac{\gamma'(\gamma^{-1}(x))} {|\gamma'(\gamma^{-1}(x)), |} \qquad \text{ for }\sigma(x) \in \{\pm 1\}. 
	\end{equation}
	Then for the vector measure $\bm \mu_\gamma$ induced by $\gamma$ we have
	\begin{equation}
		\dive(\sigma \bm \mu_\gamma) = \dive(\nu^\perp_E \H^1 \rest \gamma([0,1]))=0.
	\end{equation}
	Then by Proposition~\ref{steady continuity equation along curve}, we conclude that $\sigma$ is constant $\H^1$-a.e. on $\gamma$. Consequently,
	\begin{equation*}
		\exists \bar\sigma \in \{\pm 1\}: \qquad  \nu_E^\perp(\gamma(s)) = \bar \sigma \cdot \frac{\gamma'(s)}{|\gamma'(s)|} \qquad \text{ for } \L^1\text{-a.e. } s \in [0,1].
	\end{equation*}
	Reversing the parametrization of $\gamma$, if necessary, one can achieve that $\bar \sigma = 1$. Finally, taking the corresponding natural reparametrization of~$\gamma$, we can get  $|\gamma'(s)| = \H^1(\d^* E)$ for a.e. $s\in [0,1]$.
\end{proof}

\begin{proof}[Proof of Proposition~\ref{derivative of indicator of a set of finite perimeter}]
By isoperimetric inequality (see e.g. \cite[Theorem 3.46]{AFP}) without loss of generality we may assume that $\meas{E} < \infty.$

Let $\{E_i\}_{i\in I}$ be the family of all $M$-connected components of $E.$
By Decomposition theorem, Proposition~\ref{indec-prop} and \eqref{Gauss-Green}
\begin{equation*}
\1_E = \sum_{i} \1_{E_i}, \qquad |D \1_E| = \sum_{i} |D \1_{E_i}|.
\end{equation*}
Consequently, $D \1_E = \sum_{i} \1_{E_i}.$

Let us fix some $i \in I.$ The set $F:=E_i$ is indecomposable.
Let $\{Y_{j}\}_{j\in J}$ denote the family of its holes.
Then by Proposition~\ref{saturation-and-holes}
\begin{equation*}
\1_F = \1_{\sat(F)} - \sum_{j} \1_{Y_j}, \qquad |D \1_F| = |D \1_{\sat(F)}| + \sum_{j} |D \1_{Y_j}|.
\end{equation*}
Consequently, $D\1_F = D\1_{\sat(F)} - \sum_{j} D\1_{Y_j}.$
It remains to note that the sets $\sat(F)$ and $Y_j$ are simple by Proposition~\ref{holes-of-indecomposable-set}.
Hence Proposition~\ref{generalized-inner-normal-for-simple-set} can be applied.
\end{proof}

\begin{proof}[Proof of Proposition~\ref{existence-of-good-parametrization}]

We need two classical elegant facts from one-dimensional Real Analysis: 

\begin{lemma}[see, e.g., Theorem~6.7 of Chapter~VII in~\cite{Saks}]\label{raf-1}
{\sl A~continuous monotone function $h:[0,1]\to\R$ is absolutely continuous iff it has the~Luzin $N$-property, i.e., $\L^1(h(E))=0$ whenever $\L^1(E)=0$ for $E\subset[0,1]$.} 
\end{lemma}

\begin{lemma}[see, e.g., Theorem~6.5 of Chapter~VII in~\cite{Saks}]\label{raf-2}
{\sl Let $G\subset\R$ and a function $h:G\to\R$ is differentiable (in classical sense) on $G$. Then $h$ has the~Luzin $N$-property on~$G$, i.e., $\L^1(h(E))=0$ whenever $\L^1(E)=0$ with $E\subset G$.} 
\end{lemma}

Now let $f\in W^{1,1}(\R^2)$ be a~monotone function. Without loss of generality we may assume that $f\ge 0$. Fix arbitrary small $\delta>0$. 
Consider the set $B=\{f>\delta\}\cap\{|\nabla f|>0\}$. Of course, $B$ is measurable and 
$$\L^2(B)<\infty.$$ 
By Coarea formula  \eqref{coarea-fhi} (applied to the measurable nonnegative function~$\varphi(x)=\frac1{|\nabla f(x)|}$) we have
$$\L^2(B)=\int\limits_\delta^\infty\int\limits_{\{f=y\}\cap B}\frac1{|\nabla f(x)|}\,d \H^{1}_x \, dy=\int\limits_\delta^\infty\int\limits_{\{f=y\}}\frac1{|\nabla f(x)|}\,d \H^{1}_x \, dy=\int\limits_\delta^\infty\int\limits_{\{f=y\}^*}\frac1{|\nabla f(x)|}\,d \H^{1}_x \, dy$$
(the first equality follows from~\eqref{gradient does not vanish}, and the second equality follows from Definition~\ref{def-reg} of regular values and Proposition~\ref{reg-vals}). By arbitrariness of~$\delta>0$ this implies that 
\begin{equation}
\label{good-p01}
\frac1{|\nabla f|}\in \L^1\bigl(\H^1\rest\{f=y\}^*\bigr)\qquad\text{for a.e. } y\in \reg f.
\end{equation}
Take $y\in \reg f$, and let $\tilde\gamma:[0,1]\to\R^2$ be a~natural parametrization (up to a scaling factor) of the closed curve~$\{f=y\}^*$, that means, $|\tilde\gamma'|\equiv c_0$ a.e. on~$[0,1]$ with some constant~$c_0>0$. 
By definition of regular values and by \eqref{generalized inner normal for Sobolev function}, \eqref{generalized normal and tangent vector} we can assume without loss of generality, that
\begin{equation}
\label{good-p1}
|\nabla f(\tilde\gamma(t))|>0 \qquad\text{for a.e. } t\in[0,1],
\end{equation}
\begin{equation}
\label{good-p2}
\tilde\gamma'(t)=c_0 \frac{\nabla^\perp f(\tilde\gamma(t))}{|\nabla f(\tilde\gamma(t))|} \qquad\text{for a.e. } t\in[0,1].
\end{equation}

By \eqref{good-p01}, we can assume in addition, that  
\begin{equation}
\label{good-p3}
\frac1{|\nabla f(\tilde\gamma(\cdot))|}\in L^1([0,1]).
\end{equation}
Then the function 
$$F(t):=c_0\int\limits_0^t\frac1{|\nabla f(\tilde\gamma(\xi))|}\,d\xi$$
is strictly increasing and absolutely continuous on~$[0,1]$. Denote
$L_y=F(1)$. Then the function $h:[0,L_y]\to [0,1]$, inverse to~$F$, is well defined, continuous, and strictly increasing on the interval~$[0,L_y]$. 

We claim, that $h$ is absolutely continuous as well. Indeed, consider the ``bad'' set $$Z=\bigl\{t\in [0,1]:F'(t)\mbox{ is not well defined or }F'(t)\ne c_0\frac1{|\nabla f(\tilde\gamma(t))|}\bigr\}.$$
By our previous assumptions, $\L^1(Z)=0$, therefore,   
\begin{equation}
\label{nul1}
\L^1(F(Z))=0
\end{equation}
as well by absolute continuity of~$F$. Denote $E=F(Z)$. By construction, 
the function~$h$ is differentiable at any $s\in (0,L_y)\setminus E$ with 
\begin{equation}
\label{nul2}
h'(s)=\frac1{F'(h(s))}=\frac1{c_0}|\nabla f(\tilde\gamma(h(s)))|\ne0,
\end{equation}
where the~right-hand side is well defined. 
Then by Lemma~\ref{raf-2} the restriction~$h|_{[0,L_y]\setminus E}$ has the~Luzin $N$-property. 
But $\L^1(h(E))=\L^1(Z)=0$, hence, $h$ has the~$N$-property on the~whole interval~$[0,L_y]$. Then Lemma~\ref{raf-1} implies the required claim on absolute continuity of~$h$. 

Finally, for $s\in [0,L_y]$ put
$$\gamma(s):=\tilde\gamma(h(s)).$$
By construction (see, in particular, \eqref{good-p2}, \eqref{nul2}\,) 
$$\gamma'(s)=\tilde\gamma'(h(s))\cdot h'(s)=\nabla^\perp f(\gamma(s))$$
for almost all~$s\in[0,L_y]$. The~proposition is proved. 
\end{proof}

\bibliographystyle{alpha}

\end{document}